\documentclass[a4paper,twoside]{article}  

\usepackage{amssymb,amsmath,amsthm,dsfont,amsfonts,color,latexsym}

\usepackage{hyperref}

\usepackage{mathrsfs} 

\definecolor{blue}{rgb}{0.00,0.00,1.00}
\definecolor{red}{rgb}{1.00,0.00,0.00}

\allowdisplaybreaks[4]
\renewcommand{\baselinestretch}{1.2}
\def\bq{\begin{equation}}
\def\eq{\end{equation}}
\def\ba{\begin{array}{ccc}}
\def\bal{\begin{array}{lll}}
\def\ea{\end{array}}
\def\bsp{\begin{split}}
\def\esp{\end{split}}

\def\({\left(}\def\){\right)}
\def\[{\left[}\def\]{\right]}

    \def \O   {\mathbb{O}}

    \def \R   {\mathbb{R}}

    \def\i    {\mathrm{i}}

    \def\S    {\mathbb{S}}
    \def\eps  {\epsilon}
    \def\intr {\int_{\R^3}}
    
    \def\intt {\int^t_0}

    \def \N    {\mathbb{N}}

    \def \Dt   {\frac{\rm d}{{\rm d}t}}
    
    \def \dt    {\partial_t}

    \def \dx    {\partial_x}

    \def \dxa   {\partial^{\alpha}_x}

    \def \div  {{\rm div}}

    \def\Tdx   {\nabla_x}
    \def\Tdv   {\nabla_v}

       \def\bq{\begin{equation}}
       \def\eq{\end{equation}}
       \def\be{\begin{equation}}
       \def\ee{\end{equation}}
       \def\bma#1\ema{{\allowdisplaybreaks\begin{align}#1\end{align}}}
       \def\bmas#1\emas{{\allowdisplaybreaks\begin{align*}#1\end{align*}}}
       \def\bln#1\eln{{\allowdisplaybreaks\begin{aligned}#1\end{aligned}}}
       \def\nnm{\notag}
       \def\bgr#1\egr{\allowdisplaybreaks\begin{gather}#1\end{gather}}
       \def\bgrs#1\egrs{\allowdisplaybreaks\begin{gather*}#1\end{gather*}}

       \theoremstyle{plain}
       \newtheorem{lem}{\bf Lemma}[section]
       \newtheorem{thm}[lem]{\textbf{Theorem}}

       \newtheorem{remark}[lem]{\bf Remark}

\begin{document}


\title{Green's Function and Pointwise Space-time Behaviors of the three-Dimensional modified Vlasov-Poisson-Boltzmann System}

\author{Yanchao Li$^{1,2}$,\, Luobin Qiu$^1$,\, Mingying Zhong$^1$ \\[2mm]
 \emph
    {\small\it $^1$School of  Mathematics and Information Sciences,
    Guangxi University, P.R.China.}\\
    {\small\it  \& Center for Applied Mathematical of Guangxi (Guangxi University), Guangxi University, P.R.China.}\\
    {\small\it E-mail:\ zhongmingying@gxu.edu.cn}\\
    {\small\it $^2$Post-doctoral Research Station of the First-level Discipline in Mathematics, Guangxi University,
     P.R.China.}\\
    {\small\it E-mail:\ yanchaoli@gxu.edu.cn}}
\date{ }

\pagestyle{myheadings}
\markboth{modified Vlasov-Poisson-Boltzmann System}%
{Y.-C. Li, L.-B. Qiu and M.-Y. Zhong}

 \maketitle

 \thispagestyle{empty}

\begin{abstract}\noindent{
The pointwise space-time behavior of the Green's function of the three-dimensional modified Vlasov-Poisson-Boltzmann system is studied in this paper. It is shown that the Green's function has a decomposition of the macroscopic diffusive waves and Huygens waves with the speed $\sqrt{\frac{8}{3}}$ at low-frequency, the singular kinetic wave and the remainder term decaying exponentially in space and time. In addition, we establish the pointwise space-time estimate of the global solution to the nonlinear modified Vlasov-Poisson-Boltzmann system based on the Green's function.
}

\medskip



\end{abstract}

\tableofcontents


\section{Introduction}
\label{sect1}
\setcounter{equation}{0}

In this paper, we study the following modified Vlasov-Poisson-Boltzmann (mVPB) system:
\bq\label{3DmVPB1}
\left\{\bln
&\dt F+v\cdot\nabla_x F+\nabla_x \Phi\cdot\nabla_v F =Q(F,F),\\
&\Delta_x\Phi=\intr Fdv-e^{-\Phi},
\eln\right.
\eq
where $F=F(t,x,v)$ is the distribution function of ions with $(t,x,v)\in\R_+\times\R^3_x\times\R^3_v$, and $\Phi(t,x)$ denotes the
electric potential. The collision between particles is given by the standard Boltzmann collision operator $Q(F,G)$ as below
$$
Q(F,G)=\int_{\R^3}\int_{\mathbb{S}^2}B(|v-u|,\omega)(F(v')G(u')-F(v)G(u))dud\omega,
$$
where
$$v'=v-[(v-u)\cdot\omega]\omega,\quad u'=u+[(v-u)\cdot\omega]\omega,\quad \omega\in\S^2.$$

There has been much important progress made recently on the well-posedness and asymptotical behaviors of solution to the Vlasov-Poisson-Boltzmann (VPB) system. In particular, the global existence of a renormalized weak solution for general large initial data was shown in \cite{MISCHLER-1}. The global existence of a unique strong solution with the initial data near the normalized global Maxwellian was obtained in spatial period domain \cite{GUO-3} and in spatial three-dimensional whole space \cite{DUAN-1,YANG-2} for hard sphere, and then in \cite{DUAN-3,DUAN-4} for hard potential or soft potential. The existence of a classical solution near a vacuum was investigated in \cite{GUO-2}.

The time-asymptotic behaviors of global strong solutions to the VPB system was studied in \cite{DUAN-2,LI-1,LI-2,YANG-3}. The spectrum analysis and the optimal time-decay rate of the global solutions to the VPB systems for both one-species and two-species were studied in \cite{LI-1,LI-2}. Indeed, the time-decay rate $(1+t)^{-\frac{1}{4}}$ in $L^2$ norm for one-species VPB system in $\R^3$ was obtained in \cite{DUAN-2,LI-1}. The optimal decay rate $(1+t)^{-\frac{3}{4}}$ in $L^2$ norm for two-species and modified VPB systems in $\R^3$ was obtained in \cite{LI-2}. Moreover, the wave phenomena is observed for one-dimensional VPB system  \cite{DUAN-5,LI-3,LI-4}, such as the shock profile, rarefaction wave and viscous contact wave. The pointwise behavior of the Green's function of the Boltzmann equation for hard sphere was verified in \cite{Lin-1,LIU-2,LIU-3,LIU-5}. The pointwise space-time behaviors of the Green's function and the global solution to the one-species VPB system and one-dimensional mVPB system were studied in \cite{LI-5,LI-7}.
The  regularizing effect to the VPB system was studied in \cite{DENG-1,DENG-2}.

The fluid dynamical limit of the solution to the VPB system near Maxwellian was also studied in \cite{Guo-1,Gong-1,Jiang-1,Li-9,Tong-1,Wang-1,Wu-1}. In \cite{Guo-1}, the authors proved the convergence of the solutions to the unipolar VPB system towards a solution to the compressible Euler-Poisson system in the whole space. In \cite{Wang-1}, the author established a global convergence result of the solution to the bipolar VPB system towards a solution to the incompressible Vlasov-Navier-Stokes-Fourier system. In \cite{Li-9}, the authors established the spectral analysis of the unipolar VPB system,  and show the convergence rate of the global strong solution to  the solution of the incompressible Navier-Stokes-Poisson-Fourier system with precise estimation of the initial layer. Moreover, the authors  justified the incompressible Navier-Stokes-Fourier-Poisson limit of  the VPB system in \cite{Gong-1,Jiang-1,Tong-1}, and the incompressible Euler-Poisson limit in  \cite{Wu-1}.

In this paper, we study the pointwise space-time behaviors of the Green's function and the global solution to the three-dimensional mVPB system \eqref{3DmVPB1} based on the spectral analysis \cite{LI-2}. To begin with, let us consider the Cauchy problem for the three-dimensional mVPB system \eqref{3DmVPB1} with the following initial data:
\bq
F(0,x,v)=F_0(x,v),\quad (x,v)\in\mathbb{R}^3_x\times\mathbb{R}^3_v.\label{3DmVPB3}
\eq

The mVPB system \eqref{3DmVPB1} has an equilibrium state $(F_*,\Phi_*)=(M(v),0)$ with the normalized Maxwellian $M$ defined by
$$
M=M(v)=\frac{1}{(2\pi)^{\frac{3}{2}}}e^{-\frac{|v|^2}{2}},\quad v\in\R^3.
$$

Define the perturbation  $f(t,x,v)$ of $F(t,x,v)$ near $M$ by $F=M+\sqrt M f$, then the mVPB system \eqref{3DmVPB1} and \eqref{3DmVPB3} is reformulated into
\bma
 &\dt f+v\cdot\nabla_xf-v\sqrt{M}\cdot\nabla_x\Phi-Lf
=\frac12 (v\cdot\nabla_x\Phi)f-\nabla_x\Phi\cdot\nabla_vf+\Gamma(f,f),\label{3DmVPB8}\\
&(I-\Delta_x)\Phi=-\intr f\sqrt{M}dv+(e^{-\Phi}+\Phi-1),\label{3DmVPB9}\\
&f(0,x,v)=f_0(x,v)=(F_0-M){M}^{-1/2},\label{3DmVPB10}
 \ema
 where the linearized collision operator $Lf$ and the nonlinear term $\Gamma(f,f)$ are defined by
 \bma
&Lf=\frac1{\sqrt M}[Q(M,\sqrt{M}f)+Q(\sqrt{M}f,M)],  \label{Lf}\\
&\Gamma(f,f)=\frac1{\sqrt M}Q(\sqrt{M}f,\sqrt{M}f).   \label{gf}
 \ema
We have (cf \cite{CERCIGNANI-1})
\be \left\{\bln\label{LKF}
&(Lf)(v)=(Kf)(v)-\nu(v) f(v),   \quad (Kf)(v)=\intr k(v,u)f(u)du, \\
&k(v,u)=\frac2{\sqrt{2\pi}|v-u|}e^{-\frac{(|v|^2-|u|^2)^2}{8|v-u|^2}-\frac{|v-u|^2}8}-\frac{|v-u|}2e^{-\frac{|v|^2+|u|^2}4},\\
&\nu(v)=\frac1{\sqrt{2\pi}}\bigg(e^{-\frac{|v|^2}2}+2\(|v|+\frac1{|v|}\)\int^{|v|}_0e^{-\frac{|u|^2}2}du\bigg),
\eln\right.
 \ee
where  $\nu(v)$,  the collision frequency, is a real  function, and
$K$ is a self-adjoint compact operator on $L^2(\R^3_v)$ with  real symmetric integral kernels $k(v,u)$. For hard sphere model, $\nu(v)$ satisfies
 \be
\nu_0(1+|v|)\leq\nu(v)\leq \nu_1(1+|v|),  \label{nuv}
 \ee
 with $\nu_1\geq \nu_0>0$ two constants.

The null space of the operator $L$, denoted by $N_0$, is a subspace
spanned by the orthogonal basis $\{ \chi_j,\ j=0,1,2,3,4\}$  with
 \bq
  \chi_0=\sqrt{M},\quad  \chi_j=v_j\sqrt{M},\quad j=1,2,3,\quad \chi_4=\frac{(|v|^2-3)\sqrt{M}}{\sqrt{6}}.\label{basis}
 \eq

Let   $L^2(\R^3)$ be a Hilbert space of complex-value functions $f(v)$
on $\R^3$ with the inner product and the norm
$$
(f,g)=\intr f(v)\overline{g(v)}dv,\quad \|f\|=\(\intr |f(v)|^2dv\)^{1/2}.
$$
Introduce the macro-micro decomposition as follows
\be \label{P10}
\left\{\bal
f=P_0f+P_1f,\\
P_0f=P^1_0f+P^2_0f+P^3_0f,\quad  P_1f=f-P_0f,\\
P^1_0f=(f, \chi_0) \chi_0,\quad P^3_0f=(f,\chi_4)\chi_4,\\
P^2_0f=\sum^3_{k=1}(f,\chi_k)\chi_k.
\ea\right.
\ee

From the Boltzmann's H-theorem, $L$ is non-positive and
moreover, $L$ is locally coercive in the sense that there is a constant $\mu>0$ such that
\be
 (Lf,f) \leq -\mu \| P_1f\|^2, \quad   f\in D(L),\label{L_3}
 \ee
where $D(L)$ is the domains of $L$ given by
$$ D(L)=\left\{f\in L^2(\R^3)\,|\,\nu(v)f\in L^2(\R^3)\right\}.$$
Without the loss of generality, we assume in this paper that $\nu(0)\geq \nu_0\geq \mu>0$.

Since we only consider the pointwise behavior with respect to the space-time variable $(t,x)$, it's convenient to regard the Green's function $G(t,x)$ as an  operator on $L^2(\R^3_v)$:
\be   \label{LmVPB}
\left\{\bal
\dt G =B G , \\
G(0,x)=\delta(x)I_v,
\ea\right.
\ee
 where $I_v$ is the identity map on $L^2(\R^3_v)$ and the operator $B$ is defined by
\be
 B f=Lf-v\cdot\nabla_xf -v\cdot\nabla_x(I-\Delta_x)^{-1} P^1_0f. \label{B0(x)}
\ee
Then the solution for the initial value problem of the linearized mVPB  equation
\bq
\left\{\bal
\dt f=B f, \\
f(0,x,v)=f_0(x,v)   \label{3DmVPB}
\ea\right.
\eq
can be represented by
\be f(t,x)=G(t)\ast f_0=\int_{\R^3} G(x-y,t)f_0(y)dy, \ee
where $f_0(y)=f_0(y,v).$

For any $(t,x)$ and $f\in L^2(\R^3_v)$, we define the $L^2$ norm of $G(t,x)$ by
\be \|G(t,x)\|=\sup_{\|f\|=1}\|G(t,x)f\|, \label{norm}\ee
and define the $L^2$ norm of a operator $T$ in $L^2(\R^3_v)$ as
\be \|T\|=\sup_{\|f\|=1}\|Tf\|. \label{norm1}\ee

\textbf{Notations}: Before stating the main results in this paper, we list some notations. For any $\alpha=(\alpha_1,\alpha_2,\alpha_3)\in\N^3$ and $\beta=(\beta_1,\beta_2,\beta_3)\in\N^3$, denote
$$
\partial_x^{\alpha}=\partial_{x_1}^{\alpha_1}\partial_{x_2}^{\alpha_2}\partial_{x_3}^{\alpha_3},\quad \partial_v^{\beta}=\partial_{v_1}^{\beta_1}\partial_{v_2}^{\beta_2}\partial_{v_3}^{\beta_3}.
$$
The Fourier transform of $f=f(x,v)$ is denoted by
$$\hat{f}(\xi,v)=\mathcal{F}f(x,v)=\frac1{(2\pi)^{3/2}}\int_{\R^3} e^{-\mathrm{i} x\cdot \xi}f(x,v)dx,$$
where and throughout this paper we denote $\mathrm{i}=\sqrt{-1}$.

Let $\langle v\rangle=\sqrt{1+|v|^2},$ and the Sobolev space $ H^N_{k}=\{\,f\in L^2(\R^3_x\times \R^3_v)\,|\,\|f\|_{H^N_{k}}<\infty\}$
equipped with the norm
$$
 \|f\|_{H^N_{k}}=\sum_{|\alpha|+|\beta|\leq N}\|\langle v\rangle^k\dxa \partial_v^{\beta}f\|_{L^2_{x,v}}.
$$

In what follows,  denote by $\|\cdot\|_{L^2_{x,v}}$  the norm of the function space $L^2(\R^3_x\times \R^3_v)$, and by $\|\cdot\|_{L^2_x}$ and $\|\cdot\|_{L^2_{\xi}}$  the norms of the function spaces $L^2(\R^3_x)$ and $L^2(\R^3_{\xi})$, respectively.

First, we have the pointwise space-time behaviors of the Green's function for the linearized mVPB system \eqref{LmVPB}.
\begin{thm}\label{3Dmvpbth1}
Let $G(t,x)$ be the Green's function for the mVPB system defined by \eqref{LmVPB}. Then, the Green's function $G(t,x)$ can be decomposed into
$$
G(t,x)=G_1(t,x)+G_2(t,x)+W_{4}(t,x),
$$
where $W_4(t,x)$ is the singular kinetic wave, $G_1(t,x)$ is the fluid part at low frequency and $G_2(t,x)$ is the remainder part. For $\alpha\in\N^3$, there exist two positive constants $C$ and $D$ such that the fluid part $G_1(t,x)$ is smooth and satisfies
\bma
&\|\partial_x^{\alpha}P_0^lG_1(t,x)\|\leq C(1+t)^{-\frac{3+|\alpha|}{2}}\(e^{-\frac{|x|^2}{D(1+t)}} +(1+t)^{-\frac{1}{2}}e^{-\frac{(|x|-\mathbf{c}t)^2}{D(1+t)}}\)+Ce^{-\frac{|x|+t}{D}},\quad l=1,3,\label{3dmvpbthm111}\\
&\|\partial_x^{\alpha}P_0^2G_1(t,x)\|\leq C(1+t)^{-\frac{3+|\alpha|}{2}}\(e^{-\frac{|x|^2}{D(1+t)}}+(1+t)^{-\frac{1}{2}}e^{-\frac{(|x|-\mathbf{c}t)^2}{D(1+t)}}\)+Ce^{-\frac{|x|+t}{D}}\nnm\\
&\qquad\qquad\qquad\qquad+C(1+t)^{-\frac{3+|\alpha|}{2}}B_{\frac{3}{2}}(t,|x|)1_{\{|x|\leq\mathbf{c}t\}},\label{3dmvpbthm112}\\
&\|\partial_x^{\alpha}P_1G_1(t,x)\|,\|\partial_x^{\alpha}G_1(t,x)P_1\|\leq C(1+t)^{-\frac{4+|\alpha|}{2}}\(e^{-\frac{|x|^2}{D(1+t)}} +(1+t)^{-\frac{1}{2}}e^{-\frac{(|x|-\mathbf{c}t)^2}{D(1+t)}}\)\nnm\\
&\qquad\qquad\qquad\qquad+Ce^{-\frac{|x|+t}{D}}+C(1+t)^{-\frac{4+|\alpha|}{2}}B_{\frac{3}{2}}(t,|x|)1_{\{|x|\leq\mathbf{c}t\}},\label{3dmvpbthm113}\\
&\|\partial_x^{\alpha}P_1G_1(t,x)P_1\|\leq C(1+t)^{-\frac{5+|\alpha|}{2}}\(e^{-\frac{|x|^2}{D(1+t)}} +C(1+t)^{-\frac{1}{2}}e^{-\frac{(|x|-\mathbf{c}t)^2}{D(1+t)}}\)+Ce^{-\frac{|x|+t}{D}}\nnm\\
&\qquad\qquad\qquad\qquad+C(1+t)^{-\frac{5+|\alpha|}{2}}B_{\frac{3}{2}}(t,|x|)1_{\{|x|\leq\mathbf{c}t\}},\label{3dmvpbthm114}
\ema
where $\mathbf{c}=\sqrt{\frac{8}{3}}$ and the space-time diffusive profile $B_k(t,|x|-\lambda t)$ is defined for any $k>0$ and $\lambda\in \R$ by
$$
B_k(t,|x|-\lambda t)=\(1+\frac{(|x|-\lambda t)^2}{1+t}\)^{-k},\quad (t,x)\in\mathbb{R}_{+}\times\mathbb{R}^3.
$$
And the remainder part $G_2(t,x)$ is bounded and satisfies
\bq
\|G_2(t,x)\|\leq Ce^{-\frac{|x|+t}{D}}.\label{3DmvpbThm1-2}
\eq
The singular kinetic wave $W_{4}(t,x)$ is defined by
\bq\label{Wk}
W_{4}(t,x)=\sum^{12}_{k=0}J_k(t,x),
\eq
with
$$
\left\{\bln
&J_0(t,x)=S^t\delta(x)I_v=e^{-\nu(v)t}\delta(x-vt)I_v,\\
&J_k(t,x)=\int^t_0S^{t-s}(K+v\cdot\nabla_x(I-\Delta_x)^{-1}P^1_0)J_{k-1}ds,\quad k\geq1.
\eln\right.
$$
Here $I_v$ is the identity map in $L^2(\mathbb{R}^3_v)$ and the operator $S^t$ is defined by
\bq\label{St}
S^tg(x,v)=e^{-\nu(v)t}g(x-vt,v).
\eq
\end{thm}

Then we have the pointwise space-time behavior of the global solution to the nonlinear mVPB system \eqref{3DmVPB8}--\eqref{3DmVPB10} as follows:

\begin{thm}\label{3Dmvpbth2}
There exists a small constant $\delta_0>0$ such that if the initial data $f_0$ satisfies $\|f_0\|_{H^9_2}\leq\delta_0$ and
\bq\label{3DThm2-1}
\|\dxa f_0(x)\|_{L^{\infty}_{v,3}}+\|\nabla_vf_0(x)\|_{L^{\infty}_{v,2}}\leq C\delta_0e^{-\frac{|x|}{D_1}},
\eq
for $|\alpha|=0,1$ and $D_1>0$, then there exists a unique global solution $(f,\Phi)$ to the mVPB system \eqref{3DmVPB8}--\eqref{3DmVPB10} satisfying
\bma
&|\partial_x^{\alpha}\Phi(t,x)|,\ \|\partial_{x}^{\alpha} f(t,x)\|_{L^{\infty}_{v,3}}\leq C\delta_0(1+t)^{-\frac{3+|\alpha|}{2}}\(e^{-\frac{|x|^2}{D_2(1+t)}}+(1+t)^{-\frac{1}{2}}e^{-\frac{(|x|-\mathbf{c}t)^2}{D_2(1+t)}}\)+C\delta_0e^{-\frac{|x|+t}{D_2}}\nnm\\ &\qquad\qquad\qquad +C\delta_0(1+t)^{-\frac{3+|\alpha|}{2}}\(B_{\frac{3}{2}}(t,|x|)+(1+t)^{-\frac{1}{2}}B_{1}(t,|x|-\mathbf{c}t)\)1_{\{|x|\leq\mathbf{c}t\}},\label{3dmvpbth2-2}\\
&\|\nabla_vf\|_{L^{\infty}_{v,2}}\leq C\delta_0(1+t)^{-\frac{3}{2}}\(e^{-\frac{|x|^2}{D_2(1+t)}}+(1+t)^{-\frac{1}{2}}e^{-\frac{(|x|-\mathbf{c}t)^2}{D_2(1+t)}}\)+C\delta_0e^{-\frac{|x|+t}{D_2}}\nnm\\ &\qquad\qquad\qquad +C\delta_0(1+t)^{-\frac{3}{2}}\(B_{\frac{3}{2}}(t,|x|)+(1+t)^{-\frac{1}{2}}B_{1}(t,|x|-\mathbf{c}t)\)1_{\{|x|\leq\mathbf{c}t\}},\label{3dmvpbth2-3}
\ema
where $D_2>0$ is a constant,  and the space-time diffusive profile $B_k(t,|x|-\lambda t)$ is defined in Theorem \ref{3Dmvpbth1}.
\end{thm}

\begin{remark}
The global existence of the mVPB system can be obtain by standard energy method in the solution space $H^2_k~(k\geq0)$, namely (cf. Lemma \ref{3dmvpbpw2})
$$
\sup_{0\leq s<+\infty}(\|f(s)\|_{H^2_k}+\|\Phi(s)\|_{H^3_k})\leq C\|f_0\|_{H^2_k}.
$$
Moreover, we establish the pointwise estimate of global solution $(f,\Phi)$ based on the Green's function and the  energy estimate.
To estimate $\partial_{x}^{\alpha}f$ $(|\alpha|=1)$,  we have to deal with the nonlinear term $\partial_{x}^{\alpha} (\nabla_x \Phi\cdot\nabla_{v}f)$, that is
$$
\|\partial_{x}^{\alpha} (\nabla_x \Phi\cdot\nabla_{v}f)\|_{L^{\infty}_{v,2}}\leq  |D_{x}^{2} \Phi| \| \nabla_{v}f\|_{L^{\infty}_{v,2}}+| \nabla_{x} \Phi| \|\partial_{x}^{\alpha }\nabla_{v}f\|_{L^{\infty}_{v,2}}.
$$
However, the priori assumption only contains the pointwise estimates of $ f$ and its first order derivatives.  To obtain  the pointwise estimate of $\|\partial_{x_i}\nabla_{v}f\|_{L^{\infty}_{v,2}}$, we apply Gagliardo-Nirenberg interpolation inequality to get  (cf. \eqref{3dmvpb-non-pxvf})
 $$
\|\partial_{x_i} \partial_{v}f\|_{L^{\infty}_{v,2}}\leq CH_{9,2}(f)^{\frac19}\|\partial_{x_i} f\|^{\frac{7}{9}}_{L^{\infty}_{v,2}}+C\|\partial_{x_i} f\|_{L^{\infty}_{v,1}} ,
$$
where $H_{N,k}(f)$ is the high order energy defined by \eqref{energy2}.
 Thus, we can apply the priori assumption to obtain pointwise estimate of $\|\partial_{x_i} \partial_{v}f\|_{L^{\infty}_{v,2}}$. This is the main reason for requiring the initial data $f_0\in H^9_2$.
\end{remark}
The results in Theorem \ref{3Dmvpbth1} on the pointwise behavior of the Green's function $G$ to the mVPB system \eqref{3DmVPB8}--\eqref{3DmVPB10} is proved based on the spectral analysis \cite{LI-2} and the ideas inspired by \cite{LI-5,LIU-2,LIU-3}. First, we estimate the Green's function $G$ inside the Mach region $|x|\leq6t$ based on the spectral analysis. Indeed, we decompose the Green's function $G $ into the lower frequency part $G_L$ and the high frequency part $G_H$, and further split $G_L$ into the fluid parts $G_{L,0}$ and the non-fluid parts $G_{L,1}$, respectively, namely
$$
\left\{\bal
G=G_L+G_H,\\
G_L=G_{L,0}+G_{L,1}.
\ea\right.
$$
By using Fourier analysis techniques, we can show that the low-frequency fluid part $G_{L,0}(t,x)$ is smooth and contains Huygens waves and diffuse waves  since the Fourier transform of the linear mVPB operator $B(\xi)$ has five eigenvalues $\{\lambda_j(\xi),\ j=-1,0,1,2,3\}$ at the low frequency region $|\xi|\leq r_0$ satisfying
\bq
\left\{\bal
\lambda_{\pm1}(|\xi|)=\pm \mathrm{i}\mathbf{c}|\xi|-A_{\pm1}|\xi|^2+O(|\xi|^3),\quad \overline{\lambda_1(|\xi|)}=\lambda_{-1}(|\xi|),\\
\lambda_0(|\xi|)=-A_0|\xi|^2+O(|\xi|^3),\\
\lambda_2(|\xi|)=\lambda_3(|\xi|)=-A_2|\xi|^2+O(|\xi|^3).
\ea\right.
\eq

Next, we apply a Picard's iteration to estimate the Green's function $G$ outside the Mach region $|x|>6t$.
Since $\hat{G}_H(t,\xi)$ does not belong to $L^1_{\xi}(\R^3)$,  $G_H(t,x)$ contains the singular waves and the bounded remainder part. To exact the singular part from $G_H(t,x)$, we defined the approximate sequence $\{\hat{J}_{k}(t,\xi),\ k\ge 0\}$ for the Green's function $\hat{G}(t,\xi)$, where $\hat{J}_{k}$ can be represented by the combination of the mixture operator $\hat{\mathbb{M}}_{k}^t(\xi)$. From the Mixture lemma, $\hat{\mathbb{M}}^t_k(\xi)$ is analytic in $ D_{\nu_0}=\{\xi\in\mathbb{C}^3\,|\,|\mathrm{Im}\xi|\leq\nu_0\}$ and satisfies (refer to Lemma \ref{3dmvpbgf7})
$$
\|\hat{\mathbb{M}}^t_{3k}(\xi)\|\leq C_k(1+t)^{3k}(1+|\xi|)^{-k}e^{- \nu_0t}.
$$
This together with Cauchy theorem implies that  the approximate solution $J_{12}(t,x)$ is bounded and satisfies
\be
\|J_{12}(t,x)\|\leq Ce^{-\frac{\nu_{0}(|x|+t)}{4}}. \label{J6}
\ee
We define the singular part $W_{4}(t,x)$ as
$$
W_{4}(t,x)=\sum^{12}_{i=0}J_i(t,x).
$$

Note that \eqref{J6} implies that the remainder part $G(t, x)- W_4(t, x)$ is bounded for all
$(t, x) \in \R_+\times \R^3$. Thus, by choosing the weighted function as $w=e^{\epsilon(x\cdot\Omega-Yt)}$ with $\Omega\in\S^2$ and using the fact that for any $f=f(x)$ and $\delta\in (0,2)$,
$$
\int_{\R^3}|\nabla^{k}_x(I-\Delta_x)^{-1}f|^2e^{\delta(x\cdot\Omega)}dx\leq\frac{96}{(2-\delta)^4}\int_{\R^3}|f|^2e^{\delta(x\cdot\Omega)}dx,\quad k=0,1,
$$
we can show by the weighted energy method that the remainder part $ G(t,x)-W_4(t,x)$ satisfies (refer to Lemma \ref{3dmvpbgf12})
$$
\|G(t,x)-W_{4}(t,x)\|\leq Ce^{-\frac{|x|+t}{D}}, \quad |x|>6t.
$$
Applying the above decompositions and estimates, we can obtain the decomposition and pointwise space-time behavior for each part of the Green's function $G(t,x)$ as listed in Theorem \ref{3Dmvpbth1}.

Finally, by using the estimates of the Green's functions in Theorem \ref{3Dmvpbth1}, the energy estimate in Lemmas \ref{3dmvpbpw1}--\ref{3dmvpbpw2} and the estimates of waves coupling in Lemmas \ref{3dmvpbpw4j1}--\ref{3dmvpbpw11}, one can establish the pointwise space-time estimate on the global solution to the nonlinear mVPB system given in Theorem \ref{3Dmvpbth2}. Compared to the relativistic Boltzmann equation \cite{LI-8}, we cannot directly obtain pointwise estimates for the global solution of the mVPB system by using Green's function. This is because the direct calculation process the nonlinear term
$$P_0\Lambda_1=n\nabla_x\Phi\cdot v\chi_0+\sqrt{\frac{2}{3}}\nabla_x\Phi\cdot m\chi_4$$
yields the estimate (See \eqref{pw-hbad})
\bmas
&\quad \bigg\|\int^t_0\dxa G_1(t-s)\ast P_0\Lambda_1 ds\bigg\|\nnm\\
&\leq C\int_0^t\int_{\R^3}(1+t-s)^{-\frac{3+|\alpha|}{2}}e^{-\frac{|x-y|^2}{D(1+t-s)}}(1+s)^{-\frac{9}{2}}e^{-\frac{2(|y|-\mathbf{c} s)^2}{D_2(1+s)}}dyds+\cdot\cdot\cdot\nnm\\
&\leq C(1+t)^{-\frac{3+|\alpha|}{2}}e^{-\frac{3(|x|-\mathbf{c} t)^2}{2D_2(1+t)}}+\cdot\cdot\cdot.
\emas
Thus, we cannot close  the priori assumption. To deal with this problem, we  rewrite $P_0\Lambda_1$ as (See \eqref{ndpm}--\eqref{rw-p0h})
$$
P_0\Lambda_1=\nabla_x\cdot(H_1+H_2)+\partial_{t}(\nabla_x\cdot H_3)+\partial_{t}H_4.
$$
Through this form, we enhance the time decay of the Green's function in the convolution by using the properties of  Green's function as listed in Theorem \ref{3Dmvpbth1}, thereby closing  the priori assumption.

Note that comparing to the Green's function, there is an addition wave $(1+t)^{-2}B_1(t,|x|-\mathbf{c}t)1_{\{|x|\le \mathbf{c}t\}}$ in  the solution to the nonlinear system. This addition wave comes from the the following nonlinear waves coupling (See Lemma \ref{3dmvpbpw5}):
\bmas
&\quad \intt \intr (1+t-s)^{-\frac52}e^{-\frac{(|x-y|-\mathbf{c}(t-s))^2}{D(1+t-s)}}(1+s)^{-4}e^{-\frac{2(|y|-\mathbf{c}s)^2}{D_1(1+s)}}dyds\\
&\le  C(1+t)^{-\frac52} \(e^{-\frac{3|x|^2}{2D_1(1+t)}}+e^{-\frac{3(|x|-\mathbf{c} t)^2}{2D_1(1+t)}}\)+C(1+t)^{-2}B_1(t,|x|-\mathbf{c}t)1_{\{|x|\le \mathbf{c}t\}}.
\emas

The rest of this paper is organized as follows. In section \ref{sect2}, we present the results regarding the spectrum analysis of the linear operators related to the linearized mVPB system. In section \ref{sect3}, we establish the pointwise space-time estimates of the Green's functions to the linearized mVPB system. In section \ref{sect4}, we prove the pointwise space-time estimates of the global solutions to the original nonlinear mVPB system.

\section{Spectral analysis}
\label{sect2}
\setcounter{equation}{0}

In this section, we establish the spectral structure of the linearized three-dimensional mVPB system \eqref{3DmVPB} and show the analyticity of the eigenvalue and eigenfunction in order to study the pointwise estimate of the fluid part of the Green's function. Firstly, we take the Fourier transform to \eqref{LmVPB} with respect to $x$ to have
\bq\label{gf-flg}
\left\{\bal
\partial_t\hat{G}=B(\xi)\hat{G},\quad t>0,\\
\hat{G}(\xi,0)=1(\xi)I_v,
\ea\right.
\eq
where the operator $B(\xi)$ is defined by
\bq
B(\xi)=L-\mathrm{i}v\cdot\xi-\frac{\mathrm{i}(v\cdot\xi)}{1+|\xi|^2}P_0^1.\label{3Dmvpbfb}
\eq

Introduce the weighted Hilbert space $L^2_{\xi}(\R^3_v)$ as
$$
L^2_{\xi}(\R^3_v)=\Big\{f\in L^2(\R^3_v)|\|f\|_{\xi}=\sqrt{(f,f)_{\xi}}<\infty\Big\},
$$
equipped with the inner product
$$
(f,g)_{\xi}=(f,g)+\frac{1}{1+|\xi|^2}(P_0^1f,P_0^1g).
$$

Let $\sigma(B(\xi))$ and $\rho(B(\xi))$ denote the spectrum set and the resolvent set of the operator $B(\xi)$ respectively. First, we have the following  results about the spectral analysis of the operator $B(\xi)$.

\begin{lem}[\cite{LI-2}]\label{3dmvpbsp1}
The following results hold.

(1) For any $\delta>0$ and all $\xi\in\R^3$, there exists $y_1(\delta)>0$ such that
\be
\rho(B(\xi))\supset\{\lambda\in\mathbb{C}\,|\,\mathrm{Re}\lambda\geq-\nu_0+\delta,\,|\mathrm{Im}\lambda|\geq y_1\}\cup\{\lambda\in\mathbb{C}\,|\,\mathrm{Re}\lambda>0\}.
\ee

(2) For any $r_0>0$, there exists $\alpha=\alpha(r_0)>0$ such that for $|\xi|\geq r_0$,
\be
\sigma(B(\xi))\subset\{\lambda\in\mathbb{C}\,|\,\mathrm{Re}\lambda<-\alpha\}.
\ee
\end{lem}


\begin{lem}[\cite{LI-2}]\label{3dmvpbsp2}
(1) There exists a constant $r_0>0$ such that the spectrum $\sigma(B(\xi))\cap\{\mathrm{Re}\lambda\geq-\frac{\mu}{2}\}$ consists of five points $\{\lambda_j(|\xi|),j=-1,0,1,2,3\}$ for $|\xi|\leq r_0$. In particular, the eigenvalue $\lambda_j(|\xi|)$ is $C^\infty$ in $|\xi|$ and admits the following asymptotic expansion for $|\xi|\leq r_0$,
\bq\label{3dmvpbSA1}
\left\{\bal
\lambda_{\pm1}(|\xi|)=\pm \mathrm{i}\mathbf{c}|\xi|-A_{\pm1}|\xi|^2+O(|\xi|^3),\quad \overline{\lambda_1(|\xi|)}=\lambda_{-1}(|\xi|),\\
\lambda_0(|\xi|)=-A_0|\xi|^2+O(|\xi|^3),\\
\lambda_2(|\xi|)=\lambda_3(|\xi|)=-A_2|\xi|^2+O(|\xi|^3),
\ea\right.
\eq
where $\mathbf{c}=\sqrt{\frac{8}{3}}$ and $A_j>0$ with $j=-1,0,1,2.$

(2) The semigroup $S(t,\xi)=e^{tB(\xi)}$ with $\xi\in\R^3$ has the following decomposition:
\bq\label{3dmvpbsami1}
S(t,\xi)f=S_1(t,\xi)f+S_2(t,\xi)f,\quad f\in L^2_{\xi}(\R^3_v),\quad t\geq0,
\eq
where
\bq\label{3dmvpbsami2}
S_1(t,\xi)=\sum^3_{j=-1}e^{\lambda_j(|\xi|)t}(f,\overline{\psi_j(\xi)})_{\xi}\psi_j(\xi)1_{\{|\xi|\leq r_0\}},
\eq
with $(\lambda_j(|\xi|),\psi_j(\xi))$ being the eigenvalue and eigenfunction of the operator $B(\xi)$ for $|\xi|\leq r_0$, and $S_2(t,\xi)=:S(t,\xi)f-S_1(t,\xi)f$ satisfies for a constant $\sigma_0>0$ independent of $\xi$ that
\bq\label{3dmvpbsami3}
\|S_2(t,\xi)f\|_{\xi}\leq Ce^{-\sigma_0t}\|f\|_{\xi},\quad t\geq0.
\eq
\end{lem}

Now we analyze the analyticity and expansion of the eigenvalues and eigenfunctions of $B(\xi)$ at low frequency. To this end, we first consider one-dimensional eigenvalue problem:
\bq
\(L-\mathrm{i}v_1\eta-\frac{\mathrm{i}v_1\eta}{1+\eta^2}P_0^1\)e=\beta e, \quad \eta\in \R.\label{3dmvpb1dSA1}
\eq

We have the analyticity and expansion of the eigenvalues $\beta_j(\eta)$ and the corresponding eigenfunctions $e_j(\eta)$ at low frequency as follows:
\begin{lem}\label{3dmvpbsp3}
Let $\sigma(\eta)$ denote the spectral set of the operator $L-\mathrm{i}v_1\eta-\frac{\mathrm{i}v_1\eta}{1+\eta^2}$. We have

(1) For any $r_1>0$ there exists $\alpha=\alpha(r_1)>0$ such that it holds for $|\eta|\geq r_1$,
\bq
\sigma(\eta)\subset\{\lambda\in\mathbb{C}|\mathrm{Re}\lambda<-\alpha\}.\label{3dmvpb1dSA2}
\eq

(2) There exists a constant $r_0>0$ such that the spectrum $\sigma(\eta)\cap\{\mathrm{Re}\lambda\geq-\frac{\mu}{2}\}$ consists of five points $\{\beta_j(\eta),j=-1,0,1,2,3\}$ for $|\eta|\leq r_0$. The eigenvalues $\beta_j(\eta)$ are analytic functions of $\eta$ and satisfy
\bq\label{3dmvpb1dSA4}
\left\{\bal
\beta_{\pm1}(\eta)=\pm\mathrm{i}\eta A^1_1(\eta^2)-A^2_1(\eta^2),\\
\beta_{0}(\eta)=-A_0^2(\eta^2),\\
\beta_{2}(\eta)=\beta_{3}(\eta)=-A^2_2(\eta^2),
\ea\right.
\eq
where $A^1_j(\eta)~(j=0,1,2)$ and $A^2_1(\eta)$ are analytic functions in $\eta$ and satisfy
\bq\label{3dmvpb1dSA5}
\left\{\bal
A^1_1(0)=\mathbf{c},\quad A^2_1(0)=A_0^2(0)=A^2_2(0)=0,\\
\mathrm{\frac{d}{d\eta}}A^2_1(0)=A_1,\quad \mathrm{\frac{d}{d\eta}}A^2_0(0)=A_0,\quad \mathrm{\frac{d}{d\eta}}A^2_2(0)=A_2,\\
A_{j}=-(L^{-1}P_1v_1E_{j},v_1E_{j})>0, \quad  j=0,1,2,
\ea\right.
\eq
with $E_j$, $j=-1,0,1$ given by \eqref{EUJ}.

(3) The corresponding eigenfunctions $e_j(\eta)~(j=-1,0,1,2,3)$ are analytic in $\eta$ and satisfies
\bq\label{mvpbSA6j1}
\left\{\bal
P_0e_{\pm1}(\eta)=\(a_{1,0}(\eta^2)\pm\mathrm{i}\eta a_{1,1}(\eta^2)\)\chi_0\pm\(b_{1,0}(\eta^2)\pm\mathrm{i}\eta b_{1,1}(\eta^2)\)\chi_1\\
\qquad\qquad\quad+\(c_{1,0}(\eta^2)\pm\mathrm{i}\eta c_{1,1}(\eta^2)\)\chi_4,\\
P_0e_{0}(\eta)=a_0(\eta^2)\chi_0+ \mathrm{i}\eta b_0(\eta^2)\chi_1+c_0(\eta^2)\chi_4,\\
P_0e_{k}(\eta)=c_2(\eta^2)\chi_k,\quad k=2,3,\\
P_1e_j(\eta)=\mathrm{i}(L-\beta_j-\mathrm{i}\eta P_1v_1 P_1)^{-1}P_1(v_1 P_0e_j(\eta)),
\ea\right.
\eq
where $a_0(\eta),b_0(\eta),c_0(\eta),$ $c_2(\eta)$, $a_{1,l}(\eta), b_{1,l}(\eta), c_{1,l}(\eta)$ with $l=0,1$ are real, analytic functions of $\eta$ and satisfy
\be \label{3dmVPB-sp-abc}
\left\{\bal
a_0(0)=\frac{\sqrt{2} }{4} ,\quad b_0(0)=-\frac{\sqrt{2}(A_{0,-1}+A_{0,1})}{2\mathbf{c}},\quad c_0(0)=-\frac{\sqrt{3 }}{2} ,\quad  b_{1,0}(0)=-\sqrt{\frac{1}{2}}, \\
a_{1,0}(0)= \frac{\sqrt{3} }{4 }, \quad a_{1,1}(0)= \frac{ \sqrt{3} A_{1,-1}+2\sqrt{2}A_{1,0}}{8\mathbf{c}} ,\quad  b_{1,1}(0)=\frac{A_{1,-1}}{2\sqrt{2}\mathbf{c}}, \\
c_{1,0}(0)=\frac{\sqrt{2} }{4},\quad  c_{1,1}(0)= \frac{\sqrt{2}A_{1,-1}-4\sqrt{3}A_{1,0}}{8\mathbf{c}} ,\quad   c_2(0)=1,\\
A_{i,j}=-(L^{-1}P_1v_1E_{i},v_1E_{j})>0, \quad  i,j=-1,0,1.
\ea\right.
\ee
\end{lem}
\begin{proof}
\eqref{3dmvpb1dSA2} is already given in Lemma \ref{3dmvpbsp1}. We shall prove that eigenvalues $\beta_j(\eta)~(j=-1,0,1,2,3)$ are analytic functions in $\eta$ and satisfy \eqref{3dmvpb1dSA4} and \eqref{3dmvpb1dSA5}. By the macro-micro decomposition, the eigenfunction $e$ of \eqref{3dmvpb1dSA1} can be divided into
\bq
e=P_0e+P_1e=g_0+g_1.\label{mvpb1dp1}
\eq
Let $\beta=\eta\gamma$. Hence, \eqref{3dmvpb1dSA1} gives
\bma
&\eta\gamma g_0=-P_0\(\mathrm{i}v_1\eta (g_0+g_1)\)-\frac{\mathrm{i}v_1\eta }{1+\eta^2}P^1_0g_{0},\label{mvpb1dp2}\\
&\eta\gamma g_1=Lg_1-P_1\(\mathrm{i}v_1\eta (g_0+g_1)\).\label{mvpb1dp3}
\ema
By \eqref{mvpb1dp3}, the microscopic part $g_1$ can be represented in terms of the macroscopic part $g_0$ as
\bq
g_1=\mathrm{i}\eta(L-\eta\gamma-\mathrm{i}\eta P_1v_1)^{-1}P_1v_1g_0,\quad \mathrm{Re}(\eta\gamma)>-\mu.\label{mvpb1dp4}
\eq
Substituting it into \eqref{mvpb1dp2}, we obtain the eigenvalue problem for the macroscopic part $g_0$ as
\bq
\gamma g_0=-\mathrm{i}P_0v_1g_0-\frac{\mathrm{i}v_1}{1+\eta^2}P^1_0g_0+\eta P_0\(v_1R(\gamma,\eta)P_1v_1g_{0}\),\quad \mathrm{Re}(\eta\gamma)>-\mu,\label{mvpb1dp5}
\eq
where $R(\gamma,\eta)=(L-\eta\gamma-\mathrm{i}\eta P_1v_1)^{-1}$.

Define the operator $\mathcal{D}=P_0v_1P_0+\frac{v_1}{1+\eta^2}P^1_0$. We have matrix representation of $\mathcal{D}$ as follows
\bq
\mathcal{D}=\left(
  \begin{array}{ccccc}
    0 & 1 & 0 & 0 & 0 \\
   1+\frac{1}{1+\eta^2} & 0 & 0 & 0 & \sqrt{\frac{2}{3}} \\
    0 & 0 & 0 & 0 & 0 \\
    0 & 0 & 0 & 0 & 0 \\
    0 & \sqrt{\frac{2}{3}} & 0 & 0 & 0 \\
  \end{array}
\right).
\eq
It can be verified that the eigenvalues $u_j(\eta)$ and  eigenvectors $E_j(\eta)$ $(j=-1,0,1,2,3)$ of $\mathcal{D}$ are given by
\bq\label{EUJ}
\left\{\bal
u_{\pm1}(\eta)=\mp\sqrt{\frac{5}{3}+\frac{1}{1+\eta^2}},\quad u_j(\eta)=0, \quad j=0,2,3,\\
E_{\pm1}(\eta)=\frac{1}{\sqrt{\frac{10}{3}+\frac{2}{1+\eta^2}}}\chi_0\mp\frac{\sqrt{2}}{2}\chi_1+\frac{1}{\sqrt{5+\frac{3}{1+\eta^2}}}\chi_4,\\
E_0(\eta)=\frac{1}{\sqrt{ (1+\frac{1}{1+\eta^2})+\frac32(1+\frac{1}{1+\eta^2})^2}}\chi_0-\frac{\sqrt{1+\frac{1}{1+\eta^2}}}{\sqrt{\frac{5}{3}+\frac{1}{1+\eta^2}}}\chi_4,\\
E_k(\eta)=\chi_k,\quad k=2,3,\\
(E_j(\eta),E_k(\eta))_{\eta}=\delta_{jk},\quad -1\leq j,k\leq3.
\ea\right.
\eq
Moreover, we denote $E_j=E_j(0)$ $(j=-1,0,1,2,3)$.

To solve the eigenvalue problem \eqref{mvpb1dp5}, we rewrite $g_0\in N_0$ in terms of the basis $E_j(\eta)$ as
\bq
g_0=\sum^{4}_{j=0}C_jE_{j-1}(\eta)\quad \mathrm{ with}\quad C_j=(g_0,E_{j-1}(\eta))_{\eta},\quad j=0,1,2,3,4.
\eq
Taking the inner product $(\cdot,\cdot)_{\eta}$ between \eqref{mvpb1dp5} and $E_j(\eta)~(j=-1,0,1,2,3)$ respectively, we have the following equations about $\gamma$ and $(C_0,C_1,C_2,C_3,C_4)$ for $\mathrm{Re}(\eta\gamma)>-\mu$:
\bma
&\gamma C_j=-\mathrm{i}u_{j-1}(\eta)C_j+\eta\sum^2_{k=0}C_kR_{kj}(\gamma,\eta),\quad j=0,1,2,\label{mvpbspew6}\\
&\gamma C_l=\eta C_lR_{33}(\gamma,\eta),\quad l=3,4,\label{mvpbspew7}
\ema
where
\bq
R_{kj}=R_{kj}(\gamma,\eta)=((L-\eta\gamma -\mathrm{i}\eta P_1v_1)^{-1}P_1(v_1E_{k-1}(\eta)),v_1E_{j-1}(\eta)).\label{mvpbsprkj1J1}
\eq

Denote
\bma
D_0(\gamma,\eta)&=\gamma-\eta R_{33}(\gamma,\eta),\label{3dmvpbspd0}\\
D_1(\gamma,\eta)&=\left|\begin{array}{ccc}
                  \gamma+\mathrm{i}u_{-1}(\eta)-\eta R_{00} & -\eta R_{10} & -\eta R_{20} \\
                  -\eta R_{01} & \gamma-\eta R_{11} & -\eta R_{21} \\
                  -\eta R_{02} & -\eta R_{12} & \gamma+\mathrm{i}u_{1}(\eta)-\eta R_{22}
                \end{array}\right|.\label{3dmvpbspd1}
\ema
The eigenvalue $\beta=\gamma\eta$ can be solved by $D_0(\gamma,\eta)=0$ and $D_1(\gamma,\eta)=0$. By a direct computation and the implicit
function theorem, we can show
\begin{lem}[\cite{LI-7}]\label{3dmvpbsp4}
(1) The equation $D_0(\gamma,\eta)=0$ has a unique analytic solution $\gamma=\gamma(\eta)$ for $(\eta,\gamma)\in[-\tau_0,\tau_0]\times B_{\tau_1}(0)$ with $\tau_0,\tau_1>0$ being small constants that satisfies
$$
\gamma(0)=0,\quad \gamma'(0)=(L^{-1}P_1v_1E_2,v_1E_2).
$$
Moreover, $\gamma(\eta)$ is a real, odd function.

(2) There exist two small constants $\tau_0,\tau_1>0$ so that the equation $D_1(\gamma,\eta)=0$ admits three analytic solution $\gamma_j(\eta)$ with $ j=-1,0,1$ for $(\eta,\gamma_j)\in[-\tau_0,\tau_0]\times B_{\tau_1}(-\mathrm{i}u_j(0))$ that satisfies
\be
\gamma_j(0)=-\mathrm{i}u_j(0),\quad \gamma'_j(0)=(L^{-1}P_1v_1E_j,v_1E_j),\quad j=1,0,1.\label{3dmvpbsp5-1}
\ee
Moreover, $\gamma_j(\eta)$ satisfies
\bq
-\gamma_j(-\eta)=\overline{\gamma_j(\eta)}=\gamma_{-j}(\eta), \quad j=-1,0,1.\label{3dmvpbsp5-2}
\eq
\end{lem}

The eigenvalues $\beta_j=\eta\gamma_j(\eta)$ and the corresponding eigenfunctions $e_j(\eta)$ with $j=-1,0,1,2,3$ can be constructed as follows. For $j=2,3$, we take $\beta_j=\eta\gamma(\eta)$ with $\gamma(\eta)$ being the solution of the equation $D_0(\gamma,\eta)=0$, and choose $C_0=C_1=C_2=0$. And the corresponding eigenfunctions $e_j(\eta)$, $j=2,3$ are defined by
\bq\label{mvpbeginf1J1}
e_j(\eta)=C_3(\eta)E_j(\eta)+\mathrm{i}\eta(L-\beta_j-\mathrm{i}\eta P_1v_1)^{-1}P_1(v_1C_3(\eta) E_j(\eta)).
\eq
The coefficient $C_3(\eta)$ in \eqref{mvpbeginf1J1} is determined by the normalization condition $(e_j(\eta),\overline{e_j(\eta)})_{\eta}=1$:
$$
C_3(\eta)^2(1+\eta^2D(\eta))=1,
$$
where
$$
D(\eta)=((L-\beta_j-\mathrm{i}\eta P_1v_1)^{-1}P_1(v_1E_2(\eta)),(L-\overline{\beta_j}+\mathrm{i}\eta P_1v_1)^{-1}P_1(v_1E_2(\eta))).
$$
This together with $D(\eta)=D(-\eta)$ give $C_3(\eta)=C_3(-\eta)$ and $C_3(0)=1$.

For $j=-1,0,1$, we choose $\beta_j=\eta\gamma_j(\eta)$ with $\gamma_j(\eta)$ being the solution of $D_1(\gamma_j,\eta)=0$, and denote by $ \{C_0^j,C_1^j,C_2^j\}$ a solution of system \eqref{mvpbspew6} for $\gamma_j=\gamma_j(\eta)$. Thus, we can construct the corresponding eigenfunctions $e_j(\eta),~j=-1,0,1$ as
\bq\label{mvpb1degf2}
\left\{\bal
e_j(\eta)=P_0e_j(\eta)+P_1e_j(\eta),\\
P_0e_j(\eta)=C_0^j(\eta)E_{-1}(\eta)+C_1^j(\eta)E_0(\eta)+C_2^j(\eta)E_1(\eta),\\
P_1e_j(\eta)=\mathrm{i}\eta(L-\beta_j-\mathrm{i}\eta P_1v_1)^{-1}P_1(v_1P_0e_j(\eta)).
\ea\right.
\eq
By \eqref{mvpbspew6}, the coefficients $C_0^j(\eta),C_1^j(\eta),C_2^j(\eta)$ with $j=-1,0,1$ in \eqref{mvpb1degf2} satisfies the following system:
\bq\label{mvpbeginf3J1}
\left\{\bal
\(\beta_j+\mathrm{i}\eta u_{-1}(\eta)-\eta^2R_{00}\)C_0^j(\eta)-\eta^2R_{10}C_1^j(\eta)-\eta^2R_{20}C_2^j(\eta)=0,\\
-\eta^2R_{01}C_0^j(\eta)+\(\beta_j-\eta^2R_{11}\)C_1^j(\eta)-\eta^2R_{21}C_2^j(\eta)=0,\\
-\eta^2R_{02}C_0^j(\eta)-\eta^2R_{12}C_1^j(\eta)+\(\beta_j-\mathrm{i}\eta u_1(\eta)-\eta^2R_{22}\)C_2^j(\eta)=0,
\ea\right.
\eq
where $R_{kl}(\beta,\eta)=((L-\beta -\mathrm{i}\eta P_1v_1)^{-1}P_1(v_1E_{k-1}(\eta)),v_1E_{l-1}(\eta))$.

Since $R_{kj}(\beta,\eta)$, $k,j=0,1,2$ are analytic in $(\beta,\eta)$ and $\beta_j(\eta)=\eta\gamma_j(\eta)$, $j=-1,0,1$ are analytic functions of $\eta$ for $|\eta|\leq r_0$, then it follows that $C_0^j(\eta),C_1^j(\eta),C_2^j(\eta)$ for $j=-1,0,1$ are analytic functions of $\eta$ for $|\eta|\leq r_0$. By Lemma 2.2 in \cite{LI-7}, it holds that
\be \label{rbespbet01z}
\left\{\bal
C^j_{j+1}(0)=1,\quad C^j_{k}(0)=0,\quad k\ne j,\\
\frac{\rm d}{\rm d \eta}C^j_{j+1}(0)=0,\quad \frac{\rm d}{\rm d \eta}C^j_{k+1}(0)=\frac{\mathrm{i}(L^{-1}P_1v_1E_{j},v_1E_{k})}{(u_{j}(0)-u_k(0))},\quad k\ne j.
\ea\right.
\ee
Since $\beta_0(\eta)=\beta_0(-\eta)$, then we have by changing variable $v_1\rightarrow-v_1$ that
\bq\label{rbespr1}
\left\{\bln
&R_{00}(\beta_0,-\eta)=R_{22}(\beta_0,\eta),\quad R_{10}(\beta_0,-\eta)=R_{12}(\beta_0,\eta),\quad R_{01}(\beta_0,-\eta)=R_{21}(\beta_0,\eta),\\
&R_{20}(\beta_0,-\eta)=R_{02}(\beta_0,\eta),\quad R_{11}(\beta_0,-\eta)=R_{11}(\beta_0,\eta),
\eln\right.
\eq
which together with \eqref{mvpbeginf3J1}--\eqref{rbespbet01z} implies that $C_0^0(-\eta)=C_2^0(\eta)$ and $C_1^0(-\eta)=C_1^0(\eta)$.

Since $\beta_{-1}(\eta)=\beta_1(-\eta)$, then we have by changing variable $v_1\rightarrow-v_1$ that
$$
\left\{\bal
R_{00}(\beta_{-1},-\eta)=R_{22}(\beta_1,\eta),\quad R_{10}(\beta_{-1},-\eta)=R_{12}(\beta_1,\eta),\quad R_{01}(\beta_{-1},-\eta)=R_{21}(\beta_1,\eta),\\
R_{20}(\beta_{-1},-\eta)=R_{02}(\beta_1,\eta),\quad R_{11}(\beta_{-1},-\eta)=R_{11}(\beta_1,\eta),
\ea\right.
$$
which together with \eqref{mvpbeginf3J1}--\eqref{rbespbet01z} implies that $C_0^{-1}(-\eta)=C_2^{1}(\eta)$, $C_1^{-1}(-\eta)=C_1^{1}(\eta)$ and $C_2^{-1}(-\eta)=C_0^{1}(\eta)$.

Thus, we can rewrite $P_0e_j(\eta)$, $j=-1,0,1$ as
\bq \label{3dmVPB-sp-abc2}
\left\{\bal
P_0e_{-1}(\eta)=\(C^1_{2,0}(\eta^2)- \mathrm{i}\eta C^1_{2,1}(\eta^2)\)E_{-1}(\eta)+\(C^1_{1,0}(\eta^2)- \mathrm{i}\eta C^1_{1,1}(\eta^2)\)E_0(\eta)\\
\qquad\qquad\quad+\(C^1_{0,0}(\eta^2)- \mathrm{i}\eta C^1_{0,1}(\eta^2)\)E_1(\eta),\\
P_0e_{0}(\eta)=\(C^0_{0,0}(\eta^2)+\mathrm{i}\eta C^0_{0,1}(\eta^2)\)E_{-1}(\eta)+C^0_{1}(\eta^2)E_0(\eta)\\
\qquad\qquad\quad+\(C^0_{0,0}(\eta^2)-\mathrm{i}\eta C^0_{0,1}(\eta^2)\)E_1(\eta),\\
P_0e_{1}(\eta)=\(C^1_{0,0}(\eta^2)+ \mathrm{i}\eta C^1_{0,1}(\eta^2)\)E_{-1}(\eta)+\(C^1_{1,0}(\eta^2)+ \mathrm{i}\eta C^1_{1,1}(\eta^2)\)E_0(\eta)\\
\qquad\qquad\quad+\(C^1_{2,0}(\eta^2)+ \mathrm{i}\eta C^1_{2,1}(\eta^2)\)E_1(\eta).
\ea\right.
\eq
By \eqref{EUJ} and \eqref{3dmVPB-sp-abc2},  we can obtain \eqref{mvpbSA6j1}  with
$$
\left\{\bal
a_{1,j}(\eta^2)=\frac{1}{\sqrt{\frac{10}{3}+\frac{2}{1+\eta^2}}}\(C^1_{0,j}(\eta^2)+C^1_{2,j}(\eta^2)\)+\frac{1}{\sqrt{ (1+\frac{1}{1+\eta^2})+\frac32(1+\frac{1}{1+\eta^2})^2}}C^1_{1,j}(\eta^2),\\
b_{1,j}(\eta^2)=\frac{\sqrt{2}}{2}\(C^1_{0,j}(\eta^2)-C^1_{2,j}(\eta^2)\),\\
c_{1,j}(\eta^2)=\frac{1}{\sqrt{5+\frac{3}{1+\eta^2}}}\(C^1_{0,j}(\eta^2)+C^1_{2,j}(\eta^2)\)-\frac{\sqrt{1+\frac{1}{1+\eta^2}}}{\sqrt{\frac{5}{3}+\frac{1}{1+\eta^2}}}C^1_{1,j}(\eta^2),\\
a_0(\eta^2)=\frac{2}{\sqrt{\frac{10}{3}+\frac{2}{1+\eta^2}}}C^0_{0,0}(\eta^2)+\frac{1}{\sqrt{ (1+\frac{1}{1+\eta^2})+\frac32(1+\frac{1}{1+\eta^2})^2}}C^0_1(\eta^2),\\
b_0(\eta^2)=\sqrt{2}C^0_{0,1}(\eta^2),\quad
c_0(\eta^2)=\frac{2}{\sqrt{5+\frac{3}{1+\eta^2}}}C^0_{0,0}(\eta^2)-\frac{\sqrt{1+\frac{1}{1+\eta^2}}}{\sqrt{\frac{5}{3}+\frac{1}{1+\eta^2}}}C^0_1(\eta^2).
\ea\right.
$$
The proof of the lemma is completed.
\end{proof}

Next, we consider the three-dimensional eigenvalue problem:
\bq\label{3dmvpb3dspa1}
\(L-\mathrm{i}v\cdot\xi-\frac{\mathrm{i}v\cdot\xi}{1+|\xi|^2}\)\psi=\lambda\psi,\quad \xi\in\R^3.
\eq

With help of Lemma \ref{3dmvpbsp3}, we have the analyticity and expansion of the eigenvalues $\lambda_j(|\xi|)$ and the corresponding eigenfunctions $\psi_j(\xi)$ of $B(\xi)$ at low frequency.
\begin{lem}\label{3dmvpbsp6}
(1) The eigenvalues $\lambda_j(|\xi|)$, $j=-1,0,1,2,3$ are given in \eqref{3dmvpbSA1} are analytic functions of $|\xi|$ and satisfiy
\bq\label{3dmvpb-6-1}
\left\{\bal
\lambda_{\pm1}(|\xi|)=\pm\mathrm{i}|\xi| A^1_1(|\xi|^2)-A^2_1(|\xi|^2),\\
\lambda_{0}(|\xi|)=-A_0^2(|\xi|^2),\\
\lambda_{2}(|\xi|)=\lambda_{3}(|\xi|)=-A^2_2(|\xi|^2),
\ea\right.
\eq
where $A^1_1(\eta)$ and $A^2_j(\eta)$, $j=0,1,2$ are analytic functions in $\eta$ given in Lemma \ref{3dmvpbsp3}.

(2) The corresponding eigenfunctions $\psi_j(\xi)$, $j=-1,0,1,2,3$ satisfies
\bq\label{3dmvpb3dspa2}
\left\{\bal
P_0\mathbf{\psi}_{\pm1}(\xi)=\(a_{1,0}(|\xi|^2)\pm\mathrm{i}|\xi| a_{1,1}(|\xi|^2)\)\chi_0\pm\(b_{1,0}(|\xi|^2)\pm\mathrm{i}|\xi| b_{1,1}(|\xi|^2)\)\frac{\xi\cdot v}{|\xi|}\chi_0\\
\qquad\qquad\quad+\(c_{1,0}(|\xi|^2)\pm\mathrm{i}|\xi| c_{1,1}(|\xi|^2)\)\chi_4,\\
P_0\mathbf{\psi}_{0}(\xi)=a_0(|\xi|^2)\chi_0+ b_0(|\xi|^2)(\xi\cdot v\chi_0)+c_0(|\xi|^2)\chi_4,\\
P_0\mathbf{\psi}_{k}(\xi)=c_2(|\xi|^2)(W^k\cdot v)\chi_0,\quad k=2,3\\
P_1\mathbf{\psi}_j(\xi)=\mathrm{i}\(L-\beta_j(|\xi|)-\mathrm{i}P_1v\cdot\xi\)^{-1}P_1(v\cdot\xi P_0\mathbf{\psi}_j(\xi)),
\ea\right.
\eq
where $a_0(\eta),b_0(\eta),c_0(\eta)$, $c_2(\eta)$, $a_{1,j}(\eta),b_{1,j}(\eta),c_{1,j}(\eta)$ for $j=0,1$ are analytic functions of $\eta$ given in Lemma \ref{3dmvpbsp3}, and $W^j~(j=2,3)$ are orthonormal vectors satisfying $W^j\cdot\xi=0$.
\end{lem}
\begin{proof}
Let $\mathbb{O}$ be a rotational transformation in $\R^3$ such that $\mathbb{O}:\frac{\xi}{|\xi|}\rightarrow(1,0,0)$. We have
\bq
\mathbb{O}^{-1}\(L-\mathrm{i}v\cdot\xi-\frac{\mathrm{i}v\cdot\xi}{1+|\xi|^2}\)\mathbb{O}=L-\mathrm{i}v_1\eta-\frac{\mathrm{i}v_1\eta}{1+\eta^2}.\label{3dmvpbOtrans1}
\eq
Thus, by Lemma \ref{3dmvpbsp3}, we have the following eigenvalues and eigenfunctions for \eqref{3dmvpb3dspa1}
\bma
&\(L-\mathrm{i}v\cdot\xi-\frac{\mathrm{i}v\cdot\xi}{1+|\xi|^2}\)\psi_j(\xi)=\lambda_j(\xi)\psi_j(\xi),\nnm\\
&\lambda_j(\xi)=\beta_j(|\xi|),\quad \psi(\xi)=\mathbb{O}e_j(|\xi|),\quad j=-1,0,1,2,3.\nnm
\ema
This proves the Lemma.
\end{proof}

\section{Green's function}
\label{sect3}
\setcounter{equation}{0}

In this section, we establish the pointwise space-time estimates of the Green's function to the three-dimensional linear mVPB system \eqref{3DmVPB8}--\eqref{3DmVPB10}. First, we consider the Green's function in finite Mach region. Based on the spectral analysis given in section \ref{sect2}, we divide the Green's function into the fluid part and the non-fluid part and estimate the fluid part by complex analytical techniques. Then, we estimate the Green's function outside finite Mach region. We apply a Picard's iteration to construct the singular wave for the Green's function, and estimate the remainder part by weighted energy estimate.

\subsection{Fluid part}

In this subsection, we establish the pointwise estimates of the fluid part of  Green's function  based on the spectral analysis given in Section \ref{sect2}. To this end, we decompose the operator $G(t,x)$ into a low-frequency part and a high-frequency part:
\be \label{GL-H}
\left\{\bln
&G(t,x)=G_L(t,x)+G_H(t,x),\\
&G_L(t,x)=\frac1{(2\pi)^{\frac{3}{2}}}\int_{\{|\xi|\leq\frac{r_0}{2}\}} e^{ \mathrm{i}x\cdot\xi +tB(\xi)}d\xi,\\
&G_H(t,x)=\frac1{(2\pi)^{\frac{3}{2}}}\int_{\{|\xi|>\frac{r_0}{2}\}} e^{ \mathrm{i}x\cdot\xi +tB(\xi)}d\xi.
\eln\right.
\ee
The operator $G_L(t,x)$ can be further divided into the fluid part and the non-fluid part:
\bq\label{GL-H2}
G_L(t,x)=G_{L,0}(t,x)+G_{L,1}(t,x),
\eq
where
\bma
&G_{L,0}(t,x)=\sum^{3}_{j=-1}\frac{1}{(2\pi)^{\frac{3}{2}}}\int_{\{|\xi|\leq\frac{r_{0}}{2}\}}e^{\mathrm{i}x\cdot\xi +\lambda_j(|\xi|)t}\psi_{j}(\xi)\otimes\bigg\langle \psi_{j}(\xi)+\frac{1}{1+|\xi|^2}P_0^1\psi_{j}(\xi)\bigg|d\xi,\label{GL-HGL0}\\
&G_{L,1}(t,x)=G_{L}(t,x)-G_{L,0}(t,x),
\ema
where $(\lambda_j,\psi_j)$ are defined by \eqref{3dmvpbSA1} and \eqref{3dmvpb3dspa2}.
Here for any $f,g\in L^2(\mathbb{R}^3_v)$, the operator $f\otimes\langle g|$ on $L^2(\mathbb{R}^3_v)$ is defined by \cite{LIU-2,LIU-3}
$$
f\otimes\langle g|u=(u,\overline{g})f, \quad u\in L^2(\mathbb{R}^3_v).
$$

By \cite{LI-2}, we have the following estimates of each part of $\hat{G}(t,\xi)$.
\begin{lem}\label{3dmvpbgf1}
 For any $g_{0}\in L^{2}(\mathbb{R}^{3}_{v})$, there exist positive constants $C$ and $\kappa_{0}$ such that
\bmas
&\|\hat{G}_{L}(t,\xi)g_{0}\|\leq \|g_{0}\|,\\
&\|\hat{G}_{L,1}(t,\xi)g_{0}\|\leq Ce^{-\kappa_{0}t}\|g_{0}\|,\\
&\|\hat{G}_{H}(t,\xi)g_{0}\|\leq Ce^{-\kappa_{0}t}\|g_{0}\|,
\emas
where $\hat{G}_{L}(t,\xi)$, $\hat{G}_{L,1}(t,\xi)$, and $\hat{G}_{H}(t,\xi)$ are the Fourier transforms of $G_L(t,x)$, $G_{L,1}(t,x)$ and $G_{H}(t,x)$.
\end{lem}

Denote
\bma
&\hat{G}^1_{L,0}(t,\xi)=\sum_{j=\pm1}e^{\lambda_jt}\psi_j\otimes\bigg\langle \psi_j+\frac{1}{1+|\xi|^2}P_0^1\psi_j\bigg|-\sum_{j=\pm1}e^{\lambda_jt}\psi_j^{*}\otimes\langle \psi_j^{*}|,\label{GL01}\\
&\hat{G}^2_{L,0}(t,\xi)=e^{\lambda_0t}\psi_0\otimes\bigg\langle \psi_0+\frac{1}{1+|\xi|^2}P_0^1\psi_0\bigg|,\label{GL02}\\
&\hat{G}^3_{L,0}(t,\xi)=e^{\lambda_2t}\sum_{j=-1}^{3}\psi_j\otimes\bigg\langle \psi_j+\frac{1}{1+|\xi|^2}P_0^1\psi_j\bigg|-e^{\lambda_2t}\sum_{j=-1}^{1}\psi_j\otimes\bigg\langle \psi_j+\frac{1}{1+|\xi|^2}P_0^1\psi_j\bigg|\nnm\\
&\qquad\qquad\quad+e^{\lambda_2t}\sum_{j=\pm1}\psi_j^{*}\otimes\langle \psi_j^{*}|,\label{GL03}\\
&\hat{G}^4_{L,0}(t,\xi)=\sum_{j=\pm1}(e^{\lambda_jt}-e^{\lambda_2t})\psi_j^{*}\otimes\langle \psi_j^{*}|,\label{GL04}
\ema
where
$$
\psi^{\ast}_j=P_0^2\psi_j+\mathrm{i}\(L-\lambda_j(|\xi|)-\mathrm{i} P_1v\cdot\xi\)^{-1}P_1v\cdot\xi P_0^2\psi_j.
$$.
By \eqref{GL-HGL0} and \eqref{GL01}--\eqref{GL04}, we have
\bq
G_{L,0}(t,x)=\sum^{4}_{j=1}G^j_{L,0}(t,x),\label{GL06}
\eq
where
\bq
G^j_{L,0}(t,x)=\frac{1}{(2\pi)^{3/2}}\int_{\{|\xi|\leq\frac{r_0}{2}\}}e^{\mathrm{i}x\cdot\xi}\hat{G}^j_{L,0}(t,\xi)d\xi.
\eq

Then, we estimate the pointwise behavior of $G^j_{L,0}(t,x)$ $(j=1,2,3,4)$ as follows. By Lemma \ref{3dmvpbsp6} and noting that (cf. \cite{LIU-3,LIU-5})
$$
(L-\lambda_{\pm1}(|\xi|)-\i P_1v\cdot\xi)^{-1}=\mu_1(\xi)\pm|\xi|\mu_2(\xi),
$$  where $\mu_l(\xi)$, $l=1,2$ are analytic operators in $\xi$ for $|\xi|\leq r_0$, we have
\bma
&\psi_{\pm1}(\xi)=g_0\pm |\xi|g_1+\psi^*_{\pm 1},\label{gle1}\\
&\psi^*_{\pm1}(\xi)=\pm \frac1{|\xi|}(h\cdot\xi)+(\tilde{h}\cdot\xi),\label{gle2}
\ema
where $g_0,g_1$, $h=(h_1,h_2,h_3)$ and $\tilde{h}=(\tilde{h}_1,\tilde{h}_2,\tilde{h}_3)$ are analytic functions in $\xi$ for $|\xi|\leq r_0$ given by
\be \label{3dmVPB-gf-ghl}
\left\{\bal
g_0=u_0+\mu_1(\xi)P_1(v\cdot\xi)u_0+|\xi|^2\mu_2(\xi)P_1(v\cdot\xi)u_1,\\
g_1=u_1+\mu_1(\xi)P_1(v\cdot\xi)u_1+ \mu_2(\xi)P_1(v\cdot\xi)u_0,\\
h_i=b_{1,0}(\chi_i+\mu_1(\xi)P_1(v\cdot\xi)\chi_i)+\mathrm{i}|\xi|^2b_{1,1}\mu_2(\xi)P_1(v\cdot\xi)\chi_i,\\
\tilde{h}_i=\mathrm{i}b_{1,1}(\chi_i+\mu_1(\xi)P_1(v\cdot\xi)\chi_i)+ b_{1,0}\mu_2(\xi)P_1(v\cdot\xi)\chi_i,
\ea\right.
\ee
with $u_0=a_{1,0}\chi_0+c_{1,0}\chi_4$ and $u_1=\i a_{1,1}\chi_0+\i c_{1,1}\chi_4$.

By \eqref{gle1}--\eqref{gle2} and Lemma \ref{3dmvpbsp6}, we show the analytical property of $\hat{G}^{j}_{L,0}(t,\xi)~(j=1,2,3,4,5)$ as follows:
\begin{lem}\label{3dmvpbgf2}
The fluid parts $\hat{G}^j_{L;0}(t,\xi)~(j=1,2,3,4)$ defined by \eqref{GL01}--\eqref{GL04} can be written in the following forms
\bma
\hat{G}^1_{L,0}(t,\xi)&=e^{-A^2_1(|\xi|^2)t}\cos(|\xi| A^1_1(|\xi|^2)t)\mathcal{B}_1(\xi)+e^{-A^2_1(|\xi|^2)t}\frac{\sin(|\xi| A^1_1(|\xi|^2)t)}{|\xi|}\mathcal{B}_2(\xi),\label{3dmVPB-gf-Hgf1}
\\
\hat{G}^{2}_{L,0}(t,\xi)&=e^{-A^2_0(|\xi|^2)t}\mathcal{B}_3(\xi), \quad \hat{G}^{3}_{L,0}(t,\xi)=e^{-A^2_2(|\xi|^2)t}\mathcal{B}_4(\xi),\label{3dmVPB-gf-Hgf2-3}
\\
\hat{G}^{4}_{L,0}(t,\xi)&=\sum^3_{i,j=1} \(e^{-A^2_1(|\xi|^2)t}\cos(|\xi|A^1_1(|\xi|^2)t)-e^{-A^2_2(|\xi|^2)t}\)\frac{\xi_i\xi_j}{|\xi|^2}\mathcal{B}^{ij}_5(\xi)\nnm\\
&\qquad +e^{-A_1^2(|\xi|^2)t}\sum_{i,j=1}^3\xi_i\xi_j\frac{\sin(|\xi|A^1_1(|\xi|^2)t)}{|\xi|}\mathcal{B}^{ij}_6(\xi),\label{3dmVPB-gf-Hgf4}
\ema
where $\mathcal{B}_l(\xi)~(l=1,2,3,4)$ and $\mathcal{B}_k^{ij}(\xi)$ $(k=5,6, ~i,j=1,2,3)$ are analytic functions in $\xi$ for $|\xi|\leq r_0$ and satisfy
\bmas
\mathcal{B}_1(\xi)&=\frac{1}{1+|\xi|^2}\Big\{2g_0\otimes\langle a_{1,0}\chi_0|+2|\xi|^2g_1\otimes\langle\mathrm{i}a_{1,1}\chi_0|+2(\tilde{h}\cdot\xi)\otimes\langle a_{1,0}\chi_0|\\
&\qquad+2(h\cdot\xi)\otimes\langle\mathrm{i}a_{1,1}\chi_0|\Big\}+\Big\{2g_0\otimes\langle g_0+(\tilde{h}\cdot\xi)|+2|\xi|^2g_1\otimes\langle g_1|\\
&\qquad+2g_1\otimes\langle(h\cdot\xi)|+2(\tilde{h}\cdot\xi)\otimes\langle g_0|+2(h\cdot\xi)\otimes\langle g_1|\Big\},
\\
\mathcal{B}_2(\xi)&=\frac{\mathrm{i}}{1+|\xi|^2}\Big\{2g_0\otimes\langle\mathrm{i}|\xi|^2a_{1,1}\chi_0|+2|\xi|^2g_1\otimes\langle a_{1,0}\chi_0|+2(\tilde{h}\cdot\xi)\otimes\langle\mathrm{i}|\xi|^2a_{1,1}\chi_0|\\
&\qquad+2(h\cdot\xi)\otimes\langle a_{1,0}\chi_0|\Big\}+\mathrm{i}\Big\{2g_0\otimes\langle |\xi|^2g_1|+2g_0\otimes\langle(h\cdot\xi)|\\
&\qquad+2|\xi|^2g_1\otimes\langle g_0+(\tilde{h}\cdot\xi)|+2(\tilde{h}\cdot\xi)\otimes\langle |\xi|^2g_1|+2(h\cdot\xi)\otimes\langle g_0|\Big\},
\\
\mathcal{B}_3(\xi)&= \psi_0(\xi)\otimes \langle \psi_0(\xi)|,
\\
\mathcal{B}_4(\xi)&= \frac1{|\xi|^2} [2(h\cdot\xi)\otimes \langle (h\cdot\xi)|+(l\cdot\xi)\otimes \langle (l\cdot\xi)|]+2(\tilde{h}\cdot\xi)\otimes \langle (\tilde{h}\cdot\xi)|+l\otimes \langle l|,
\\
\mathcal{B}_5^{ij}(\xi)&=4(h_i\otimes\langle h_j|+|\xi|^2\tilde{h}_i\otimes\langle \tilde{h}_j|),\quad \mathcal{B}_6^{ij}(\xi)=4\mathrm{i}(h_i\otimes\langle \tilde{h}_j|+ \tilde{h}_i\otimes\langle h_j|),
\emas
where $g_0,g_1$, $h=(h_1,h_2,h_3)$ and $\tilde{h}=(\tilde{h}_1,\tilde{h}_2,\tilde{h}_3)$ are defined by \eqref{3dmVPB-gf-ghl}, and $l=(l_1,l_2,l_3)$ with
$$l_i =c_2(|\xi|^2)\chi_i+\mathrm{i}c_2(|\xi|^2)(L-\lambda_{2}(|\xi|)-\mathrm{i}P_1v\cdot\xi)^{-1}P_1(v\cdot\xi)\chi_i, \ \ i=1,2,3.$$
\end{lem}

\begin{lem}[\cite{LI-8}]\label{rbegf4}
For any positive integer $l$, we have
\bmas
&\Big|w\ast(1+t)^{-\frac{l}{2}} e^{-\frac{|x|^2}{C(1+t)}}\Big|\leq C(1+t)^{-\frac{l}{2}} e^{-\frac{(|x|-ct)^2}{3C(1+t)}} ,\\
&\Big|w_t\ast(1+t)^{-\frac{l}{2}} e^{-\frac{|x|^2}{C(1+t)}}\Big|\leq C(1+t)^{-\frac{l+1}{2}} e^{-\frac{(|x|-ct)^2}{4C(1+t)}},
\emas
where $w(t,x)$ be a function given by its three-dimensional Fourier transformation:
$$
\hat{w}=\frac{\sin(c|\xi|t)}{c|\xi|}, \quad \hat{w}_t=\cos(c|\xi|t).
$$
\end{lem}

With the help of Lemmas \ref{3dmvpbgf2}--\ref{rbegf4} and using a similar argument as Lemma 3.6 in \cite{LI-8}, we are able to establish the pointwise space-time behaviors of $G^{j}_{L,0}(t,x)~(j=1,2,3,4,5)$ as follows.
\begin{lem}\label{3dmvpbgf6}
For any given positive constant $N>\mathbf{c}$ and $\alpha\in\N^3$, there exist constants  $C,D>0$ such that for $|x|\leq Nt$,

(1) For $G^{1}_{L,0}(t,x)$, we have
\bq\label{3dmvpbgf6-1}
\left\{\bal
\|\partial_{x}^{\alpha}P_0^lG^{1}_{L,0}(t,x)\|\leq C(1+t)^{-\frac{4+|\alpha|}{2}}e^{-\frac{(|x|-\mathbf{c}t)^2}{D(1+t)}}+Ce^{-\frac{t}{D}},\quad l=1,2,3,\\
\|\partial_{x}^{\alpha}P_1G^{1}_{L,0}(t,x)\|,\|\partial_{x}^{\alpha}G^{1}_{L,0}(t,x)P_1\|\leq C(1+t)^{-\frac{5+|\alpha|}{2}}e^{-\frac{(|x|-\mathbf{c}t)^2}{D(1+t)}}+Ce^{-\frac{t}{D}},\\
\|\partial_{x}^{\alpha}P_1G^{1}_{L,0}(t,x)P_1\|\leq C(1+t)^{-\frac{6+|\alpha|}{2}}e^{-\frac{(|x|-\mathbf{c}t)^2}{D(1+t)}}+Ce^{-\frac{t}{D}}.
\ea\right.
\eq

(2) For $G^{k}_{L,0}(t,x)$ with $k=2,3$,
\bq
\left\{\bal
\|\partial_{x}^{\alpha}P_0^lG^{k}_{L,0}(t,x)\|\leq C(1+t)^{-\frac{3+|\alpha|}{2}}e^{-\frac{|x|^2}{D(1+t)}}+Ce^{-\frac{t}{D}},\quad  l=1,2,3,\\
\|\partial_{x}^{\alpha}P_1G^{k}_{L,0}(t,x)\|,\|\partial_{x}^{\alpha}G^{k}_{L,0}(t,x)P_1\|\leq C(1+t)^{-\frac{4+|\alpha|}{2}}e^{-\frac{|x|^2}{D(1+t)}}+Ce^{-\frac{t}{D}},\\
\|\partial_{x}^{\alpha}P_1G^{k}_{L,0}(t,x)P_1\|\leq C(1+t)^{-\frac{5+|\alpha|}{2}}e^{-\frac{|x|^2}{D(1+t)}}+Ce^{-\frac{t}{D}}.
\ea\right.
\eq

(3) For $G^{4}_{L,0}(t,x)$, we have
\bq
\left\{\bal
\|\partial_{x}^{\alpha}P^2_0G^{4}_{L,0}(t,x)\|\leq C(1+t)^{-\frac{3+|\alpha|}{2}}\((1+t)^{-\frac{1}{2}}e^{-\frac{(|x|-\mathbf{c}t)^2}{D(1+t)}}+e^{-\frac{|x|^2}{C(1+t)}}\)+Ce^{-\frac{t}{D}}\\
~~~~~~~~~~~~~~~~~~~~~~~~~+C(1+t)^{-\frac{3+|\alpha|}{2}}B_{\frac{3}{2}}(t,|x|)1_{\{|x|\leq\mathbf{c}t\}},\\
\|\partial_{x}^{\alpha}P_1G^{4}_{L,0}(t,x)\|,\|\partial_{x}^{\alpha}G^{4}_{L,0}(t,x)P_1\|\leq C(1+t)^{-\frac{4+|\alpha|}{2}}\((1+t)^{-\frac{1}{2}}e^{-\frac{(|x|-\mathbf{c}t)^2}{D(1+t)}}+e^{-\frac{|x|^2}{D(1+t)}}\)\\
~~~~~~~~~~~~~~~~~~~~~~~~~+Ce^{-\frac{t}{D}}+C(1+t)^{-\frac{4+|\alpha|}{2}}B_{\frac{3}{2}}(t,|x|)1_{\{|x|\leq\mathbf{c}t\}},\\
\|\partial_{x}^{\alpha}P_1G^{4}_{L,0}(t,x)P_1\|\leq C(1+t)^{-\frac{5+|\alpha|}{2}}\((1+t)^{-\frac{1}{2}}e^{-\frac{(|x|-\mathbf{c}t)^2}{D(1+t)}}+e^{-\frac{|x|^2}{D(1+t)}}\)+Ce^{-\frac{t}{D}}\\
~~~~~~~~~~~~~~~~~~~~~~~~~~~+C(1+t)^{-\frac{5+|\alpha|}{2}}B_{\frac{3}{2}}(t,|x|)1_{\{|x|\leq\mathbf{c}t\}}.
\ea\right.
\eq
\end{lem}

\subsection{Kinetic Part}
In this subsection, we extract the singular kinetic part from the Green's function $G$ and establish the pointwise estimate of the remainder part. Since $\hat{G}(t,\xi)$ does not belongs to $L^1(\mathbb{R}^3_{\xi})$, $G(t,x)$ can be decompose into the singular part and the remainder smooth part. Indeed, we construct the approximate sequences of $\hat{G}$ with faster decay rate in frequency space, which is equivalent to the higher regularity in physical space, and estimate the smooth remainder term by the weighted energy method. To begin with, we define the $k$-th degree Mixture operator $\mathbb{M}^t_k(\xi)$ by (cf. \cite{LIU-2,LIU-3})
\bq
\hat{\mathbb{M}}^t_k(\xi)=\int^t_0\int^{s_1}_0\cdot\cdot\cdot\int^{s_{k-1}}_0\hat{S}^{t-s_1}K\hat{S}^{s_{1}-s_2}\cdot\cdot\cdot K\hat{S}^{s_{k-1}-s_k}K\hat{S}^{s_k}ds_k\cdot\cdot\cdot ds_1,
\eq
where $\xi\in\mathbb{C}^3$, and $\hat{S}^t$ is a operator on $L^2(\mathbb{R}^3_v)$ defined by
$$
\hat{S}^t=e^{-(\nu(v)+\mathrm{i}v\cdot\xi)t}.
$$

Denote
$$
D_{\delta}=\{\xi\in\mathbb{C}^3\,|\,|\mathrm{Im}\xi|\leq\delta\},\quad \delta>0.
$$

\begin{lem}[\cite{LI-5}]\label{3dmvpbgf7}
For any $k\geq1$, $\hat{\mathbb{M}}^t_k(\xi)$ is analytic for $\xi\in D_{\nu_0}$ and satisfies
\bmas
&\|\hat{\mathbb{M}}^t_{2}(\xi)g_0\| \le C(1+|\xi|)^{-1}(1+t)^2e^{-\nu_0t}(\|g_0\| +\|\Tdv g_0\| ),
\\
&\|\hat{\mathbb{M}}^t_{3k}(\xi)g_0\|\leq C_k(1+|\xi|)^{-k}(1+t)^{3k}e^{-\nu_0t}\|g_0\|  ,
\emas
for any $g_0\in L^2_v$ and  two positive constants $C$ and  $C_k$, where  $\nu_0>0$ is defined by \eqref{nuv}.
\end{lem}

In terms of \eqref{3Dmvpbfb}, the system \eqref{gf-flg} can be written as
\bq\label{3dmvpbgh1}
\left\{\bal
\partial_{t}\hat{G}+\mathrm{i}v\cdot\xi \hat{G}-L\hat{G}+\frac{\mathrm{i}v\cdot\xi}{1+|\xi|^2}P_0^1\hat{G}=0,\\
\hat{G}(0,\xi)=1(\xi)I_{v}.\nnm
\ea\right.
\eq

We define the approximate sequence $\hat{J}_{k}$ for the Green's function $\hat{G}$ as follow:
\bq\label{3dmvpbgh2}
\left\{\bal
\partial_{t}\hat{J}_{0}+\mathrm{i}v\cdot\xi \hat{J}_{0}+\nu(v)\hat{J}_{0}=0,\\
(1+|\xi|^2)\hat{\Theta}_{0}=-(\hat{J}_{0},\chi_{0}),\\
\hat{J}_{0}(0,\xi)=1(\xi)I_{v},
\ea\right.
\eq
and
\bq\label{3dmvpbgh3}
\left\{\bal
\partial_{t}\hat{J}_{k}+\mathrm{i}v\cdot\xi \hat{J}_{k}+\nu(v)\hat{J}_{k}=K\hat{J}_{k-1}+\mathrm{i}v\cdot\xi \chi_{0}\hat{\Theta}_{k-1},\\
(1+|\xi|^2)\hat{\Theta}_{k}=-(\hat{J}_{k},\chi_{0}),\\
\hat{J}_{k}(0,\xi)=0, \quad k\ge 1.
\ea\right.
\eq

With the help of Lemma \ref{3dmvpbgf7}, we can show the pointwise estimates of the approximate sequence $J_{k}(t,x)$ and $\nabla_x\Theta_{3k}(t,x)$ as follows:
\begin{lem}\label{3dmvpbgf8}
For each $k\geq0$, $\hat{J}_{k}(t,\xi)$ and $\hat{\Theta}_{3k}(t,\xi)$ are analytic for $\xi\in D_{\nu_0}$, and satisfy
\bq
\|\hat{J}_{3k}(t,\xi)\|+|\xi\hat{\Theta}_{3k}(t,\xi)|\leq C_k(1+|\xi|)^{-k}e^{-\frac{\nu_{0}t}{2}},\label{gfhjk1}
\eq
where $C_k>0$ is a constant dependent of $k$. In particular, it holds for $k\geq4$ that
\bq
\|J_{3k}(t,x)\|+|\nabla_x\Theta_{3k}(t,x)|\leq C_ke^{-\frac{\nu_{0}(|x|+t)}{2}}.\label{gfhjk2}
\eq
\end{lem}
\begin{proof}First, we want to show \eqref{gfhjk1}.
For $k=0$, it follows from \eqref{3dmvpbgh1} that $\hat{J}_{0}(t,\xi)$ are analytic for $\xi\in D_{\nu_0}$, and satisfy
\bmas
&\| \hat{J}_{0}(t,\xi)\| \le \| e^{-[\nu(v)+\mathrm{i}v\cdot\xi]t}\| \le e^{-\nu_0t},
\\
&| \xi\hat{\Theta}_{0}(t,\xi)|\le  \frac{|\xi|}{1+|\xi|^2} | (\hat{J}_{0}(t,\xi),\chi_0) |\le  \frac{1}{1+|\xi|}e^{-\nu_0t}.\label{gfhjk4}
\emas
Suppose that \eqref{gfhjk1} holds for $k\leq l-1$  with $l\geq1$. Since
\bq\label{gfhjk5}
\begin{split}
&\hat{J}_{3l}(t,\xi)=\mathbb{M}^{t}_{3l}(\xi)+\sum^{3l-1}_{n=0}\int^{t}_{0}\mathbb{M}^{t-s}_{n}(\xi)\mathrm{i}v\cdot\xi \chi_{0}\hat{\Theta}_{3l-n-1}(s)ds,\\
&\hat{\Theta}_{3l}(t,\xi)=-\frac{1}{1+|\xi|^2}(\hat{J}_{3l}(t,\xi),\chi_{0})
\end{split}
\eq
with $\mathbb{M}^{t}_{0}=\hat{S}^{t}$, it follows from Lemma \ref{3dmvpbgf7} that $\hat{J}_{3l}(t,\xi)$  and $\hat{\Theta}_{3l}(t,\xi)$ are analytic for $\xi\in D_{\nu_0}$ and satisfy
\bq
\|\hat{J}_{3l}(t,\xi)\|+|\xi\hat{\Theta}_{3l}(t,\xi)|\leq C(1+|\xi|)^{-l}e^{-\frac{\nu_{0}t}{2}}.\label{gfhjk6}
\eq
Thus, \eqref{gfhjk1} holds for any $k\ge 0$.

Next, we want to show \eqref{gfhjk2}. Let $\xi=\xi^1+\mathrm{i}\xi^2$.
Since $\hat{J}_{k}(t,\xi)$ and $\hat{\Theta}_k(t,\xi)$ are analytic for $\xi\in D_{\nu_0}$,  we have by Cauchy Theorem that
\bma
&\quad \|J_{3k}(t,x)+\nabla_x\Theta_{3k}(t,x) \|\nnm\\
&=C\bigg\|\int_{\R^3}e^{\i x\cdot\xi}\[\hat{J}_{3k}(t,\xi)+\mathrm{i}\xi\hat{\Theta}_{3k}(t,\xi)\]d\xi\bigg\|\nnm\\
&=Ce^{-x\cdot\xi^2}\bigg\|\int_{\R^3}e^{\i x\cdot\xi_1}\[\hat{J}_{3k}(t,\xi_1+\mathrm{i}\xi_0)+\mathrm{i}(\xi^1+\mathrm{i}\xi^2)\hat{\Theta}_{3k}(t,\xi^1+\mathrm{i}\xi^2)\]d(\xi^1+\mathrm{i}\xi^2)\bigg\|\nnm\\
&\leq C_ke^{-x\cdot\xi^2}\int_{\R^3}\(1+|\xi^1+\mathrm{i}\xi^2|\)^{-k}e^{-\frac{\nu_{0}t}{2}}d\xi^1\nnm\\
&\leq C_ke^{-\frac{\nu_{0}(|x|+t)}{2}},\label{gfhjk7}
\ema
where $|\xi^2|\leq\nu_0$ and $k\geq4$. The proof of the lemma is completed.
\end{proof}

We define
\bma
&W_{k}(t,x)=\sum^{3k}_{i=0}J_{i}(t,x), \quad\theta_{k}(t,x)=\sum^{3k}_{i=0}\Theta_{i}(t,x),\label{gfhjk7j1}\\
&R_{k}(t,x)=(G-W_{k})(t,x),\quad\phi_{k}(t,x)=(T-\theta_{k})(t,x),\label{gfhjk8}
\ema
where $T(t,x)=(I-\Delta_x)^{-1}(G,\chi_{0})$, and $J_{k}(t,x)$ are given by \eqref{3dmvpbgh2} and \eqref{3dmvpbgh3}.

It follows from \eqref{3dmvpbgh2}--\eqref{3dmvpbgh3} and \eqref{gfhjk7j1} that $W_{k}(t,x)$  and $\theta_{k}(t,x)$ satisfies
\bq\label{gfhjk9}
\left\{\bal
\partial_{t}W_{k}+v\cdot\nabla_x W_{k}-LW_{k}-\chi_{0}v\cdot\nabla_x \theta_{k}=-KJ_{3k}-\chi_{0}v\cdot\nabla_x \Theta_{3k},\\
(I-\Delta_x)\theta_{k}=-(W_{k},\chi_{0}),\\
W_{k}(0,x)=\delta(x)I_{v},
\ea\right.
\eq
and $R_{k}(t,x)$ and $\phi_{k}(t,x)$ satisfy
\bq\label{gfhjk10}
\left\{\bal
\partial_{t}R_{k}+v\cdot\nabla_x R_{k}-LR_{k}-\chi_{0}v\cdot\nabla_x \phi_{k}=KJ_{3k}+\chi_{0}v\cdot\nabla_x \Theta_{3k},\\
(I-\Delta_x)\phi_{k}=-(R_{k},\chi_{0}),\\
R_{k}(0,x)=0.
\ea\right.
\eq

With the help of Lemma \ref{3dmvpbgf1} and Lemma \ref{3dmvpbgf8}, we can show the pointwise estimate of the term $G_H(t,x)-W_{k}(t,x)$ as follow.
\begin{lem}\label{3dmvpbgf9}
There exists a constant $C>0$ such that for any $k\geq4$,
\bq
\|G_H(t,x)-W_{k}(t,x)\|\leq Ce^{-\frac{\nu_2t}{6}},
\eq
where $\nu_2=\min\{\kappa_0,\frac{\nu_0}{2}\}$, and $W_k(t,x)$ is the singular kinetic wave defined by \eqref{gfhjk7}.
\end{lem}

\begin{proof}
From Lemma \ref{3dmvpbgf1},  \eqref{gfhjk1} and \eqref{gfhjk10}, it holds that
\bma
\|\hat{R}_{k}(t,\xi)\|_{\xi}&=\bigg\|\int^t_0\hat{G}(t-s)[K\hat{J}_{3k}(s,\xi)+\mathrm{i}\chi_0v\cdot\xi\hat{\Theta}_{3k}(s,\xi)]ds\bigg\|_{\xi}\nnm\\
&\leq C\int^t_0  (\|K\hat{J}_{3k}(s,\xi)\|_{\xi}+|\xi\hat{\Theta}_{3k}(s,\xi)|) ds\nnm\\
&\leq C\int^t_0e^{-\frac{\nu_0s}{2}}\frac{1}{(1+|\xi|)^k}ds \leq C\frac{1}{(1+|\xi|)^k}.\label{gfhjk11}
\ema
By \eqref{GL-H} and \eqref{gfhjk8}, we have
$$
\hat{G}_L(t,\xi)-\hat{R}_k(t,\xi)=\hat{W}_k(t,\xi)-\hat{G}_H(t,\xi).
$$
This together with Lemma \ref{3dmvpbgf1} and \eqref{gfhjk11} implies that
\be\label{gfhjk12}
\|\hat{W}_k(t,\xi)-\hat{G}_H(t,\xi)\|\leq \|\hat{G}_L(t,\xi)\|+\|\hat{R}_k(t,\xi)\|\leq C\frac{1}{(1+|\xi|)^k}.
\ee

By Lemma \ref{3dmvpbgf1} and Lemma \ref{3dmvpbgf8}, we have
\be\label{gfhjk13}
\|\hat{W}_k(t,\xi)\|+\|\hat{G}_H(t,\xi)\|\leq Ce^{-\nu_2t},
\ee
where $\nu_2=\min\{\kappa_0,\frac{\nu_0}{2}\}$.

Thus, it follows from \eqref{gfhjk12} and \eqref{gfhjk13} that
$$
\|\hat{G}_H(t,\xi)-\hat{W}_k(t,\xi)\|\leq Ce^{-\frac{\nu_2t}{6}}\frac{1}{(1+|\xi|)^{\frac{5k}{6}}},
$$
which gives
\bma
\|G_H(t,x)-W_k(t,x)\|&=C\bigg\|\int_{\R^3}e^{\mathrm{i}x\cdot\xi}[\hat{G}_H(t,\xi)-\hat{W}_k(t,\xi)]d\xi\bigg\|\nnm\\
&\leq C\int_{\R^3}e^{-\frac{\nu_2t}{6}}\frac{1}{(1+|\xi|)^\frac{5k}{6}}d\xi\leq Ce^{-\frac{\nu_2t}{6}},
\ema
where $k\geq4$. The proof is completed.
\end{proof}


\begin{lem}\label{3dmvpbgf10}
For any $f=f(x)\in L^2(\R^3_x)$, $0<\delta<2$ and $\Omega\in\mathbb{S}^2$, it holds that
\bma
&\int_{\R^3}|(I-\Delta_x)^{-1}f|^2e^{\delta(x\cdot\Omega)}dx\leq \frac{32}{(2-\delta)^4}\int_{\R^3}|f|^2e^{\delta(x\cdot\Omega)}dx,\label{out2}\\
&\int_{\R^3}|\nabla_x(I-\Delta_x)^{-1}f|^2e^{\delta(x\cdot\Omega)}dx\leq\frac{96}{(2-\delta)^4}\int_{\R^3}|f|^2e^{\delta(x\cdot\Omega)}dx.\label{out2j1}
\ema
\end{lem}
\begin{proof}
First, we know that the Green's function of $(I-\Delta_x)^{-1}$ is $\frac{1}{4\pi|x|}e^{-|x|}$. By H$\ddot{\mathrm{o}}$lder's inequality and Fubini's Theorem, we have
\bmas
&\quad\int_{\R^3}|(I-\Delta_x)^{-1}f|^2e^{\delta(x\cdot\Omega)}dx \nnm\\
&=\int_{\R^3}\[\int_{\R^3}\frac{1}{4\pi|x-y|}e^{-|x-y|}f(y)dy\]^2 e^{\delta(x\cdot\Omega)}dx\nnm\\
&\le \int_{\R^3}\(\int_{\R^3}\frac{1}{16\pi^2|x-y|^2}e^{-(1+\frac{\delta}2)|x-y|}e^{\delta(x-y)\cdot\Omega}dy\)\(\int_{\R^3}e^{-(1-\frac{\delta}2)|x-y|}e^{\delta(y\cdot\Omega)}|f(y)|^2dy\)dx\nnm\\
&\le \int_{\R^3}\(\int_{\R^3}\frac{1}{16\pi^2|x-y|^2}e^{-(1-\frac{\delta}2)|x-y|} dy \)\(\int_{\R^3}e^{-(1-\frac{\delta}2)|x-y|}e^{\delta(y\cdot\Omega)}|f(y)|^2dy\)dx\nnm\\
&\leq \frac1{2\pi(2-\delta)}\int_{\R^3}e^{\delta(y\cdot\Omega)}|f(y)|^2\[\int_{\R^3}e^{-(1-\frac{\delta}2)|x-y|}dx\]dy\nnm\\
&\leq \frac{32}{(2-\delta)^4}\int_{\R^3}e^{\delta(y\cdot\Omega)}|f(y)|^2dy.
\emas

Next, we have by H$\ddot{\mathrm{o}}$lder's inequality and Fubini's Theorem that
\bmas
&\quad\int_{\R^3}|\Tdx(I-\Delta_x)^{-1}f|^2e^{\delta(x\cdot\Omega)}dx \nnm\\
&=\int_{\R^3}\[\int_{\R^3}\frac{1}{4\pi}\(\frac{(x-y)}{|x-y|^3}+\frac{(x-y)}{|x-y|^2}\)e^{-|x-y|}f(y)dy\]^2 e^{\delta(x\cdot\Omega)}dx\nnm\\
&\leq \frac18\int_{\R^3}\(\int_{\R^3} \frac{(x-y)}{\pi|x-y|^3}e^{-|x-y|}f(y)dy\)^2 e^{\delta(x\cdot\Omega)}dx\nnm\\
&\quad+\frac18\int_{\R^3}\(\int_{\R^3} \frac{(x-y)}{\pi|x-y|^2}e^{-|x-y|}f(y)dy\)^2 e^{\delta(x\cdot\Omega)}dx\nnm\\
&=:I_1+I_2.
\emas
For $I_1$, we have
\bmas
I_1
&\le \frac18\int_{\R^3}\(\int_{\R^3}\frac{1}{\pi^2|x-y|^2}e^{-(1-\frac{\delta}2)|x-y|} dy \)\(\int_{\R^3}\frac{1}{|x-y|^2}e^{-(1-\frac{\delta}2)|x-y|}e^{\delta(y\cdot\Omega)}|f(y)|^2dy\)dx\nnm\\
&\leq \frac1{\pi(2-\delta)}\int_{\R^3}e^{\delta(y\cdot\Omega)}|f(y)|^2\[\int_{\R^3}\frac{1}{|x-y|^2}e^{-(1-\frac{\delta}2)|x-y|}dx\]dy\nnm\\
&\leq \frac{8}{(2-\delta)^2}\int_{\R^3}e^{\delta(y\cdot\Omega)}|f(y)|^2dy.
\emas
For $I_2$, we have
\bmas
I_2
&\le \frac18\int_{\R^3}\(\int_{\R^3}\frac{1}{\pi^2|x-y|^2}e^{-(1-\frac{\delta}2)|x-y|} dy \)\(\int_{\R^3}e^{-(1-\frac{\delta}2)|x-y|}e^{\delta(y\cdot\Omega)}|f(y)|^2dy\)dx\nnm\\
&\leq \frac1{\pi(2-\delta)}\int_{\R^3}e^{\delta(y\cdot\Omega)}|f(y)|^2\[\int_{\R^3}e^{-(1-\frac{\delta}2)|x-y|}dx\]dy\nnm\\
&\leq \frac{64}{(2-\delta)^4}\int_{\R^3}e^{\delta(y\cdot\Omega)}|f(y)|^2dy.
\emas
Note that
$$
\frac{8}{(2-\delta)^2}+\frac{64}{(2-\delta)^4}= \frac{8(2-\delta)^2+64}{(2-\delta)^4}\leq \frac{96}{(2-\delta)^4},
$$
we can obtain \eqref{out2j1}. This proves the lemma.
\end{proof}

\begin{lem}[\cite{LI-8}]\label{3dmvpbgf11}
Assume that $\hat{V}(\xi)$ is analytic for $\xi\in D_{\delta}$ with $\delta>0$ and $|\xi|^{\alpha}|\hat{V}(\xi)|\rightarrow0$ when $|\xi|\rightarrow\infty$.
For any $ \Omega\in \S^2$ and $0\le b\le \delta$, there exists a constant $C>0$ such that
\be
\int_{\R^3}e^{2b\Omega\cdot x}|\partial_x^{\alpha}V(x)|^2dx\leq C\int_{\R^3}|u+\i b\Omega|^{2\alpha}|\hat{V}(u+\i b\Omega)|^2du. \label{vvv}
\ee
\end{lem}

Next, we have the pointwise space-time estimates of the remainder term $R_k(t,x)$ and $\phi_k(t,x)$ outside  Mach number region as follows.
\begin{lem}\label{3dmvpbgf12}
Given any constant $k\geq4$, there exist constants $C,D>0$   such that for $|x|>6t$,
\bq
\|R_{k}(t,x)\|+|\nabla_x\phi_k(t,x)|\leq Ce^{-\frac{|x|+t}{D}},\label{gfout1}
\eq
where $R_k(t,x)$ and $\phi_k(t,x)$ are defined by \eqref{gfhjk8}.
\end{lem}

\begin{proof}
We use the weighted energy method. Set
\bq\label{3dmvpboutm1}
w=e^{\eps (x\cdot\Omega-Yt)},
\eq
where $0<\eps \ll1$ is a sufficient small constant, direction $\Omega\in \S^2$, and $Y>4$ are determined later. It holds that
$$
\partial_{t}w=-\eps Yw,\quad\nabla_x w=\eps\Omega w.
$$

Taking the inner product between \eqref{gfhjk10} and $R_{k}w$, and integrate it over $x$, we have
\bma
&\quad\frac12\Dt\int_{\R^3}\|R_{k}\|^{2}wdx+\frac{\eps Y}{2}\int_{\R^3}\|R_k\|^2wdx-\frac{\eps}{2} \int_{\R^3}\(\Omega\cdot vR_{k},R_{k}\)wdx \nnm\\
&\quad-\int_{\R^3}\nabla_x \phi_{k}\cdot(R_{k}, v\chi_{0})wdx-\int_{\R^3}(LR_{k},R_{k})wdx \nnm\\
&=\int_{\R^3}(KJ_{3k},R_{k})w dx+\int_{\R^3}(v\chi_{0}\cdot\nabla_x \Theta_{3k},R_{k})wdx. \label{3dmvpboutm2}
\ema

By \cite{LIU-2,LIU-3}, we have
$$
|(v_iP_0f,P_0f)|\leq \sqrt{\frac{5}{3}}(P_0f,P_0f),\quad i=1,2,3,
$$
which leads to
\bma
&\quad\eps \int_{\R^3}\(v_{i}R_{k},R_{k}\)wdx \nnm\\
&\le \eps \int_{\R^3}|(v_{i}P_0R_{k},P_0R_{k}) |wdx+2\eps \int_{\R^3}|(v_{i}P_0R_{k},P_1R_{k})| wdx+\eps \int_{\R^3}|(v_{i}P_1R_{k},P_1R_{k})|wdx \nnm\\
&\leq  \eps \frac54\sqrt{\frac{5}{3}}\int_{\R^3}\|P_0R_k\|^2wdx+C\eps \int_{\R^3}(-LR_k,R_k)wdx,\quad i=1,2,3.\label{3dmvpboutm3}
\ema

By  integration by part, we have
\bq\label{3dmvpboutm4}
-\int_{\R^3}\nabla_x \phi_{k}\cdot(R_{k}, v\chi_{0})wdx=\int_{\R^3}\phi_k(\nabla_xR_k\cdot v,\chi_0)wdx+\int_{\R^3}\phi_k(R_k,v\chi_0) \nabla_xwdx.
\eq
Taking the inner product between $\eqref{gfhjk10}_1$ and $\chi_0$, we have
\bq\label{3dmvpboutm5}
(\partial_tR_{k},\chi_0)+(\nabla_xR_{k}\cdot v,\chi_0)=(KJ_{3k},\chi_0),
\eq
which together with \eqref{3dmvpboutm4} implies that
\be
\int_{\R^3}\phi_k(\nabla_xR_k\cdot v,\chi_0)wdx=\int_{\R^3}\phi_{k}(KJ_{3k},\chi_0)wdx-\int_{\R^3}\phi_{k}(\partial_tR_{k},\chi_0)wdx .\label{3dmvpboutm6}
\ee
The term $\int_{\mathbb{R}}\phi_{k}(\partial_tR_{k},\chi_0)wdx$ can be estimated as follows. By $\eqref{gfhjk10}_2$, we have
\bma
&\quad-\int_{\R^3}\phi_{k}(\partial_tR_{k},\chi_0)wdx=\int_{\R^3}\phi_{k}(\partial_t\phi_k-\partial_t\Delta_x\phi_k)wdx\nnm\\
&=\frac{1}{2}\int_{\R^3}\partial_t|\phi_k|^2wdx+\frac12\int_{\R^3} \partial_t|\nabla_x \phi_k|^2wdx+\int_{\R}\phi_k \nabla_x\partial_t\phi_k\cdot\nabla_xwdx\nnm\\
&=\frac{1}{2}\Dt\int_{\R^3}(|\phi_k|^2+|\nabla_x \phi_k|^2)wdx+\frac{\eps Y}{2}\int_{\R^3}(|\phi_k|^2+|\nabla_x \phi_k|^2) wdx\nnm\\
&\quad+\int_{\R^3}\phi_k\nabla_x\partial_t\phi_k\cdot\nabla_xwdx. \label{3dmvpboutm7}
\ema


For the last term in the r.h.s. of \eqref{3dmvpboutm7}, by \eqref{3dmvpboutm5} and $\eqref{gfhjk10}_2$ we obtain
$$-\partial_t\nabla_x \phi_{k}+\nabla_x(I-\Delta_x)^{-1}(\nabla_xR_{k}\cdot v,\chi_0)=\nabla_x(I-\Delta_x)^{-1}(KJ_{3k},\chi_0), $$
which together with Lemma \ref{3dmvpbgf10} gives rise to
\bma
&\int_{\R^3}\phi_k\nabla_x\partial_t\phi_k\cdot\nabla_xwdx\nnm\\
=&\int_{\R^3}\phi_k\nabla_x(I-\Delta_x)^{-1}(\nabla_xR_k\cdot v,\chi_0)\cdot\nabla_xwdx-\int_{\R^3}\phi_k\nabla_x(I-\Delta_x)^{-1}(KJ_{3k},\chi_0)\cdot\nabla_xwdx\nnm\\
=&-\int_{\R^3}\nabla_x\phi_k(I-\Delta_x)^{-1}(\nabla_xR_k\cdot v,\chi_0)\nabla_xwdx-\int_{\R^3}\phi_k(I-\Delta_x)^{-1}(R_k,\chi_i)\Delta_xwdx\nnm\\
&-\int_{\R^3}\phi_k\nabla_x(I-\Delta_x)^{-1}(KJ_{3k},\chi_0)\cdot\nabla_xwdx\nnm\\
\leq&\frac{3\epsilon}{2}\int_{\R^3}(|\phi_k|^2+|\nabla_x\phi_k|^2)wdx+\frac{48\epsilon}{(2-\epsilon)^4}\int_{\R^3}(\|R_k\|^2+\|KJ_{3k}\|^2)wdx.\label{3dmvpboutm8}
\ema
By \eqref{3dmvpboutm4}--\eqref{3dmvpboutm8} and using Cauchy inequality, we have
\bma
&\quad-\int_{\R^3}\nabla_x \phi_{k}\cdot(R_{k}, v\chi_{0})wdx
\nnm\\
&\ge \frac12\Dt\int_{\R^3}(|\phi_k|^2+|\nabla_x \phi_k|^2)wdx +\frac{\eps (Y-4)}{2} \int_{\R}(|\phi_k|^2+|\dx \phi_k|^2)wdx\nnm\\
&\quad-\frac{48\epsilon}{(2-\epsilon)^4}\int_{\R^3} \|R_k\|^2wdx -\frac{C}{\eps}\int_{\R^3}\|J_{3k}\|^2 wdx. \label{3dmvpboutm9}
\ema

From  \eqref{3dmvpboutm3}--\eqref{3dmvpboutm9} and \eqref{3dmvpboutm2}, it holds that for $0<\eps\ll 1$ and $Y=5$,
\bma
&\quad\Dt\int_{\R^3}(\|R_{k}\|^{2}+|\phi_k|^2+|\nabla_x \phi_k|^2)wdx+\eps\int_{\R^3}(\|R_{k}\|^{2}+|\phi_k|^2+|\nabla_x \phi_k|^2)wdx \nnm\\
&\leq \frac{C}{\eps}\int_{\R^3}(\|J_{3k}\|^{2}+|\nabla_x \Theta_{3k}|^{2})wdx .\label{3dmvpboutm10}
\ema
Applying Gronwall's inequality to \eqref{3dmvpboutm10} and using Lemma \ref{3dmvpbgf8}, we have
\be
\int_{\R^3}(\|R_{k}\|^{2}+|\phi_k|^2+|\nabla_x \phi_k|^2)wdx\le C\intt\int_{\R^3} e^{-  \eps (t-s) } (\|J_{3k}\|^{2}+|\nabla_x \Theta_{3k}|^{2})w dx ds\le C . \label{3dmvpboutm11}
\ee

By Lemma \ref{3dmvpbgf8} and Lemma \ref{3dmvpbgf11}, it holds that for $k\geq4$ and $0\leq2\epsilon\leq\ \nu_0 $,
\bmas
\int_{\R^3}\|D_x^2 J_{3k}\|^{2}e^{2\epsilon\Omega\cdot x}dx&\leq C\int_{\R^3}|\xi+ i\eps\Omega|^4\|\hat{J}_{3k}(t,\xi+ i\eps\Omega)\|^{2}d\xi\\
&\leq C\int_{\R^3}e^{-\frac{\nu_0t}{4}}(1+|\xi+ i\eps\Omega|)^{4-2k}d\xi\leq Ce^{-\frac{\nu_0t}{4}},
\emas
which together with Gronwall's inequality and \eqref{3dmvpboutm11} implies
\bq\label{3dmvpboutm12}
\begin{split}
\int_{\R^3}(\|\nabla^2_x R_{k}\|^{2}+|\nabla^3_x \phi_k|^2)wdx\leq C.
\end{split}
\eq
Thus, by \eqref{3dmvpboutm11}, \eqref{3dmvpboutm12} and Sobolev's embedding theorem, we have
\bq\label{3dmvpboutm13}
e^{\frac{\eps (x\cdot\Omega-Yt)}{2}}\(\|R_k(t,x)\|+|\nabla_x\phi_k(t,x)|\)\leq C.
\eq
It holds that for $|x|>6t$ and $\Omega=\frac{x}{|x|}$,
\bma
x\cdot\Omega-Yt&=|x|-5t=\frac{|x|}{7}+\frac{6|x|}{7}-\frac{35t}{7}\nnm\\
&>\frac{|x|}{7}+\frac{t}{7}. \label{3dmvpboutm14}
\ema
 By \eqref{3dmvpboutm13} and \eqref{3dmvpboutm14}, we prove \eqref{gfout1}.
\end{proof}


With the help of Lemma \ref{3dmvpbgf1}, Lemma \ref{3dmvpbgf6}, Lemma \ref{3dmvpbgf9} and Lemma \ref{3dmvpbgf12}, we can prove Theorem \ref{3Dmvpbth1} as follows.

\bigskip

\textbf{Proof of Theorem \ref{3Dmvpbth1}} By \eqref{GL-H}--\eqref{GL-H2}, we can decompose $G(t,x)$ into
\bma
G(t,x)&=[G(t,x)-W_4(t,x)]1_{\{|x|\leq 6t\}}+[G(t,x)-W_4(t,x)]1_{\{|x|> 6t\}}+W_4(t,x)\nnm\\
&=G_{L,0}(t,x)1_{\{|x|\leq 6t\}}+[G_{L,1}(t,x)+G_{H}(t,x)-W_4(t,x)]1_{\{|x|\leq 6t\}}\nnm\\
&\quad+R_4(t,x) 1_{\{|x|> 6t\}}+W_4(t,x).\label{gfth1-1}
\ema
Thus
\bq
G(t,x)=G_1(t,x)+G_2(t,x)+W_4(x,t),\label{gfth1-2}
\eq
where
\bma
G_1(t,x)&=G_{L,0}(t,x)1_{\{|x|\leq 6t\}},\label{gfth1-3}\\
G_2(t,x)&=[G_{L,1}(t,x)+G_{H}(t,x)-W_4(t,x)]1_{\{|x|\leq 6t\}}+R_4(t,x) 1_{\{|x|> 6t\}}.\label{gfth1-4}
\ema
By Lemma \ref{3dmvpbgf6}, \eqref{3dmvpbthm111}--\eqref{3dmvpbthm114} hold. By Lemma \ref{3dmvpbgf1}, we have for $|x|\leq 6t$,
\bq\label{gfth1-5}
\|\dxa G_{L,1}(t,x)\|\leq Ce^{-\kappa_0t}\le Ce^{-\frac{|x|+t}{D}}.
\eq
By Lemma \ref{3dmvpbgf9}, we have for $|x|\leq 6t$,
\bq\label{gfth1-6}
\|G_{H}(t,x)-W_{4}(t,x)\|\leq Ce^{-\frac{|x|+t}{D}}.
\eq
For $|x|>6t$, it follows from Lemma \ref{3dmvpbgf12} that
\bq\label{gfth1-7}
\|R_4(t,x)\|\leq Ce^{-\frac{|x|+t}{D}}.
\eq
By combining \eqref{gfth1-6}--\eqref{gfth1-7}, we can obtain
$$
\|G_2(t,x)\|\leq Ce^{-\frac{|x|+t}{D}},
$$
which proves  \eqref{3DmvpbThm1-2}. This completes the proof.

\section{The Nonlinear system}\setcounter{equation}{0}
\label{sect4}

\subsection{Energy estimate}
In this subsection, we establish the energy estimate of three-dimensional mVPB system. First of all, let $N$ be a positive integer, and set
\bma
E_{N,k}(f)&=\sum_{|\alpha|+|\beta|\leq N}\|\langle v\rangle^k\dx^{\alpha}\partial_{v}^{\beta}f\|^2_{L^2_{x,v}}+\sum_{|\alpha|\leq N}\|\dx^{\alpha}\Phi\|^2_{H^1_x},\label{energy1}\\
H_{N,k}(f)&=\sum_{|\alpha|+|\beta|\leq N}\|\langle v\rangle^k\dx^{\alpha}\partial_{v}^{\beta}P_1f\|^2_{L^2_{x,v}}+\sum_{|\alpha|\leq N-1}(\|\dxa \nabla_x P_0f\|^2_{L^2_{x,v}}+\|\dxa \nabla_x \Phi\|^2_{H^1_x}),\label{energy2}\\
D_{N,k}(f)&=\sum_{|\alpha|+|\beta|\leq N}\|\langle v\rangle^{k+\frac{1}{2}}\dx^{\alpha}\partial_{v}^{\beta}P_1f\|^2_{L^2_{x,v}}+\sum_{|\alpha|\leq N-1}(\|\dxa \nabla_x P_0f\|^2_{L^2_{x,v}}+\|\dxa \nabla_x \Phi\|^2_{H^1_x}),\label{energy3}
\ema
where $\alpha,\beta\in\N^3$ and $k$ are nonnegative integers. For the sake of brevity, we write $E_N(f)=E_{N,0}(f)$, $H_N(f)=H_{N,0}(f)$, and $D_N(f)=D_{N,0}(f)$ for $k=0$.

\begin{lem}[\cite{LI-2}]\label{3dmvpbpw1}
For $N\geq2$, there are two equivalent energy functionals $\mathcal{E}_{N}(\cdot)\sim E_{N}(\cdot)$ and $\mathcal{H}_{N}(\cdot)\sim H_{N}(\cdot)$ such that if $E_N(f)$ is sufficiently small, then the solution $f(t,x,v)$ to the three-dimensional mVPB system \eqref{3DmVPB8}--\eqref{3DmVPB9} satisfies
\bmas
&\Dt\mathcal{E}_{N}(f)(t)+ D_{N}(f)(t)\leq 0,\\
&\Dt\mathcal{H}_{N}(f)(t)+ D_{N}(f)(t)\leq C\|\nabla_x P_0f\|^2_{L^2_{x,v}}.
\emas
\end{lem}

\begin{lem}[\cite{LI-2}]\label{3dmvpbpw2}
For $N\geq2$ and $k\ge 1$, there are two equivalent energy functionals $\mathcal{E}_{N,k}(\cdot)\sim E_{N,k}(\cdot)$ and $\mathcal{H}_{N,k}(\cdot)\sim H_{N,k}(\cdot)$ such that if $E_{N,k}(f)$ is sufficiently small, then the solution $f(t,x,v)$ to the three-dimensional mVPB system \eqref{3DmVPB8}--\eqref{3DmVPB9} satisfies
\bma
&\Dt\mathcal{E}_{N,k}(f)(t)+ D_{N,k}(f)(t)\leq 0,\label{ENI1}\\
&\Dt\mathcal{H}_{N,k}(f)(t)+ D_{N,k}(f)(t)\leq C\|\nabla_x P_0f\|^2_{L^2_{x,v}}.\label{ENI2}
\ema
\end{lem}

\subsection{The pointwise estimate}
In this subsection, we prove Theorem \ref{3Dmvpbth2} on the pointwise behaviors of the global solution to the nonlinear mVPB equation with the help of the estimates of the Green's function given in Section \ref{sect3}.

\textbf{Proof of Theorem 1.2.} Let $f$ be a solution to the IVP problem \eqref{3DmVPB8}--\eqref{3DmVPB10} for $t>0$. We can represent this solution as
\bq\label{rbethpr2-1}
f(t,x)=G(t)\ast f_0+\int^t_0G(t-s)\ast \Lambda(s)ds+\int^t_0G(t-s)\ast \Gamma(f,f)ds,
\eq
where $\Lambda$ is the nonlinear term containing the electric potential $\Phi $ given by
\be\label{rbethpr2-1j1}
\left\{\bal
\Lambda=\Lambda_{1}+\Lambda_{2},\\
\Lambda_{1}=\frac{1}{2}(v\cdot\nabla_x \Phi)f-\nabla_x \Phi\cdot\nabla_{v}f,\\
\Lambda_{2}=\sqrt{M}v\cdot\nabla_x (I-\Delta_x)^{-1}(e^{-\Phi}+\Phi-1).
\ea\right.
\ee

Define
\bq
\begin{split}
Q(t)=\sup_{x\in\mathbb{R}^3,0\leq s\leq t}&\bigg\{\big\|\partial_{x}^{\alpha} f\|_{L^{\infty}_{v,3}}(1+t)^{\frac{|\alpha|}{2}}\Psi(s,x;D_2)^{-1}+\big(\|\nabla_v f\|_{L^{\infty}_{v,2}}+|\Phi|\big)\Psi(s,x;D_2)^{-1}\\
&+\big(|\nabla_{x}\Phi|+|D^2_x\Phi|+|\partial_t\Phi|\big)(1+t)^{\frac{1}{2}}\Psi(s,x;D_2)^{-1}+(1+s)^{\frac{5}{4}}\sqrt{H_{9,2}(f)}(s)\bigg\},\nonumber
\end{split}
\eq
where
\bmas
\Psi(s,x;D)&=(1+s)^{-\frac{3}{2}}\(e^{-\frac{|x|^2}{D(1+s)}}+B_{\frac{3}{2}}(s,|x|)1_{\{|x|\leq \mathbf{c} s\}}\)+e^{-\frac{|x|+s}{D}}\\
&\quad +(1+s)^{-2}\(e^{-\frac{(|x|-\mathbf{c} s)^2}{D(1+s)}}+B_{1}(s,|x|-\mathbf{c} s)1_{\{|x|\leq \mathbf{c} s\}}\).
\emas

From \cite{LI-5}, it holds that for any $\beta\in \N^3$ and $\gamma\ge 0$,
\bq\label{rbethpr2-2}
\|\partial_v^{\beta}\Gamma(f,g)\|_{L^{\infty}_{v,\gamma-1}}\leq C\sum_{|\beta_1|+|\beta_2|\leq|\beta|}\|\partial_v^{\beta_1}f\|_{L^{\infty}_{v,\gamma}}\|\partial_v^{\beta_2}g\|_{L^{\infty}_{v,\gamma}}.
\eq
By \eqref{rbethpr2-2}, we obtain that  for $|\alpha|=0,1,$
\be
\|\dxa \Gamma(f,f)\|_{L^{\infty}_{v,2}}\leq C\sum_{\alpha'\leq\alpha}\|\dx^{\alpha'}f\|_{L^{\infty}_{v,3}}\|\dx^{\alpha-\alpha'}f\|_{L^{\infty}_{v,3}}\leq CQ^2(t)(1+t)^{-\frac{|\alpha|}{2}}\Psi(s,x;D_2)^2. \label{3dmvpbthpr2-3}
\ee

By \eqref{rbethpr2-1j1}, we have
\bma
\|\partial_{x}^{\alpha} \Lambda_{1}(s,x)\|_{L^{\infty}_{v,2}}
&\leq C\sum_{\alpha'\le \alpha}\(|\partial_{x}^{\alpha'}\nabla_{x} \Phi| \|\partial_{x}^{\alpha-\alpha'}f\|_{L^{\infty}_{v,3}}+|\partial_{x}^{\alpha'}\nabla_{x} \Phi| \|\partial_{x}^{\alpha-\alpha'}\nabla_{v}f\|_{L^{\infty}_{v,2}}\)\nnm\\
&\leq CQ^2(t)
\left\{\bal
(1+s)^{-\frac{1}{2}}\Psi(s,x;D_2)^2,& |\alpha|=0,\\
(1+s)^{-\frac{7}{6}}\Psi(s,x;D_2)^{\frac{16}{9}},& |\alpha|=1,
\ea\right. \label{MH1}
\ema
where we have used (Gagliardo-Nirenberg interpolation inequality)
\bma
\|\partial_{x_i} \partial_{v}f(s,x)\|_{L^{\infty}_{v,2}}&\leq C\|\partial_{x_i} \partial_v^6(\langle v\rangle^2f)\|^{\frac{2}{9}}\|\partial_{x_i} (\langle v\rangle^2f)\|^{\frac{7}{9}}_{L^{\infty}_v}+C\|\partial_{x_i} (\langle v\rangle f)\|_{L^{\infty}_v}\nnm\\
&\leq CH_{9,2}(f)^{\frac19}\|\partial_{x_i} f\|^{\frac{7}{9}}_{L^{\infty}_{v,2}}+C\|\partial_{x_i} f\|_{L^{\infty}_{v,1}}\nnm\\
&  \leq CQ(t)(1+s)^{-\frac{2}{3}}\Psi(s,x;D_2)^{\frac{7}{9}}.\label{3dmvpb-non-pxvf}
\ema

By \eqref{rbethpr2-1}, we have
\bma\label{nldxaf}
\dxa f&=\dxa G(t)\ast f_0+\int^t_0\dxa G(t-s)\ast \Lambda_{1}(s)ds+\int^t_0\dxa G(t-s)\ast \Lambda_{2}(s)ds\nnm\\
&\quad+\int^t_0\dxa G(t-s)\ast \Gamma(f,f)ds\nnm\\
&=:I_1+I_2+I_3+I_4.
\ema
We estimate $I_j$, $j=1,2,3,4$ as follows. By Theorem \ref{3Dmvpbth1}, we decompose $I_1$ into
\bma
I_1&=\dxa G_1(t)\ast f_0+G_2(t)\ast \dxa f_0+W_4(t)\ast \dxa f_0\nnm\\
&=:I_1^1+I_1^2+I_1^3\nnm.
\ema
For $I_1^1$, it follows from Theorem \ref{3Dmvpbth1} and Lemma \ref{3dmvpbpw4} that
\bma
\|I_1^1\|&\leq C\delta_0\(\int_{\R^3}(1+t)^{-\frac{3+|\alpha|}{2}}e^{-\frac{|x-y|^2}{D(1+t)}}e^{-\frac{|y|}{D_1}}dy+\int_{\R^3}(1+t)^{-\frac{4+|\alpha|}{2}}e^{-\frac{(|x-y|-\mathbf{c}t)^2}{D(1+t)}}e^{-\frac{|y|}{D_1}}dy\)\nnm\\
&\quad+C\delta_0\(\int_{\{|x-y|\leq\mathbf{c}t\}}(1+t)^{-\frac{3+|\alpha|}{2}}B_{\frac{3}{2}}(t,|x-y|)e^{-\frac{|y|}{D_1}}dy+\int_{\R^3}e^{-\frac{|x-y|+t}{D}}e^{-\frac{|y|}{D_1}}dy\)\nnm\\
&\leq C\delta_0(1+t)^{-\frac{|\alpha|}{2}}\Psi\(t,x;\frac{2D_2}{3}\).  \label{3dmvpbthpr2-5}
\ema
For $I_1^2$ and $I_1^3$, we can obtain by Theorem \ref{3Dmvpbth1} and Lemma \ref{3dmvpbpw10} that
\be
 \|I_1^2\|+\|I_1^3\|\leq Ce^{-\frac{3(|x|+t)}{2D_2}}.\label{3dmvpbthpr2-6}
\ee 
Thus, it follows from \eqref{3dmvpbthpr2-5}--\eqref{3dmvpbthpr2-6} that
\be
\|I_1\|\leq C\delta_0(1+t)^{-\frac{|\alpha|}{2}}\Psi\(t,x;\frac{2D_2}{3}\).\label{3dmvpbpwnl1}
\ee

By Theorem \ref{3Dmvpbth1} and noting that $P_0\Gamma(f,f)=0$, we decompose $I_4$ into
\bma
I_4&=\int^t_0\dxa G_1(t-s)P_1\ast \Gamma(f,f)ds+\int^t_0G_2(t-s)\ast \dxa \Gamma(f,f)ds\nnm\\
&\quad+\int^t_0W_4(t-s)\ast \dxa \Gamma(f,f)ds\nnm\\
&=:I_4^1+I_4^2+I_4^3.\nnm
\ema
For $I_3^1$, we obtain by Theorem \ref{3Dmvpbth1}, Lemmas \ref{3dmvpbpw5}--\ref{3dmvpbpw9} and \eqref{3dmvpbthpr2-3} that
\bma
\|I_4^1\|&\leq\bigg\|\int^{t}_0\partial^{\alpha}_{x} G_1(t-s)P_1\ast \Gamma(f,f)ds\bigg\|\nnm\\
&\leq CQ^2(t)\int^{t}_0\int_{\R^3}\bigg\{(1+t-s)^{-\frac{4+|\alpha|}{2}}\(e^{-\frac{|x-y|^2}{D(1+t-s)}}+(1+t-s)^{-\frac{1}{2}}e^{-\frac{(|x-y|-\mathbf{c}(t-s))^2}{D(1+t-s)}}\)\nnm\\
&\qquad +(1+t-s)^{-\frac{4+|\alpha|}{2}}B_{\frac{3}{2}}(t-s,|x-y|)1_{\{|x-y|\leq\mathbf{c}(t-s)\}}+e^{-\frac{|x-y|+t-s}{D}}\bigg\}\nnm\\
&\qquad \times\bigg\{(1+s)^{-3}\(e^{-\frac{2|y|^2}{D_2(1+s)}}+B_{3}(s,|y|)1_{\{|y|\leq \mathbf{c} s\}}\)+e^{-\frac{2(|y|+s)}{D_2}}\nnm\\
&\qquad +(1+s)^{-4}\(e^{-\frac{2(|y|-\mathbf{c} s)^2}{D_2(1+s)}}+B_{2}(s,|y|-\mathbf{c} s)1_{\{|y|\leq \mathbf{c} s\}}\)\bigg\}dyds\nnm\\
&\leq CQ^2(t)(1+t)^{-\frac{|\alpha|}{2}}\Psi\(t,x;\frac{2D_2}{3}\). \label{3dmvpbthg-1}
\ema

For $I_4^2$ and $I_4^3$, we have by \eqref{3dmvpbthpr2-3}, Lemma \ref{3dmvpbpw9} and Lemma \ref{3dmvpbpw11} that
\be
 \|I_4^2\|+\|I_4^3\|\leq CQ^2(t)(1+t)^{-\frac{|\alpha|}{2}}\Psi\(t,x;\frac{2D_2}{3}\).\label{3dmvpbthg-2}
\ee 
Thus, it follows from \eqref{3dmvpbthg-1}--\eqref{3dmvpbthg-2} that
\be
\|I_4\|\leq CQ^2(t)(1+t)^{-\frac{|\alpha|}{2}}\Psi\(t,x;\frac{2D_2}{3}\).\label{3dmvpbthg-3}
\ee

Now we deal with $I_2$. Decompose
\bma
I_2&=\int^t_0\dxa G_1(t-s)\ast P_0\Lambda_{1}(s)ds+\int^t_0\dxa G_1(t-s)\ast P_1\Lambda_{1}(s)ds\nnm\\
&\quad+\int^t_0G_2(t-s)\ast \dxa \Lambda_{1}(s)ds+\int^t_0W_4(t-s)\ast \dxa \Lambda_{1}(s)ds\nnm\\
&=:I_2^1+I_2^2+I_2^3+I_2^4.\label{3dmvpbthg-4}
\ema
Note that
$$
P_0\Lambda_{1}=n\nabla_x\Phi\cdot v\chi_0+\sqrt{\frac{2}{3}}\nabla_x\Phi\cdot m\chi_4,
$$
where $n=(f,\chi_0)$, $m=(m_1,m_2,m_3)$ and $m_i=(f,\chi_i)$ with $i=1,2,3$. Thus, we can obtain
\bma\label{pw-hbad}
\|I_2^1\|&\le \bigg\|\int^t_0\dxa G_1(t-s)\ast \(n\nabla_x\Phi\cdot v\chi_0\)ds\bigg\|+\sqrt{\frac{2}{3}}\bigg\|\int^t_0\dxa G_1(t-s)\ast \(\nabla_x\Phi\cdot m\chi_4\)ds\bigg\|\nnm\\
&\leq CQ^2(t)\int_0^t\int_{\R^3}(1+t-s)^{-\frac{3+|\alpha|}{2}}e^{-\frac{|x-y|^2}{D(1+t-s)}}(1+s)^{-\frac{9}{2}}e^{-\frac{2(|y|-\mathbf{c} s)^2}{D_2(1+s)}}dyds+\cdot\cdot\cdot\nnm\\
&\leq CQ^2(t)\((1+t)^{-\frac{3+|\alpha|}{2}}e^{-\frac{3(|x|-\mathbf{c} t)^2}{2D_2(1+t)}}+\cdot\cdot\cdot\).
\ema
It is evident that the results from our direct calculations cannot close  the priori assumption.

Next, we will tackle this problem. Since
\bq\label{ndpm}
n=\Delta_x\Phi+(e^{-\Phi}-1),\quad \div_xm=-\partial_tn,
\eq
it follows that
\bmas
&n\nabla_x\Phi
=-\frac{1}{2}\nabla_x|\nabla_x\Phi|^2+\div_x(\nabla_x\Phi\otimes \nabla_x\Phi )-\nabla_x(e^{-\Phi}+\Phi-1),\\
&\nabla_x\Phi\cdot m
=\div_x(\Phi m)+\partial_t\div_x(\Phi\nabla_x\Phi)-\div_x(\partial_t\Phi\nabla_x\Phi)-\frac{1}{2}\partial_t|\nabla_x\Phi|^2\\
&\qquad\qquad\qquad+\partial_t(e^{-\Phi}+\Phi-1)+\partial_t[\Phi(e^{-\Phi}-1)].
\emas
Thus, we rewrite $P_0\Lambda_1$ as
\bq\label{rw-p0h}
P_0\Lambda_1=\nabla_x\cdot(H_1+H_2)+\partial_{t}(\nabla_x\cdot H_3)+\partial_{t}H_4,
\eq
where
\bmas
H_1&=-\frac{1}{2}|\nabla_x\Phi|^2v\chi_0+\nabla_x\Phi\(\nabla_x\Phi\cdot v\chi_0\)-\sqrt{\frac{2}{3}}\partial_t\Phi\nabla_x\Phi\chi_4,\\
H_2&= -(e^{-\Phi}+\Phi-1)v\chi_0+\Phi m\chi_4,\quad H_3=\Phi\nabla_x\Phi\chi_4,\\
H_4&=\frac{1}{2}\sqrt{\frac{2}{3}}|\nabla_x\Phi|^2\chi_4-\sqrt{\frac{2}{3}}\[(e^{-\Phi}+\Phi-1)+\Phi(e^{-\Phi}-1)\]\chi_4.
\emas
It's easy to verify that
\be\label{h123}
\left\{\bal
\|H_3\|\leq CQ^2(t)(1+s)^{-\frac{1}{2}}\Psi(s,x;D_2)^2,\\
\|H_1\|,\|H_2\|, \|H_4\|\leq CQ^2(t)\Psi(s,x;D_2)^2.
\ea\right.
\ee
Then
\bma
I^1_2&=\int_0^t\nabla_x\partial_x^{\alpha}G_1(t-s)\ast \(H_1+H_2\)ds+\int_0^t\partial_t\nabla_x\partial_x^{\alpha}G_1(t-s)\ast H_3ds\nnm\\
&\quad+\int_0^t\partial_{t}\partial_x^{\alpha}G_1(t-s)\ast H_4ds+\nabla_x\partial_x^{\alpha}G_1(t)\ast H_3(0)+\partial_x^{\alpha}G_1(t)\ast H_4(0),\label{3dmvpbthg-5}
\ema
where we have used that $G_1(0,x)=0$ because $G_1=G_{L,0}1_{\{|x|\leq6t\}}$ is supported in $\{|x|\leq6t\}$. By \eqref{GL-HGL0} and \eqref{gfth1-3}, it follows that
$$
\partial_tG_1(t,x)
=\sum^{4}_{j=1}\frac{1}{(2\pi)^{\frac32}}\int_{\{|\xi|\leq\frac{r_0}{2}\}}e^{\mathrm{i}x\cdot\xi}\partial_t\hat{G}^j_{L,0}(t,\xi)d\xi,
$$
where
\bmas
\partial_t\hat{G}^1_{L,0}(t,\xi)&=e^{-A^2_1(|\xi|^2)t}\cos(|\xi| A^1_1(|\xi|^2)t)\(A^1_1(|\xi|^2)\mathcal{B}_2(\xi)-A^2_1(|\xi|^2)\mathcal{B}_1(\xi)\)\nnm\\
&\quad-e^{-A^2_1(|\xi|^2)t}\frac{\sin(|\xi| A^1_1(|\xi|^2)t))}{|\xi|}\(|\xi|^2 A^1_1(|\xi|^2)\mathcal{B}_1(\xi)+A^2_1(|\xi|^2)\mathcal{B}_2(\xi)\),\\
\partial_t\hat{G}^{2}_{L,0}(t,\xi)&=-e^{-A^2_0(|\xi|^2)t}A^2_0(|\xi|^2)\mathcal{B}_3(\xi),\quad \partial_t\hat{G}^{3}_{L,0}(t,\xi)=-e^{-A^2_2(|\xi|^2)t}A^2_2(|\xi|^2)\mathcal{B}_4(\xi),\\
\partial_t\hat{G}^{4}_{L,0}(t,\xi)&=e^{-A^2_1(|\xi|^2)t}\cos(|\xi|A^1_1(|\xi|^2)t)\sum_{i,j=1}^3\xi_i\xi_j\(A^1_1(|\xi|^2)\mathcal{B}^{ij}_6(\xi)-\frac{A^2_1(|\xi|^2) }{|\xi|^2}\mathcal{B}^{ij}_5(\xi)\)\nnm\\
&\quad+e^{-A^2_1(|\xi|^2)t}\frac{\sin(|\xi|A^1_1(|\xi|^2)t)}{|\xi|}\sum_{i,j=1}^3\xi_i\xi_j\(-A_1^2(|\xi|^2) \mathcal{B}^{ij}_6(\xi)-A^1_1(|\xi|^2) \mathcal{B}^{ij}_5(\xi)\)\nnm\\
&\quad+e^{-A^2_2(|\xi|^2)t}\sum^3_{i,j=1} \frac{\xi_i\xi_j}{|\xi|^2}A^2_2(|\xi|^2)\mathcal{B}^{ij}_5(\xi)
\emas
with $\mathcal{B}_l(\xi)$ $(l=1,2,3,4)$ and $\mathcal{B}^{ij}_k(\xi) $ $(k=5,6, ~i,j=1,2,3)$ given in Lemma \ref{3dmvpbgf2}. Repeating the similar arguments as Lemma \ref{3dmvpbgf6}, we have for $\alpha\in \N^3$ that
\be
\|\partial_t\partial_x^{\alpha}G_1(t,x)\| \leq C(1+t)^{-\frac{4+|\alpha|}{2}}\(e^{-\frac{|x|^2}{D(1+t)}} +(1+t)^{-\frac{1}{2}}e^{-\frac{(|x|-\mathbf{c}t)^2}{D(1+t)}}\)+Ce^{-\frac{|x|+t}{D}}.\label{ptg1ne}
\ee

By Lemma \ref{3dmvpbpw5}--\ref{3dmvpbpw9} and \eqref{h123}, it holds that
\bma
&\quad\bigg\|\int^{t}_0\nabla_x\partial_{x}^{\alpha} G_1(t-s)\ast \(H_1+H_2\)ds\bigg\|+\bigg\|\int_0^t\partial_t\nabla_x\partial_{x}^{\alpha}G_1(t-s)\ast H_3ds\bigg\|\nnm\\
&\quad+\bigg\|\int_0^t\partial_{t}\partial_{x}^{\alpha}G_1(t-s)\ast H_4ds\bigg\|\nnm\\
&\leq CQ^2(t)\int^t_0\int_{\R^3}\bigg\{(1+t-s)^{-\frac{4+|\alpha|}{2}}\(e^{-\frac{|x-y|^2}{D(1+t-s)}}+(1+t-s)^{-\frac{1}{2}}e^{-\frac{(|x-y|-\mathbf{c}(t-s))^2}{D(1+t-s)}}\)\nnm\\
&\qquad +(1+t-s)^{-\frac{4+|\alpha|}{2}}B_{\frac{3}{2}}(t-s,|x-y|)1_{\{|x-y|\leq\mathbf{c}(t-s)\}}+e^{-\frac{|x-y|+t-s}{D}}\bigg\}\nnm\\
&\qquad \times\bigg\{(1+s)^{-3}\(e^{-\frac{2|y|^2}{D_2(1+s)}}+B_{3}(s,|y|)1_{\{|y|\leq \mathbf{c} s\}}\)+e^{-\frac{2(|y|+s)}{D_2}}\nnm\\
&\qquad +(1+s)^{-4}\(e^{-\frac{2(|y|-\mathbf{c} s)^2}{D_2(1+s)}}+B_{2}(s,|y|-\mathbf{c} s)1_{\{|y|\leq \mathbf{c} s\}}\)\bigg\}dyds\nnm\\
&\leq CQ^2(t)(1+t)^{-\frac{|\alpha|}{2}}\Psi\(t,x;\frac{2D_2}{3}\).\label{3dmvpbthg-6}
\ema

By Lemma \ref{3dmvpbpw4} and \eqref{3dmvpbthpr2-5}, we have
\be
 \|\nabla_x\partial_x^{\alpha}G_1(t)\ast H_3(0) \|+ \|\partial_x^{\alpha}G_1(t)\ast H_4(0) \|\leq C\delta_0(1+t)^{-\frac{|\alpha|}{2}}\Psi\(t,x;\frac{2D_2}{3}\),\label{3dmvpbthg-11}
\ee
where we have used
$$
\|H_3(0)\|+ \|H_4(0)\|\leq C\delta_0e^{-\frac{|x|}{D_1}}.
$$
By combining \eqref{3dmvpbthg-5}--\eqref{3dmvpbthg-11}, it follows that
\be
\|I_2^1\|\leq C(\delta_0+Q^2(t))(1+t)^{-\frac{|\alpha|}{2}}\Psi\(t,x;\frac{2D_2}{3}\).\label{3dmvpbthg-12}
\ee

For $I_2^2$, by \eqref{MH1}, Lemma \ref{3dmvpbpw5}--\ref{3dmvpbpw9} and Lemma \ref{3dmvpbpw11}, it follows that
\bma
\|I_2^2\|&=\bigg\|\int^{t}_0\partial^{\alpha}_{x} G_1(t-s)\ast P_1\Lambda_{1}(s)ds\bigg\|\nnm\\
&\leq CQ^2(t)\int^{t}_0\int_{\R^3}\bigg\{(1+t-s)^{-\frac{4+|\alpha|}{2}}\(e^{-\frac{|x-y|^2}{D(1+t-s)}}+(1+t-s)^{-\frac{1}{2}}e^{-\frac{(|x-y|-\mathbf{c}(t-s))^2}{D(1+t-s)}}\)\nnm\\
&\qquad +(1+t-s)^{-\frac{4+|\alpha|}{2}}B_{\frac{3}{2}}(t-s,|x-y|)1_{\{|x-y|\leq\mathbf{c}(t-s)\}}+e^{-\frac{|x-y|+t-s}{D}}\bigg\}\nnm\\
&\qquad \times\bigg\{(1+s)^{-\frac{7}{2}}\(e^{-\frac{2|y|^2}{D_2(1+s)}}+B_{3}(s,|y|)1_{\{|y|\leq \mathbf{c} s\}}\)+e^{-\frac{2(|y|+s)}{D_2}}\nnm\\
&\qquad +(1+s)^{-\frac{9}{2}}\(e^{-\frac{2(|y|-\mathbf{c} s)^2}{D_2(1+s)}}+B_{2}(s,|y|-\mathbf{c} s)1_{\{|y|\leq \mathbf{c} s\}}\)\bigg\}dyds\nnm\\
&\leq CQ^2(t)(1+t)^{-\frac{|\alpha|}{2}}\Psi\(t,x;\frac{2D_2}{3}\). \label{3dmvpbthg-i22}
\ema

For $I_2^3$ and $I_2^4$, by \eqref{MH1}, Lemma \ref{3dmvpbpw9} and Lemma \ref{3dmvpbpw11}, it follows that
\be
\|I_2^3\|, \|I_2^4\|\leq CQ^2(t)(1+t)^{-\frac{|\alpha|}{2}}\Psi\(t,x;\frac{2D_2}{3}\),\label{3dmvpbthg-13}
\ee
Thus, it follows from \eqref{3dmvpbthg-12}--\eqref{3dmvpbthg-13} that
\be
\|I_2\|\leq CQ^2(t)(1+t)^{-\frac{|\alpha|}{2}}\Psi\(t,x;\frac{2D_2}{3}\).\label{3dmvpbpwn-l4}
\ee

Finally, we deal with $I_3$. By \eqref{rbethpr2-1j1} and Theorem \ref{3Dmvpbth1}, we have
\bma
I_3&=\int^t_0(I-\Delta_x)^{-1}\dxa G_1(t-s)\ast [\sqrt{M}v\cdot\nabla_x (e^{-\Phi}+\Phi-1)] ds\nnm\\
&\quad+\int^t_0 G_2(t-s)\ast \dxa [\sqrt{M}v\cdot\nabla_x (I-\Delta_{x})^{-1}(e^{-\Phi}+\Phi-1)]ds\nnm\\
&\quad+\int^t_0W_4(t-s)\ast \dxa [\sqrt{M}v\cdot\nabla_x (I-\Delta_{x})^{-1}(e^{-\Phi}+\Phi-1)]ds\nnm\\
&=:I_3^1+I_3^2+I_3^3.
\ema
By Theorem \ref{3Dmvpbth1}, Lemma \ref{3dmvpbpw4j1} and Lemmas \ref{3dmvpbpw5}--\ref{3dmvpbpw9}, we have
\bma
\|I_3^1\|
&\leq CQ^2(t)\int^{t}_0\int_{\R^3}\bigg\{(1+t-s)^{-\frac{4+|\alpha|}{2}}\(e^{-\frac{|x-y|^2}{4D(1+t-s)}}+(1+t-s)^{-\frac{1}{2}}e^{-\frac{(|x-y|-\mathbf{c}(t-s))^2}{4D(1+t-s)}}\)\nnm\\
&\qquad +(1+t-s)^{-\frac{4+|\alpha|}{2}}B_{\frac{3}{2}}(t-s,|x-y|)1_{\{|x-y|\leq\mathbf{c}(t-s)\}}+e^{-\frac{|x-y|+t-s}{4D}}\bigg\}\nnm\\
&\qquad \times\bigg\{(1+s)^{-\frac{7}{2}}\(e^{-\frac{2|y|^2}{D_2(1+s)}}+B_{3}(s,|y|)1_{\{|y|\leq \mathbf{c} s\}}\)+e^{-\frac{2(|y|+s)}{D_2}}\nnm\\
&\qquad +(1+s)^{-\frac{9}{2}}\(e^{-\frac{2(|y|-\mathbf{c} s)^2}{D_2(1+s)}}+B_{2}(s,|y|-\mathbf{c} s)1_{\{|y|\leq \mathbf{c} s\}}\)\bigg\}dyds\nnm\\
&\leq CQ^2(t)(1+t)^{-\frac{|\alpha|}{2}}\Psi\(t,x;\frac{2D_2}{3}\),\label{3dmvpbpwn-HII}
\ema
where 
we have used
$$
|\nabla_x (e^{-\Phi}+\Phi-1)|\leq Ce^{\|\Phi\|_{L^{\infty}_x}}|\nabla_x\Phi||\Phi|\leq C(1+s)^{-\frac{1}{2}}\Psi\(s,x;D_2\)^2.
$$
For $I_3^2$ and $I_3^3$, we have by Lemma \ref{3dmvpbpw9} and Lemma \ref{3dmvpbpw11} that
\be
\|I_3^2\|+\|I_3^3\|\leq CQ^2(t)(1+t)^{-\frac{|\alpha|}{2}}\Psi\(t,x;\frac{2D_2}{3}\),\label{3dmvpbpwn-HIIj1}
\ee
where 
we have used
\bmas
&\quad |\partial_{x}^{\alpha}\nabla_{x}(I-\Delta_{x})^{-1}(e^{-\Phi}+\Phi-1) |\\
&\leq C\(\partial_{x}^{\alpha}\(|x|^{-1}e^{-|x|}\)\)\ast\(e^{\|\Phi\|_{L^{\infty}_{x}}}| \nabla_{x} \Phi| |\Phi|\)\\
&\leq CQ^2(t)(1+t)^{-\frac{1}{2}}\Psi\(s,x;\frac{9}{8}D_2\)^2,\quad |\alpha|=0,1.
\emas

By taking summation $\eqref{3dmvpbpwnl1}+\eqref{3dmvpbthg-3}+\eqref{3dmvpbpwn-l4}+\eqref{3dmvpbpwn-HII}+\eqref{3dmvpbpwn-HIIj1}$, it holds that
\bq\label{nldf1}
\|\partial_{x}^{\alpha} f(t,x)\|\leq C(\delta_0+Q^2(t))(1+t)^{-\frac{|\alpha|}{2}}\Psi\(t,x;\frac{2D_2}{3}\).
\eq

By \eqref{3DmVPB9}, we have
$$
\Phi=-(I-\Delta_{x})^{-1}n+(I-\Delta_{x})^{-1}(e^{-\Phi}+\Phi-1),
$$
which and \eqref{nldf1} and Lemma \ref{3dmvpbpw4j1} imply that for $1<\eta<\frac{3}{2}$,
\bma
|\Phi(t,x)|&\leq C\(|x|^{-1}e^{-|x|}\)\ast |n|+C\(|x|^{-1}e^{-|x|}\)\ast\(e^{\|\Phi\|_{L^{\infty}_x}}|\Phi||\Phi|\)\nnm\\
&\leq C(\delta_0+Q^2(t))\Psi\(t,x;\frac{2\eta D_2}{3}\),\label{p0p}\\
|\partial_t\Phi(t,x)|&\leq C\(|x|^{-1}e^{-|x|}\)\ast |\div_xm|+C\(|x|^{-1}e^{-|x|}\)\ast\(e^{\|\Phi\|_{L^{\infty}_x}}|\partial_t\Phi||\Phi|\)\nnm\\
&\leq C(\delta_0+Q^2(t))(1+t)^{-\frac{1}{2}}\Psi\(t,x;\frac{2\eta D_2}{3}\),\label{ptp}\\
|\partial_x^{\alpha}\nabla_x\Phi(t,x)|&\leq C\partial_x^{\alpha}\(|x|^{-1}e^{-|x|}\)\ast |\nabla_xn|+C\partial_x^{\alpha}\(|x|^{-1}e^{-|x|}\)\ast\(e^{\|\Phi\|_{L^{\infty}_x}}|\nabla_x\Phi||\Phi|\)\nnm\\
&\leq C(\delta_0+Q^2(t))(1+t)^{-\frac{1}{2}}\Psi\(t,x;\frac{2\eta D_2}{3}\), \quad |\alpha|=0,1.\label{pn0p}
\ema

By \eqref{3DmVPB8}, we have
\bq\label{rbenlv-1}
\partial_tf+v\cdot\nabla_x f+\nu(v)f=Kf+v\cdot\nabla_x\Phi\chi_0+\Lambda+\Gamma(f,f).
\eq
Then, we can represent $\dx^\alpha f$ as
\bq\label{rbenlv-2}
\dx^\alpha f =S^t\dx^\alpha f_0+\int^t_0S^{t-s}\dx^\alpha(Kf+v\cdot\nabla_x\Phi\chi_0+\Lambda+\Gamma(f,f))ds.
\eq
By Lemma \ref{3dmvpbpw10}, it holds that
\bq\label{rbenlv-3}
\|S^t\partial_{x}^{\alpha} f_0(x)\|_{L^{\infty}_{v,3}}\leq C\delta_0e^{-\frac{3(|x|+t)}{2D_2}}.
\eq
By \eqref{nldf1}, we obtain
$$
\|\partial_{x}^{\alpha} Kf \|_{L^{\infty}_{v,0}} \leq C\|\partial_{x}^{\alpha}  f\|\leq C(\delta_0+Q^2(t))(1+t)^{-\frac{|\alpha|}{2}}\Psi\(t,x;\frac{2\eta D_2}{3}\),
$$
which, together with Lemma \ref{3dmvpbpw11} and \eqref{MH1}, implies that for $\eta>1$,
\bma\label{rbenlv-4}
&\quad\bigg\|\int^t_0S^{t-s}\partial_{x}^{\alpha}(Kf+v\cdot\nabla_{x}\Phi\chi_0+\Lambda+\Gamma(f,f))ds\bigg\|_{L^{\infty}_{v,1}}\nnm\\
&\leq C(\delta_0+Q^2(t))(1+t)^{-\frac{|\alpha|}{2}}\Psi\(t,x;\frac{2\eta^2 D_2}{3}\).
\ema
It follows from \eqref{rbenlv-1}--\eqref{rbenlv-4} that
\bq\label{aaa}
\|\partial_{x}^{\alpha} f(t,x)\|_{L^{\infty}_{v,1}}\leq C(\delta_0+Q^2(t))(1+t)^{-\frac{|\alpha|}{2}}\Psi\(t,x;\frac{2\eta^2 D_2}{3}\).
\eq
By induction and
$$
\|\dx^\alpha Kf(t,x)\|_{L^{\infty}_{v,k}}\leq C\|\dx^\alpha f(t,x)\|_{L^{\infty}_{v,k-1}},\quad k\geq1,
$$
we have
\bq\label{nldf4}
\|\partial_{x}^{\alpha} f(t,x)\|_{L^{\infty}_{v,3}}\leq C(\delta_0+Q^2(t))(1+t)^{-\frac{|\alpha|}{2}}\Psi\(t,x;D_2\).
\eq

Taking the derivative $\partial^{\beta}_{v}$ to \eqref{rbenlv-1} with $|\beta|=1$, we have
\bq
\partial_t\partial^{\beta}_{v}f+v\cdot\nabla_x \partial^{\beta}_{v}f+\nu(v)\partial^{\beta}_{v}f= \Lambda_3+\partial^{\beta}_{v}\Gamma(f,f),
\eq
where
\bq
\begin{split}
 \Lambda_3&=-\partial_x^{\beta} f-\partial^{\beta}_{v}\nu(v)f+\partial^{\beta}_{v}(v\chi_0)\cdot\nabla_x \Phi+\partial^{\beta}_{v}(Kf)\\
&\quad+\frac{1}{2}\nabla_x \Phi\cdot\partial_v^{\beta}(vf)-\nabla_x \Phi\cdot \nabla_v\partial^{\beta}_{v}f+\partial_v^{\beta}(v\chi_0)\cdot\nabla_x(I-\Delta_x)^{-1}(e^{-\Phi}+\Phi-1).\nonumber
\end{split}
\eq
Thus, we can represent $\partial^{\beta}_{v}f$ as
\bq\label{H14}
\partial^{\beta}_{v}f(t,x)=S^t\partial^{\beta}_{v}f_0+\int^t_0S^{t-s}( \Lambda_3+\partial^{\beta}_{v}\Gamma(f,f))ds.
\eq
It follows from \eqref{3dmvpbthpr2-3}, \eqref{p0p} and \eqref{aaa} that
\be\label{H141}
\| \Lambda_3(s,x)\|_{L^{\infty}_{v,1}}+\|\partial^{\beta}_{v}\Gamma(f,f)(s,x)\|_{L^{\infty}_{v,1}}\leq C(\delta_0+Q^2(t)) \Psi\(t,x;\frac{2\eta D_2}{3}\),
\ee
where we have used
\bmas
&\|(\nabla_{v}K) f\|_{L^{\infty}_{v,0}}\leq C\| f\|_{L^{\infty}_{v,0}},\quad \|(\nabla_{v}K) f\|_{L^{\infty}_{v,k}}\leq C\| f\|_{L^{\infty}_{v,k-1}} ,\quad k\geq1,\\
&\|\nabla^2_{v}f\|_{L^{\infty}_{v,1}}\leq C(\|\langle v\rangle f\|_{H^4_{v}}+\|f\|_{H_v^3})\leq C(1+t)^{-\frac{3}{4}}Q(t).
\emas
By Lemma \ref{3dmvpbpw10}--\ref{3dmvpbpw11} and \eqref{H14}--\eqref{H141}, we obtain
\bq\label{H15}
\|\nabla_{v}f\|_{L^{\infty}_{v,2}}\leq C(\delta_0+Q^2(t)) \Psi(t,x;D_2).
\eq

By \eqref{ENI2} and $d_1H_{9,2}(f)\leq D_{9,2}(f)$ with $d_1>0$, we have
\bma
\mathcal{H}_{9,2}(f)(t)&\leq Ce^{-d_1 t}\mathcal{H}_{9,2}(f_0)+C\int^t_0e^{-d_1(t-s)}\|\nabla_{x} P_0f(s)\|^2_{L^2_{x,v} }ds\nnm\\
&\leq C\delta_0^2e^{-d_1 t}+C(\delta_0+Q^2(t))^2\int^t_0e^{-d_1(t-s)}(1+s)^{-\frac{5}{2}}ds\nnm\\
&\leq C(1+t)^{-\frac{5}{2}}(\delta_0+Q^2(t))^2.\label{H16}
\ema

By combining  \eqref{p0p}--\eqref{ptp}, \eqref{nldf4}, \eqref{H15} and \eqref{H16} we have
$$
Q(t)\leq C_1\delta_0+C_2Q(t)^2,
$$
from which \eqref{3dmvpbth2-2}--\eqref{3dmvpbth2-3} can be verified so long as $\delta_0>0$ is small enough.


\section{Appendix}

In this section, we  give some basic estimates  of convolution of the initial data and different waves to analyze the pointwise behaviors of the solution as follows.

\begin{lem}\label{3dmvpbpw4j1}
For any $D,\lambda\geq0$, $\eta>1$ and $|\alpha|=0,1$, there exists a constant $C>0$ such that
\bma
&\quad\int_{\R^3}\partial_x^{\alpha}\(|x-y|^{-1}e^{-|x-y|}\)e^{-\frac{|y|^2}{D(1+t)}}dy\leq Ce^{-\frac{|x|^2}{\eta^2D(1+t)}}+Ce^{-\frac{|x|+t}{ \eta^2D}},\label{3dmvpbpw4j1-1}\\
&\quad\int_{\R^3}\partial_x^{\alpha}\(|x-y|^{-1}e^{-|x-y|}\)e^{-\frac{(|y|-\lambda t)^2}{D(1+t)}}dy\leq Ce^{-\frac{(|x|-\lambda t)^2}{ \eta^2D(1+t)}}+Ce^{-\frac{|x|+t}{\eta^2D}},\label{3dmvpbpw4j1-2}\\
&\quad\int_{\{|y|\leq\lambda t\}}\partial_x^{\alpha}\(|x-y|^{-1}e^{-|x-y|}\)B_{\frac{3}{2}}(t,|y|)dy\nnm\\
&\leq C(1+t)^{-\frac{3}{2}}e^{-\frac{(|x|-\lambda t)^2}{\eta^2D(1+t)}}+Ce^{-\frac{|x|+t}{\eta^2D}}+CB_{\frac{3}{2}}(t,|x|)1_{\{|x|\leq\lambda t\}},\label{3dmvpbpw4j1-3}
\ema
where $D>\frac{4\max\{1,\lambda\}}{\eta^2-\eta}$.
\end{lem}
\begin{proof}
First, we prove \eqref{3dmvpbpw4j1-1}. We split $y$ into $|y|\leq\frac{1}{\eta}|x|$ and $|y|\geq\frac{1}{\eta}|x|$. If $|y|\leq\frac{1}{\eta}|x|$, we have $|x-y|\geq(1-\frac{1}{\eta})|x|$. Thus,
\bma
&\quad\int_{\R^3}\partial_x^{\alpha}\(|x-y|^{-1}e^{-|x-y|}\)e^{-\frac{|y|^2}{D(1+t)}}dy\nnm\\
&=\(\int_{\{|y|\leq\frac{1}{\eta}|x|\}}+\int_{\{|y|\geq\frac{1}{\eta}|x|\}}\)\partial_x^{\alpha}\(|x-y|^{-1}e^{-|x-y|}\)e^{-\frac{|y|^2}{D(1+t)}}dy\nnm\\
&\leq Ce^{-\frac{(1-\frac{1}{\eta})|x|}{2}}\bigg|\int_{\R^3}\(\frac1{|x-y|}+\frac1{|x-y|^{1+|\alpha|}}\)e^{-\frac{|x-y|}{2}}dy\bigg|\nnm\\
&\quad+Ce^{-\frac{|x|^2}{\eta^2D(1+t)}}\bigg|\int_{\R^3}\(\frac1{|x-y|}+\frac1{|x-y|^{1+|\alpha|}}\)e^{-|x-y|}dy\bigg|\nnm\\
&\leq Ce^{-\frac{|x|^2}{\eta^2D(1+t)}}+Ce^{-\frac{(1-\frac{1}{\eta})|x|}{2}}, \quad |\alpha|=0,1.
\ema
Since for any positive constant $D_3>0$,
\bq\label{rbepw204j1}
\left\{\bal
e^{-\frac{|x|}{D_3}}\leq e^{-\frac{|x|^2}{2D_3(1+t)}}, & |x|\leq  t,\\
e^{-\frac{|x|}{D_3}}\leq e^{-\frac{|x|+ t}{2D_3}}, & |x|> t,
\ea\right.
\eq
then we can obtain \eqref{3dmvpbpw4j1-1} for $D>\frac{4}{\eta^2-\eta}$ with $\eta>1$.

Next, we prove \eqref{3dmvpbpw4j1-2} as follows. We split $y$ into $|x-y|\leq(1-\frac{1}{\eta})||x|-\lambda t|$ and $|x-y|\geq(1-\frac{1}{\eta})||x|-\lambda t|$.
For $|x|\leq\lambda t$ and $|x-y|\leq(1-\frac{1}{\eta})||x|-\lambda t|$, we have $\lambda t-|y|\geq\lambda t-|x|-|x-y|\geq\frac{1}{\eta}(\lambda t-|x|)$.
For $|x|>\lambda t$ and $|x-y|\leq(1-\frac{1}{\eta})||x|-\lambda  t|$, we have $|y|-\lambda t\geq|x|-|x-y|-\lambda t\geq\frac{1}{\eta}(|x|-\lambda t)$.
Thus,
\bma
&\quad\int_{\R^3}\partial_x^{\alpha}\(|x-y|^{-1}e^{-|x-y|}\)e^{-\frac{(|y|-\lambda t)^2}{D(1+t)}}dy\nnm\\
&=\(\int_{\{|x-y|\leq(1-\frac{1}{\eta})||x|-\lambda t|\}}+\int_{\{|x-y|\geq(1-\frac{1}{\eta})||x|-\lambda t|\}}\)\partial_x^{\alpha}\(|x-y|^{-1}e^{-|x-y|}\)e^{-\frac{(|y|-\lambda t)^2}{D(1+t)}}dy \nnm\\
&\leq Ce^{-\frac{(|x|-\lambda t)^2}{\eta^2D(1+t)}}\bigg|\int_{\R^3}\(\frac1{|x-y|}+\frac1{|x-y|^{1+|\alpha|}}\)e^{-|x-y|}dy\bigg|\nnm\\
&\quad+Ce^{-\frac{(1-\frac{1}{\eta})||x|-\lambda t|}{2}}\bigg|\int_{\R^3}\(\frac1{|x-y|}+\frac1{|x-y|^{1+|\alpha|}}\)e^{-\frac{|x-y|}{2}}dy\bigg|\nnm\\
&\leq Ce^{-\frac{(|x|-\lambda t)^2}{\eta^2D(1+t)}}+Ce^{-\frac{(1-\frac{1}{\eta})||x|-\lambda t|}{2}}.
\ema
Since it holds for $\lambda_1=\max\{1,\lambda\}$ that
\bq\label{rbepw204}
\left\{\bal
e^{-\frac{||x|-\lambda t|}{D_3}}\leq e^{-\frac{(|x|-\lambda t)^2}{2D_3\lambda_1 (1+t)}},& ||x|-\lambda t|\leq2\lambda_1 t,\\
e^{-\frac{||x|-\lambda t|}{D_3}}\leq e^{-\frac{|x|+\lambda_1 t}{2D_3}},& ||x|-\lambda t|\geq2\lambda_1 t,
\ea\right.
\eq
then we can obtain \eqref{3dmvpbpw4j1-2} for $D>\frac{4\max\{1,\lambda\}}{\eta^2-\eta}$ with $\eta>1$.

Finally, we prove \eqref{3dmvpbpw4j1-3}. For $ |x|\leq\lambda t$, we have
\bma
&\quad \int_{\{|y|\leq\lambda t\}}\partial_x^{\alpha}\(|x-y|^{-1}e^{-|x-y|}\)B_{\frac{3}{2}}(t,|y|)dy\nnm\\
&=\(\int_{\{|y|\leq\frac{|x|}{2}\}}+\int_{\{|y|\geq\frac{|x|}{2}\}}\)\partial_x^{\alpha}\(|x-y|^{-1}e^{-|x-y|}\)B_{\frac{3}{2}}(t,|y|)dy\nnm\\
&\leq CB_{\frac{3}{2}}(t,|x|).
\ema
For $|x|>\lambda t$ and $|y|\leq\lambda t$, we have $|x-y|\geq|x|-\lambda t$. We split $y$ into $|y|\leq\frac{\lambda t}{2}$ and $|y|\geq\frac{\lambda t}{2}$. If $|y|\leq\frac{\lambda t}{2}$ and $|x|>\lambda t$, then we have $|x-y|\ge \frac{\lambda t}{2}$. Thus,
\bma
&\quad\int_{\{|y|\leq\lambda t\}}\partial_x^{\alpha}\(|x-y|^{-1}e^{-|x-y|}\)B_{\frac{3}{2}}(t,|y|)dy\nnm\\
&\leq C\(1+\frac{\lambda^2t^2}{4(1+t)}\)^{-\frac{3}{2}}e^{-\frac{||x|-\lambda t|}{2}}\bigg|\int_{\{|y|\geq\frac{\lambda t}{2}\}}\(\frac1{|x-y|}+\frac1{|x-y|^{1+|\alpha|}}\)e^{-\frac{|x-y|}{2}}dy\bigg|\nnm\\
&\quad+Ce^{-\frac{\lambda t}{8}}e^{-\frac{||x|-\lambda t|}{4}}\bigg|\int_{\{|y|\leq\frac{\lambda t}{2}\}}\(\frac1{|x-y|}+\frac1{|x-y|^{1+|\alpha|}}\)e^{-\frac{|x-y|}{4}}dy\bigg|\nnm\\
&\leq C(1+t)^{-\frac{3}{2}}e^{-\frac{(|x|-\lambda t)^2}{\eta^2D(1+t)}}+Ce^{-\frac{|x|+t}{\eta^2D}},
\ema
where we have used \eqref{rbepw204} and $ D\geq\frac{8\max\{1,\lambda\}}{\eta^2}$. The proof of the lemma is completed.
\end{proof}

\begin{lem}[\cite{LI-8}]\label{3dmvpbpw4}
For any $\alpha\geq0$, $D,D_1>0$ and $\lambda\geq0$, there exist $C,D_2>0$ such that for $D_2\ge 12\max\{D,D_1\}\times\max\{1,\lambda\}$,
\bmas
&\quad\int_{\R^3}(1+t)^{-\alpha}e^{-\frac{|x-y|^2}{D(1+t)}}e^{-\frac{|y|}{D_1}}dy\leq C(1+t)^{-\alpha}e^{-\frac{3|x|^2}{2D_2(1+t)}}+Ce^{-\frac{3 (|x|+t)}{2D_2}},\\
&\quad\int_{\R^3}(1+t)^{-\alpha}e^{-\frac{(|x-y|-\lambda t)^2}{D(1+t)}}e^{-\frac{|y|}{D_1}}dy\leq C(1+t)^{-\alpha}e^{-\frac{3(|x|-\lambda t)^2}{2D_2(1+t)}}+Ce^{-\frac{3(|x|+t)}{2D_2}},\\
&\quad\int_{\{|x-y|\leq\lambda t\}}(1+t)^{-\alpha}B_{\frac{3}{2}}(t,|x-y|)e^{-\frac{|y|}{D_1}}dy\nnm\\
&\leq C(1+t)^{-\alpha-\frac{3}{2}}e^{-\frac{3(|x|-\lambda t)^2}{2D_2(1+t)}}+Ce^{-\frac{3(|x|+t)}{2D_2}}+C(1+t)^{-\alpha}B_{\frac{3}{2}}(t,|x|)1_{\{|x|\leq\lambda t\}}.
\emas
\end{lem}



We now consider the following integrals for estimating the nonlinear interactions:
\bmas
&\quad I^{\alpha,\beta}(t,x;t_1,t_2;\mu_1,\mu_2,D,D_1)\\
&=\int^{t_2}_{t_1}\int_{\R^3}(1+t-s)^{-\alpha}e^{-\frac{(|x-y|-\mu_1(t-s))^2}{D(1+t-s)}}(1+s)^{-\beta}e^{-\frac{2(|y|-\mu_2 s)^2}{D_1(1+s)}}dyds,\\
&\quad J^{\alpha,\beta}(t,x;t_1,t_2;\lambda,\mu,D_1)\\
&=\int^{t_2}_{t_1}\int_{\{|x-y|\leq\lambda(t-s)\}}(1+t-s)^{-\alpha}B_{\frac{3}{2}}(t-s,|x-y|)(1+s)^{-\beta}e^{-\frac{2(|y|-\mu s)^2}{D_1(1+s)}}dyds,\\
&\quad K^{\alpha,\beta}(t,x;t_1,t_2;\lambda,\mu)\\
&=\int^{t_2}_{t_1}\int_{\{|x-y|\leq\lambda(t-s)\}}(1+t-s)^{-\alpha}B_{\frac{3}{2}}(t-s,|x-y|) (1+s)^{-\beta}B_{k}(s,|y|-\mu s)1_{\{|y|\leq\lambda s\}}dyds,\\
&\quad L^{\alpha,\beta}(t,x;t_1,t_2;\lambda,\mu_1,\mu,D)\\
&=\int^{t_2}_{t_1}\int_{\R^3}(1+t-s)^{-\alpha}e^{-\frac{(|x-y|-\mu_1(t-s))^2}{D(1+t-s)}} (1+s)^{-\beta}B_{k}(s,|y|-\mu s)1_{\{|y|\leq\lambda s\}}dyds,
\emas
where $k=3$ for $\mu=0$, and $k=2$ for $\mu>0.$ And
\bmas
&M^{\alpha}(t,x;t_1,t_2;\mu,D,D_1)=\int^{t_2}_{t_1}\int_{\R^3}e^{-\frac{|x-y|+t-s}{D}}(1+s)^{-\alpha}e^{-\frac{ (|y|-\mu s)^2}{D_1(1+s)}}dyds,\\
&N^{\alpha}(t,x;t_1,t_2;\lambda,\mu,D)=\int^{t_2}_{t_1}\int_{\{|y|\leq\lambda s\}}e^{-\frac{|x-y|+t-s}{D}}(1+s)^{-\alpha}B_{l}(s,|y|-\mu s)dyds,
\emas
where $l>3/2$.

Set
$$
\Gamma_{\alpha}(t)=\int_0^t(1+s)^{-\alpha}ds=C
\left\{\bln
&(1+t)^{1-\alpha},\quad\alpha<1,\\
&\ln(1+t), \quad~~\alpha=1,\\
&1,\qquad\qquad\quad \alpha>1.
\eln\right.
$$

\begin{lem}[\cite{LI-8}]\label{3dmvpbpw5}
For given $\alpha,\beta>0$ and $D>0$, there exist two constants $C, D_1>0$ such that for $D_1\ge 16D$,
\bma
&\quad I^{\alpha,\beta}(t,x;0,t;0,0,D,D_1)\nnm\\
&\leq C\((1+t)^{-\alpha}\Gamma_{\beta-\frac{3}{2}}(t)+(1+t)^{-\beta}\Gamma_{\alpha-\frac{3}{2}}(t)\)e^{-\frac{3|x|^2}{2D_1(1+t)}},\label{rbepw301}\\
&\quad I^{\alpha,\beta}(t,x;0,t;0,\lambda,D,D_1),~I^{\beta,\alpha}(t,x;0,t;0,\lambda,D,D_1)\nnm\\
&\leq C\((1+t)^{-\alpha}\Gamma_{\beta-\frac{5}{2}}(t)+(1+t)^{-\beta}\Gamma_{\alpha-\frac{5}{2}}(t)\)\(e^{-\frac{3|x|^2}{2D_1(1+t)}}+e^{-\frac{3(|x|-\lambda t)^2}{2D_1(1+t)}}\)\nnm\\
&\quad+C(1+t)^{-\alpha+2}\(1+|x|\)^{-\beta}1_{\{\sqrt{1+t}\leq|x|\leq\lambda t-\sqrt{1+t}\}}\nnm\\
&\quad+C(1+t)^{-\beta+2}\(1+\lambda t-|x|\)^{-\alpha}1_{\{\sqrt{1+t}\leq|x|\leq\lambda t-\sqrt{1+t}\}},\label{rbepw302}\\
&\quad I^{\alpha,\beta}(t,x;0,t;\lambda,\lambda,D,D_1)\nnm\\
&\leq C\((1+t)^{-\alpha}\Gamma_{\beta-\frac{5}{2}}(t)+(1+t)^{-\beta}\Gamma_{\alpha-\frac{5}{2}}(t)\)\(e^{-\frac{3|x|^2}{2D_1(1+t)}}+e^{-\frac{3(|x|-\lambda t)^2}{2D_1(1+t)}}\)\nnm\\
&\quad  +C (1+t)^{-\alpha+1}(1+\lambda t-|x|)^{-\beta+\frac52}1_{\{\sqrt{1+t}\leq|x|\leq\lambda t-\sqrt{1+t}\}} \nnm\\
&\quad  +C(1+t)^{-\beta+1}(1+\lambda t-|x|)^{-\alpha+\frac52}1_{\{\sqrt{1+t}\leq|x|\leq\lambda t-\sqrt{1+t}\}} .\label{rbepw303}
\ema
\end{lem}


\begin{lem}[\cite{LI-8}]\label{3dmvpbpw6}
For given $\alpha,\beta,\lambda>0$ and $D>0$, there exist two constants $C, D_1>0$ such that for $D_1\ge 2D$,
\bmas
&\quad  L^{\alpha,\beta}(t,x;0,t;\lambda,0,0,D)\nnm\\
&\leq C\((1+t)^{-\alpha}\Gamma_{\beta-\frac{3}{2}}(t)+(1+t)^{-\beta}\Gamma_{\alpha-\frac{3}{2}}(t)\) \(e^{-\frac{3(|x|-\lambda t)^2}{2D_1(1+t)}}+B_{3}(t,|x|)1_{\{|x|\leq\lambda t\}}\),\\
&\quad J^{\alpha,\beta}(t,x;0,t;\lambda,0,D_1)\nnm\\
&\leq C\((1+t)^{-\alpha}\Gamma_{\beta-\frac{3}{2}}(t)+(1+t)^{-\beta}\ln(1+t)\Gamma_{\alpha-\frac{3}{2}}(t)\) \(e^{-\frac{3(|x|-\lambda t)^2}{2D_1(1+t)}}+B_{\frac{3}{2}}(t,|x|)1_{\{|x|\leq\lambda t\}}\),\\
&\quad J^{\alpha,\beta}(t,x;0,t;\lambda,\lambda,D_1)\nnm\\
&\leq C\((1+t)^{-\alpha}\Gamma_{\beta-\frac{5}{2}}(t)+(1+t)^{-\beta}\Gamma_{\alpha-\frac{5}{2}}(t)\)\(e^{-\frac{3|x|^2}{2D_1(1+t)}}+e^{-\frac{3(|x|-\lambda t)^2}{2D_1(1+t)}}\)\nnm\\
&\quad+C\((1+t)^{-\alpha}\Gamma_{\beta-\frac{5}{2}}(t)B_{\frac{3}{2}}(t,|x|)+(1+t)^{-\beta+1}\Gamma_{\alpha-\frac{3}{2}}(t)B_{\frac{3}{2}}(t,|x|-\lambda t)\)1_{\{|x|\leq\lambda t\}}\nnm\\
&\quad+C(1+t)^{-\alpha+2}\ln(1+t)\(1+|x|\)^{-\beta}1_{\{\sqrt{1+t}\leq|x|\leq\lambda t-\sqrt{1+t}\}}\nnm\\
&\quad+C(1+t)^{-\beta+2}\ln(1+t)\(1+\lambda t-|x|\)^{-\alpha}1_{\{\sqrt{1+t}\leq|x|\leq\lambda t-\sqrt{1+t}\}}.
\emas
\end{lem}


\begin{lem}[\cite{LI-8}]\label{3dmvpbpw7}
For given $\alpha,\beta>0$, there exists a constant $C>0$ such that
\bmas
&\quad K^{\alpha,\beta}(t,x;0,t;\lambda,0)\nnm\\
&\leq C\((1+t)^{-\alpha}\Gamma_{\beta-\frac{3}{2}}(t)+(1+t)^{-\beta}\Gamma_{\alpha-\frac{3}{2}}(t)\)\ln(1+t)B_{\frac{3}{2}}(t,|x|)1_{\{|x|\leq\lambda t\}},\\
&\quad K^{\alpha,\beta}(t,x;0,t;\lambda,\lambda)\nnm\\
&\leq C\((1+t)^{-\alpha}\Gamma_{\beta-\frac{5}{2}}(t)+(1+t)^{-\beta+1}\Gamma_{\alpha-\frac{3}{2}}(t)\) \(B_{\frac{3}{2}}(t,|x|)+B_{\frac32}(t,|x|-\lambda t)\)1_{\{|x|\leq\lambda t\}}\nnm\\
&\quad+C(1+t)^{-\alpha+2}\ln(1+t)\(1+|x|\)^{-\beta}1_{\{\sqrt{1+t}\leq|x|\leq\lambda t-\sqrt{1+t}\}}\nnm\\
&\quad+C(1+t)^{-\beta+2}\ln(1+t)\(1+\lambda t-|x|\)^{-\alpha}1_{\{\sqrt{1+t}\leq|x|\leq\lambda t-\sqrt{1+t}\}}.
\emas
\end{lem}


\begin{lem}[\cite{LI-8}]\label{3dmvpbpw8}
For given $\alpha,\beta,D>0$, there exist constants $C, D_1>0$ such that for $D_1\ge \frac{3D}{2}$,
\bmas
&\quad L^{\alpha,\beta}(t,x;0,t;\lambda,0,\lambda,D),\,  L^{\beta,\alpha}(t,x;0,t;\lambda,0,\lambda,D)\nnm\\
&\leq C\((1+t)^{-\alpha}\Gamma_{\beta-\frac{5}{2}}(t)+(1+t)^{-\beta}\Gamma_{\alpha-\frac{5}{2}}(t)\)\(e^{-\frac{3|x|^2}{2D_1(1+t)}}+e^{-\frac{3(|x|-\lambda t)^2}{2D_1(1+t)}}\)\nnm\\
&\quad +C\((1+t)^{-\alpha}\Gamma_{\beta-\frac{5}{2}}(t)B_{2}(t,|x|)+(1+t)^{-\beta}\Gamma_{\alpha-\frac{5}{2}}(t)B_{2}(t,|x|-\lambda t)\)1_{\{|x|\leq\lambda t\}}\nnm\\
&\quad+C(1+t)^{-\alpha+2}\(1+|x|\)^{-\beta}1_{\{\sqrt{1+t}\leq|x|\leq\lambda t-\sqrt{1+t}\}}\nnm\\
&\quad+C(1+t)^{-\beta+2}\(1+\lambda t-|x|\)^{-\alpha}1_{\{\sqrt{1+t}\leq|x|\leq\lambda t-\sqrt{1+t}\}},\\
&\quad L^{\alpha,\beta}(t,x;0,t;\lambda,\lambda,\lambda,D)\nnm\\
&\leq C\((1+t)^{-\alpha}\Gamma_{\beta-\frac{5}{2}}(t)+(1+t)^{-\beta}\Gamma_{\alpha-\frac{5}{2}}(t)\)\(e^{-\frac{3|x|^2}{2D_1(1+t)}}+e^{-\frac{3(|x|-\lambda t)^2}{2D_1(1+t)}}\)\nnm\\
&\quad+C\((1+t)^{-\alpha+\frac{1}{2}}\Gamma_{\beta-2}(t)+(1+t)^{-\beta}\Gamma_{\alpha-\frac{5}{2}}(t)\) B_{2}(t,|x|-\lambda t)1_{\{|x|\leq\lambda t\}}\nnm\\
&\quad + C (1+t)^{-\alpha+1}(1+\lambda t-|x|)^{-\beta+\frac52}1_{\{\sqrt{1+t}\leq|x|\leq\lambda t-\sqrt{1+t}\}} \nnm\\
&\quad + C(1+t)^{-\beta+1}(1+\lambda t-|x|)^{-\alpha+\frac52}1_{\{\sqrt{1+t}\leq|x|\leq\lambda t-\sqrt{1+t}\}} .
\emas
\end{lem}


\begin{lem}\label{3dmvpbpw9}
For any $0<\eta<\frac{2}{3}$ and $k>\frac{3}{2}$, there exist two constants $C, D_2>0$ such that for $D_2>\max \{\frac{3D}{2-2\sqrt{\frac{3\eta}{2}}}\max\{1,\lambda\}^2,\frac{6D}{1-\sqrt{\frac{3\eta}{2}}} \}$,
\bma
&\quad M^{\alpha}(t,x;0,t;0,D,\eta D_2)\nnm\\
&\leq C(1+t)^{-\alpha}e^{-\frac{3|x|^2}{2D_2(1+t)}}+Ce^{-\frac{3(|x|+t)}{2D_2}},\label{rbe-pw-rbepw701}\\
&\quad N^{\alpha}(t,x;0,t;\lambda,0,D)\nnm\\
&\leq C(1+t)^{-\alpha}e^{-\frac{3(|x|-\lambda t)^2}{2D_2(1+t)}}+C(1+t)^{-\alpha}B_{k}(t,|x|)1_{\{|x|\leq\lambda t\}}+Ce^{-\frac{3(|x|+t)}{2D_2}},\label{rbe-pw-rbepw702}\\
&\quad M^{\alpha}(t,x;0,t;\lambda,D,\eta D_2),~ N^{\alpha}(t,x;0,t;\lambda,\lambda,D)\nnm\\
&\leq C(1+t)^{-\alpha}\(e^{-\frac{3(|x|-\lambda t)^2}{2D_2(1+t)}}+e^{-\frac{3|x|^2}{2D_2(1+t)}}\)+Ce^{-\frac{3(|x|+t)}{2D_2}}\nnm\\
 &\quad+C(1+t)^{-\alpha}B_{k}(t,|x|-\lambda t)1_{\{|x|\leq\lambda t\}},\label{rbe-pw-rbepw703}
\ema
where $\lambda,~\alpha$ and $D$ are positive constants.
\end{lem}
\begin{proof}
First, we prove \eqref{rbe-pw-rbepw701}.  We split $y$ into $|y|\leq\frac{|x|}{\epsilon}$ and $|y|\geq\frac{|x|}{\epsilon}$ with $ \epsilon=\sqrt{\frac{2}{3\eta}}>1$. Thus,
\bma
&\quad M^{\alpha}(t,x;0,t;0,D,\eta D_2)\nnm\\
&=\int_0^t\(\int_{\{|y|\leq\frac{|x|}{\epsilon}\}}+\int_{\{|y|\geq\frac{|x|}{\epsilon}\}}\)e^{-\frac{|x-y|+t-s}{D}}(1+s)^{-\alpha}e^{-\frac{|y|^2}{\eta D_2(1+s)}}dyds\nnm\\
&\leq Ce^{-(1-\frac{1}{\epsilon})\frac{|x|}{2D}}\int_0^te^{-\frac{t-s}{D}}(1+s)^{-\alpha}\int_{\R^3}e^{-\frac{|x-y|}{2D}}dyds\nnm\\
&\quad+Ce^{-\frac{|x|^2}{\eta\epsilon^2D_2(1+t)}}\int_0^te^{-\frac{t-s}{D}}(1+s)^{-\alpha}\int_{\R^3}e^{-\frac{|x-y|}{D}}dyds\nnm\\
&\leq C(1+t)^{-\alpha}e^{-\frac{3|x|^2}{2D_2(1+t)}}+C(1+t)^{-\alpha}e^{-\frac{3(|x|+t)}{2D_2}},\label{rbe-pw-rbepw7011}
\ema
 where we have used \eqref{rbepw204j1} and $ D_2\ge \frac{6D}{1-\sqrt{\frac{3\eta}{2}}} $. This proves \eqref{rbe-pw-rbepw701}.

Then, we prove \eqref{rbe-pw-rbepw702}. For $|x|\leq\lambda t$, we split $y$ into $|y|\leq\frac{|x|}{2}$ and $|y|\geq\frac{|x|}{2}$. Thus,
\bma
N^{\alpha}(t,x;0,t;\lambda,0,D)
&\leq C\(e^{-\frac{|x|}{4D}}+B_k(t,|x|)\)\int_0^te^{-\frac{t-s}{D}}(1+s)^{-\alpha}\int_{\R^3}e^{-\frac{|x-y|}{2D}}dy ds\nnm\\
&\leq C(1+t)^{-\alpha}B_{k}(t,|x|)+Ce^{-\frac{3|x|^2}{2D_2(1+t)}}+Ce^{-\frac{3(|x|+t)}{2D_2}},\label{rbe-pw-rbepw702-1}
\ema
where we have used \eqref{rbepw204j1}, and $D_2\ge 12D $.

For $|x|\geq\lambda t$ and $|y|\leq\lambda s$, we have $|x-y|\geq|x|-\lambda t$. Thus,
\bma
N^{\alpha}(t,x;0,t;\lambda,0,D)&\leq Ce^{-\frac{||x|-\lambda t|}{2D}}\int_0^te^{-\frac{t-s}{D}}(1+s)^{-\alpha}\int_{\R^3}e^{-\frac{|x-y|}{2D}}dy ds\nnm\\
&\leq C(1+t)^{-\alpha}e^{-\frac{3(|x|-\lambda t)^2}{2D_2(1+t)}}+Ce^{-\frac{3(|x|+t)}{2D_2}},\label{rbe-pw-rbepw702-2}
\ema
where we have used \eqref{rbepw204} and $D_2\ge 6D\lambda_1$ with $\lambda_1=\max\{1,\lambda\}$. By \eqref{rbe-pw-rbepw702-1} and \eqref{rbe-pw-rbepw702-2}, we obtain \eqref{rbe-pw-rbepw702}.

Finally, we prove \eqref{rbe-pw-rbepw703}. Let
\be
Q^{\alpha} =\int_0^{t}\int_{\R^3}e^{-\frac{|x-y|+(t-s)}{D}}(1+s)^{-\alpha}B_{k}(s,|y|-\lambda s)dyds, \quad k>\frac{3}{2}.\label{rbe-pw-rbepw703-2}
\ee
It holds that
\be
 M^{\alpha}(t,x;0,t;\lambda,D,\eta D_2),\quad N^{\alpha}(t,x;0,t;\lambda,\lambda,D)\leq Q^{\alpha} .\label{rbe-pw-rbepw703-3}
\ee
For $|x|<\lambda t$, since $ e^{-\frac{|x-y|+t-s}{2D}} \leq e^{-\frac{ ||x-y|-\lambda(t-s)|}{DE}} $ for $E=2\max\{1,\lambda\}$, it follows that
\bma
Q^{\alpha}&\leq\(\int_0^{\frac{\lambda t-|x|}{4\lambda}}+\int^{t-\frac{\lambda t-|x|}{4\lambda}}_{\frac{\lambda t-|x|}{4\lambda}}+\int_{t-\frac{\lambda t-|x|}{4\lambda}}^{t}\)\int_{\R^3}e^{-\frac{|x-y|+t-s}{2D}}\nnm\\
&\qquad\qquad\times(1+s)^{-\alpha}e^{-\frac{||x-y|-\lambda(t-s)|}{DE}}B_{k}(s,|y|-\lambda s)dyds\nnm\\
&=:I_1+I_2+I_3.\label{rbe-pw-rbepw703-4}
\ema

For $I_1$, we split $y$ into $|y|\leq\frac{\lambda t-|x|}{2}$ and $|y|\geq\frac{\lambda t-|x|}{2}$. If $|y|\geq\frac{\lambda t-|x|}{2}$ and $s\leq\frac{\lambda t-|x|}{4\lambda}$, then we have $|y|-\lambda s\geq\frac{\lambda t-|x|}{4}$. If $|y|\leq\frac{\lambda t-|x|}{2}$ and $s\leq\frac{\lambda t-|x|}{4\lambda}$, then we have $\lambda(t-s)-|x-y|\geq\lambda t-\lambda s-|x|-|y|\geq\frac{\lambda t-|x|}{4}$. Thus,
\bma
I_1&\leq C\(e^{-\frac{|\lambda t-|x||}{4DE}}+B_{k}(t,|x|-\lambda t)\)\int_0^{\frac{\lambda t-|x|}{4\lambda}}e^{-\frac{ t-s}{2D}}(1+s)^{-\alpha} ds\nnm\\
&\leq C(1+t)^{-\alpha}B_{k}(t,|x|-\lambda t).
\ema

For $I_3$, we split $y$ into $|x-y|\leq\frac{\lambda t-|x|}{2}$ and $|x-y|\geq\frac{\lambda t-|x|}{2}$. Since $|x-y|\leq\frac{\lambda t-|x|}{2}$ and $t-s\leq\frac{\lambda t-|x|}{4\lambda}$, then we have $\lambda s-|y|\geq\lambda t-|x|-|x-y|-\lambda(t-s)\geq\frac{\lambda t-|x|}{4}$. Since $|x-y|\geq\frac{\lambda t-|x|}{2}$ and $t-s\leq\frac{\lambda t-|x|}{4\lambda}$, then we have $|x-y|-\lambda(t-s)\geq\frac{\lambda t-|x|}{4}$. Thus,
\bma
I_3&\leq C\(B_{k}(t,|x|-\lambda t)+e^{-\frac{|\lambda t-|x||}{4DE}}\)\int_{t-\frac{\lambda t-|x|}{4\lambda}}^{t}e^{-\frac{ t-s}{2D}}(1+s)^{-\alpha} ds\nnm\\
&\leq C(1+t)^{-\alpha}B_{k}(t,|x|-\lambda t).
\ema

For $I_2$, it holds that
\bma
I_2&\leq \(\int^{\frac{t}{2}}_{\frac{\lambda t-|x|}{4\lambda}}+\int^{t-\frac{\lambda t-|x|}{4\lambda}}_{\frac{t}{2}}\)e^{-\frac{ t-s}{2D}}(1+s)^{-\alpha} ds\nnm\\
&\leq Cte^{-\frac{t}{4D}}+C(1+t)^{-\alpha}e^{-\frac{(\lambda t-|x|)}{16\lambda D}}\nnm\\
&\leq Ct(1+t)^{2k}e^{-\frac{t}{4D}}\(1+\frac{\lambda t-|x|}{4\lambda}\)^{-2k}+C(1+t)^{-\alpha}e^{-\frac{(\lambda t-|x|)}{16\lambda D}}\nnm\\
&\leq C(1+t)^{-\alpha}B_{k}(t,|x|-\lambda t).
\ema

For $|x|\geq\lambda t$ and $|y|\leq\lambda s$, we have $|x-y|\geq|x|-\lambda t$. Thus, it holds that
\bma
 N^{\alpha}(t,x;0,t;\lambda,\lambda,l,D)
&\leq Ce^{-\frac{||x|-\lambda t|}{2D}}\int_0^te^{-\frac{t-s}{D}}(1+s)^{-\alpha}\int_{\R^3}e^{-\frac{|x-y|}{2D}}dyds\nnm\\
&\leq C(1+t)^{-\alpha}e^{-\frac{3(|x|-\lambda t)^2}{2D_2(1+t)}}+Ce^{-\frac{3(|x|+t)}{2D_2}},\label{rbe-pw-rbepw703-5}
\ema
where we have used \eqref{rbepw204}, and $D_1\ge 6D\max\{1,\lambda\}$.

For $|x|\geq\lambda t$, we split $y$ into $||y|-\lambda s|\leq\frac{|x|-\lambda t}{\epsilon}$ and $||y|-\lambda s|\geq\frac{|x|-\lambda t}{\epsilon}$ with $\epsilon=\sqrt{\frac{2}{3\eta}}>1$. Thus,
\bma
&\quad M^{\alpha}(t,x;0,t;\lambda,D,\eta D_2)\nnm\\
&\leq C\(e^{-\frac{(1-\frac{1}{\epsilon})|x|-\lambda t}{DE}}+e^{-\frac{(|x|-\lambda t)^2}{\eta\epsilon^2D_2(1+t)}}\)\int_0^te^{-\frac{t-s}{2D}}(1+s)^{-\alpha}\int_{\R^3}e^{-\frac{|x-y|}{2D}}dyds\nnm\\
&\leq C(1+t)^{-\alpha}e^{-\frac{3(|x|-\lambda t)^2}{2D_2(1+t)}}+Ce^{-\frac{3(|x|+t)}{2D_2}},\label{rbe-pw-rbepw703-6}
\ema
where $D_2\geq\frac{3 D}{2(1-\sqrt{\frac{3\eta}{2}})}\max\{1,\lambda\}^2$, and we have used \eqref{rbepw204}. By combining \eqref{rbe-pw-rbepw703-2}--\eqref{rbe-pw-rbepw703-6}, we can obtain \eqref{rbe-pw-rbepw703}. The proof of the Lemma is completed.
\end{proof}


\begin{lem}\label{3dmvpbpw10}
For any given $\beta\geq0$, there exists a constant $C>0$ such that
$$
\|S^tg_0(x)\|_{L^{\infty}_{v,\beta}}\leq Ce^{-\frac{2\nu_0t}{3}}\max_{y\in\R^3}e^{-\frac{\nu_0|x-y|}{3}}\|g_0(y)\|_{L^{\infty}_{v,\beta}},
$$
where $S^t$ and $\nu_0$ are defined by \eqref{St} and \eqref{nuv} respectively. In particular, if $g_0(x,v)$ satisfies
$$
\|g_0(x)\|_{L^{\infty}_{v,\beta}}\leq Ce^{-\frac{|x|}{D_1}},
$$
then
$$
\|W_{k}(t)\ast g_0(x)\|_{L^{\infty}_{v,\beta}}\leq C(1+t)^ke^{-\frac{2\nu_0t}{3}}e^{-\frac{|x|}{2D_1}},
$$
where $D_1>0$ and $W_k,\ k\geq0$ is defined by \eqref{Wk}.
\end{lem}
\begin{proof}
The proof is same as Lemma
4.8 in \cite{LI-7} by using the fact that the Green's function of $(I-\Delta_{x})^{-1}$ is $\frac{1}{4\pi|x|}e^{-|x|}$, and  the detail of the proof is omitted for brevity.
\end{proof}


\begin{lem}\label{3dmvpbpw11}
For any given $\gamma\geq\frac{1}{2}$, $\alpha,\beta\geq0$ and $D_1,\lambda>0$, if the function $F(t,x,v)$ satisfies
$$
\|F(t,x)\|_{L^{\infty}_{v,\beta-1}}\leq U(t,x;D_1,\alpha,\gamma),
$$
then we have for any $\eta >1$ and $D_1\ge \frac{6\sqrt{\eta}\max\{1,\lambda\}^2}{\nu_0(\sqrt{\eta}-1)}$ that
\bma
&\quad\bigg\|\int^t_0S^{t-s}F(s,x)ds\bigg\|_{L^{\infty}_{v,\beta}}\leq CU(t,x;\eta D_1,\alpha,\gamma),\label{3dmvpbpw11-1}\\
&\quad\bigg\|\int^t_0W_{k}(t-s)\ast F(s,x)ds\bigg\|_{L^{\infty}_{v,\beta}}
\leq CU(t,x;\eta D_1,\alpha,\gamma),\label{3dmvpbpw11-2}
\ema
where $W_k,\ k\geq0$ is defined by \eqref{Wk} and
\bmas
U(t,x;D_1,\alpha,\gamma)&=(1+t)^{-\alpha}\(e^{-\frac{|x|^2}{D_1(1+t)}}+B_{\gamma}(t,|x|)1_{\{|x|\leq\lambda t\}}\)+e^{-\frac{|x|+t}{D_1}}\\
&\quad +(1+t)^{-\alpha-\frac12}\(e^{-\frac{(|x|-\lambda t)^2}{D_1(1+t)}}+B_{\gamma}(t,|x|-\lambda t)1_{\{|x|\leq\lambda t\}}\).
\emas
\end{lem}
\begin{proof}
The proof is same as Lemma 4.8 in \cite{LI-8} by using the fact that
$$
\int_{0}^t\|e^{-\frac{\nu(v)(t-s)}{3}}F(s,x-vs)\|_{L^{\infty}_{v,\beta}}ds\leq  \int_{0}^{t}e^{-\frac{2\nu(v)(t-s)}{3}}\nu(v)e^{-\frac{|x-y|}3}\|F(s,y)\|_{L^{\infty}_{v,\beta-1}}ds,
$$
and  the detail of the proof is omitted for brevity.
\end{proof}

\medskip
\noindent {\bf Conflict of interest:} The authors declared that they have no conflicts of interest to this work.

\medskip
\noindent {\bf Data availability:} No data was used for the research described in the article.

\medskip
\noindent {\bf Acknowledgements:} The research of this work was supported  by  the special foundation for Guangxi Ba Gui Scholars, and the National Natural Science Foundation of China  grants (No.  12171104).




\begin{thebibliography}{99}
\setlength{\itemsep}{-4pt}
\renewcommand{\baselinestretch}{1}
\small

\bibitem{CERCIGNANI-1}C. Cercignani, R. Illner and M. Pulvirenti,  {\it The mathematical theory of dilute gases}, Applied Mathematical Sciences, vol. 106, Springer, New York, 1994.

\bibitem{DENG-1}D.-Q. Deng, Regularity of the Vlasov-Poisson-Boltzmann system without angular cutoff, {\it Comm. Math. Phys.}, \textbf{387} (2021), 1603-1654.

\bibitem{DENG-2}D.-Q. Deng, Global regularity of the Vlasov-Poisson-Boltzmann system near Maxwellian without angular cutoff for soft potential. {\it Commun. Math. Anal. Appl.}, \textbf{2} (2023), no. 4, 421-468.


\bibitem{DUAN-1}R.-J. Duan and T. Yang, Stability of the one-species Vlasov-Poisson-Boltzmann system, {\it SIAM J. Math. Anal.}, \textbf{41} (2010), 2353-2387.

\bibitem{DUAN-2}R.-J. Duan and R.-M. Strain, Optimal time decay of the Vlasov-Poisson-Boltzmann system in $\R^{3}$, {\it Arch. Ration. Mech. Anal.}, \textbf{199} (2011), 291-328.

\bibitem{DUAN-3}R.-J. Duan, T. Yang and H.-J. Zhao, The Vlasov-Poisson-Boltzmann system in the whole space: The hard potential case, {\it J. Differ. Equ.}, \textbf{252} (2012), 6356-6386.

\bibitem{DUAN-4}R.-J. Duan, T. Yang and H.-J. Zhao, The Vlasov-Poisson-Boltzmann system for soft potentials. {\it Math. Models Appl. Sci.}, \textbf{23} (2013), 979-1028.

\bibitem{DUAN-5} R.-J. Duan and S. Liu, Stability of the rarefaction wave of the Vlasov-Poisson-Boltzmann system, {\it SIAM J. Math. Anal.}, \textbf{47} (2015), 3585-3647.



\bibitem{Guo-1} Y. Guo and J. Jang, Global Hilbert expansion for the Vlasov-Poisson-Boltzmann system, {\it Comm. Math. Phys.,}, \textbf{299} (2010), 469-501.

\bibitem{GUO-2}Y. Guo, The Vlasov-Poisson-Boltzmann system near vacuum, {\it Comm. Math. Phys.}, \textbf{218} (2001), 293-313.

\bibitem{GUO-3}Y. Guo, The Vlasov-Poisson-Boltzmann system near Maxwellians, {\it Comm. Pure Appl. Math.}, \textbf{55} (2002), 1104-1135.

\bibitem{Gong-1} W.-H. Gong, F.-J. Zhou and W.-J. Wu, Global strong solution and incompressible Navier-Stokes-Fourier-Poisson limit of the Vlasov-Poisson-Boltzmann system, {\it SIAM J. Math. Anal.}, \textbf{53} (2021), 6424-6470.

\bibitem{Jiang-1} N. Jiang and X. Zhang, Uncertainty qualification of Vlasov-Poisson-Boltzmann equations in the diffusive scaling, {\it J. Funct. Anal.}, \textbf{288} (2025), Paper No. 110794.

\bibitem{Li-9} H.-L. Li, T. Yang and M.-Y. Zhong, Diffusion limit of the Vlasov-Poisson-Boltzmann system, {\it Kinet. Relat. Models,}, \textbf{14} (2021), 211-255.



\bibitem{LI-1}H.-L. Li, T. Yang and M.-Y. Zhong, Spectrum analysis for the Vlasov-Poisson-Boltzmann system, {\it Arch. Rational
  Mech. Anal.}, \textbf{241} (2021), 311-355.

\bibitem{LI-2}H.-L. Li, T. Yang and M.-Y. Zhong, Spectrum analysis and optimal decay rates of the bipolar Vlasov-Poisson-Boltzmann equations, {\it Indiana Univ. Math. J.}, \textbf{65} (2016), 665-725.

\bibitem{LI-3}H.-L. Li, Y. Wang, T. Yang and M.-Y. Zhong,  Stability of nonlinear wave patterns to the bipolar Vlasov-Poisson-Boltzmann system, {\it Arch. Rational Mech. Anal.}, \textbf{228} (2018), 39-127.

\bibitem{LI-4}H.-L. Li, T. Yang and Y. Wang, Stability of the superposition of a viscous contact wave with two rarefaction waves to the bipolar Vlasov-Poisson-Boltzmann system, {\it SIAM J. Math. Anal.}, \textbf{50} (2018), 1829-1876.

\bibitem{LI-5}H.-L. Li, T. Yang and M.-Y. Zhong, Green's function and pointwise Space-time behaviors of the Vlasov-Poisson-Boltzmann system, {\it Arch. Rational Mech. Anal.}, \textbf{235} (2020), 1011-1057.


\bibitem{LI-7}Y.-C. Li and M.-Y. Zhong, Green's function and pointwise Behavior of the one-dimensional Vlasov-Poisson-Boltzmann system, {\it Kinetic and Related Models}, \textbf{16} (2023), 676-716.

\bibitem{LI-8}Y.-C. Li and M.-Y. Zhong, Green’s function and pointwise space-time behaviors of the three-dimensional relativistic Boltzmann equation, {\it J. Differ. Equ.}, \textbf{412} (2024), 140-213.

\bibitem{Lin-1}Y.-C. Lin, H.-T. Wang and K.-C. Wu,  3D hard sphere Boltzmann equation: explicit structure and the transition process from polynomial tail to Gaussian tail, arXiv:2408.02183.




\bibitem{LIU-2}T.-P. Liu and S.-H. Yu, The Green's function and large-time behavior of solutions for the one-dimensional Boltzmann equation, {\it Comm. Pure Appl. Math.}, \textbf{57} (2004), 1543-1608.

\bibitem{LIU-3}T.-P. Liu and S.-H. Yu, The Green's function of Boltzmann equation, 3D waves, {\it Bull. Inst. Math. Acad. Sin. (N.S.)}, \textbf{1} (2006), 1-78.


\bibitem{LIU-5}T.-P. Liu and S.-H. Yu, Soving Boltzmann equation, part I: Green's function, {\it Bull. Inst. Math. Acad. Sin.}, \textbf{6} (2011), 115-243.


\bibitem{MISCHLER-1} S. Mischler, On the initial boundary value problem for the Vlasov-Poisson-Boltzmann System, {\it Math. Phys.}, \textbf{210} (2000), 447-466.

\bibitem{Tong-1} L.-L Tong, Z. Tan and X. Zhang, The diffusive limit of the bipolar Vlasov-Poisson-Boltzmann equations, {\it J. Statist. Phys.,}, \textbf{188} (2022), Paper No. 2.


\bibitem{Wang-1} Y.-J. Wang, The diffusive limit of the Vlasov-Boltzmann system for binary fluids, {\it SIAM J. Math. Anal.}, \textbf{43} (2011), 253-301.

\bibitem{Wu-1} W.-J. Wu, F.-J. Zhou and Y.-S. Li, Incompressible Euler-Poisson limit of the Vlasov-Poisson-Boltzmann system, {\it J. Math. Phys.}, \textbf{63} (2022), Paper No. 081502.




\bibitem{YANG-3}T. Yang and H.-J. Yu, Optimal convergence rates of classical solutions for Vlasov-Poisson-Boltzmann system, {\it Comm. Math. Phys.}, \textbf{301} (2011), 319-355.

\bibitem{YANG-2} T. Yang and H.-J. Zhao, Global existence of classical solutions to the Vlasov-Poisson-Boltzmann system, {\it Comm. Math. Phys.}, \textbf{268} (2006), 569-605.

\end{thebibliography}
\end{document}